\documentclass[preprint,10.5pt,onside,a4paper]{article}

\title{Distribution rules of crystallographic systematic absences on the Conway topograph and their application to powder auto-indexing\thanks{This research was partly supported by a Grant-in-Aid for Young Scientists (B) (No. 22740077) and the Ibaraki prefecture (J-PARC-23D06).}} 

\author{Ryoko Oishi-Tomiyasu\thanks{High Energy Accelerator Research Organization, Tsukuba, Japan. 
(ryoko.tomiyasu@kek.jp)}}

\usepackage{amssymb}
\usepackage{amsmath}
\usepackage{amscd}
\usepackage{amsthm}
\usepackage{enumerate}
\usepackage{multirow}
\usepackage[dvips]{graphics}

\newtheorem{theorem}{Theorem}
\newtheorem{proposition}{Proposition}[section]
\newtheorem{lemma}{Lemma}[section]
\newtheorem{corollary}{Corollary}[section]
\newtheorem{definition}{Definition}[section]

\newtheorem{fact}{Fact}
\newtheorem{example}{Example}
\newtheorem{remark}{Remark}

\newcommand{\IE}{\textit{i.e., }}
\newcommand{\CF}{\textit{cf.\ }}	
\newcommand{\EG}{\textit{e.g., }}

\newcommand{\IntegerRing}{{\mathbb Z}}

\newcommand{\RationalField}{{\mathbb Q}}
\newcommand{\RealField}{{\mathbb R}}
\newcommand{\ComplexField}{{\mathbb C}}
\newcommand{\tr}[1]{\hspace{0.5mm}{}^t\hspace{-0.5mm}#1}
\newcommand{\abs}[1]{ {\left\lvert #1 \right\rvert} }

\begin{document}
\maketitle

\begin{abstract}
Powder auto-indexing is the crystallographic problem of lattice determination from an average theta series.
There, in addition to all the multiplicities, the lengths of part of lattice vectors  
cannot be obtained owing to systematic absences.
As a consequence,  solutions are not always unique.
We develop a new algorithm to enumerate powder auto-indexing solutions.
This is a novel application of the reduction theory of positive-definite quadratic forms to a problem of crystallography.
Our algorithm is proved to be effective for all types of systematic absences,
using their newly obtained common properties. 
The properties are stated as distribution rules for 
lattice vectors corresponding to systematic absences on a topograph.
Conway defined topographs for 2-dimensional lattices
as graphs whose edges are associated with $\abs{l_1}^2$, $\abs{l_2}^2$, $\abs{l_1+l_2}^2$, $\abs{l_1-l_2}^2$,
where $l_1, l_2$ are lattice vectors.
In our enumeration algorithm, topographs are utilized as a network of lattice vector lengths.  
As a crystal structure is a lattice of rank 3,
the definition of topographs is generalized to any higher dimensional lattices using Voronoi's second reduction theory.
The use of topographs allows us to speed up the algorithm.
The computation time is reduced to 1/250--1/32, when it is applied to real powder diffraction patterns.
Another advantage of 
our algorithm is its robustness to missing or false elements in the set of lengths extracted from a powder diffraction pattern.
Conograph is the powder indexing software which implements the algorithm.
We present results of Conograph for 30 diffraction patterns, including some very difficult cases.
\end{abstract}



\section{Introduction}

Lattice determination problems have been the target of much interest in mathematics.
In particular, many mathematicians worked on lattice determination from a theta series in the second half of the 20th century,
with the aim of providing higher dimensional counterexamples for the isospectral problem proposed by Kac \cite{Kac66}.  
Lattice determination from a theta series was finally resolved in \cite{Schiemann90}, \cite{Schiemann97}.
However, it is not well known in the mathematics community that a similar problem was being studied in crystallography during this period.

The lattice determination problem in crystallography is called powder auto-indexing.
There a set of lattice parameters is determined from a powder diffraction pattern. Powder auto-indexing is equivalent to lattice determination from an average theta series, as shown in Appendix \ref{On equivalence between powder diffraction pattern and average theta series}.
At present, powder diffraction is the principal technique for determining the structure of materials that do not necessarily form large-sized crystals.  A highly automated system of analyzing powder diffraction patterns
is required, both for basic scientific research and a wide range of industrial applications, not least in the pharmaceutical industry.
Hence, it is necessary to establish a powerful and reliable powder auto-indexing algorithm and software.

In this paper, the powder auto-indexing problem is formulated and solved.
This is regarded as a new application of the reduction theory of positive-definite quadratic forms to crystallography.
In consideration of its application to crystallography, we mainly discuss cases when the lattice has a rank of $N = 2, 3$.
Similar to a theta series,
information about the lengths of lattice vectors is extracted from an average theta series (Section \ref{Powder indexing problem}).
The most notable difference is that 
it is very difficult to acquire the multiplicities of the lengths (\IE number of lattice vectors satisfying $\abs{l}^2 = q$ for fixed $q \in \RealField_{> 0}$) from an average theta series,
and only a part of the lengths can be obtained owing to the crystallographic phenomenon of systematic absences.

Systematic absences have not been defined sufficiently explicitly as a mathematical notion.
We provide a definition in Section \ref{Crystallographic group and extinction rules}.
As in the International Tables \cite{Hahn83},
systematic absences are classified by the pair $(G, H)$: 
\begin{itemize}
	\item isomorphism class of the crystallographic group $G \subset O(N) \ltimes \RealField^N$,
	\item conjugate class of the finite subgroup $H \subset G$.
\end{itemize}

For fixed $N$, systematic absences are categorized by
finitely many pairs (Proposition \ref{prop:basic property of extinction rules1}).
When a periodic function $\wp$ belongs to the type $(G, H)$,
the lengths of $l^* \in \Gamma_{ext}(G, H)$ is not directly extracted from the average theta series of $\wp$,
where $\Gamma_{ext}(G, H)$ is a subset of the reciprocal lattice $L^*$ of $L$ determined from $(G, H)$.

To date, powder auto-indexing algorithms have been studied and improved by experts in crystallography.
Among existing software packages, Ito \cite{Wolff58}, \cite{Visser69}, TREOR (trial-and-error method \cite{Werner85}), and DICVOL (dichotomy method \cite{Boultif2004}) are widely used.
McMaille (a grid search based on the Monte Carlo method \cite{LeBail2002}) and X-cell (dichotomy method \cite{Neumann2003}) were developed comparatively recently.
With regard to existing powder auto-indexing software, several considerable problems have been reported.
These include: 
\begin{itemize}
	\item existence of more than one solutions, 
	\item systematic absences, 
	\item missing or false elements in the set of extracted lengths,
	\item observation errors contained in the length values, 
	\item large zero-point shift, 
	\item overlapping peaks.
\end{itemize}
In some software, these problems are caused by limitations intended to suppress the computation time.
Therefore, both computation time and success rate should be considered to discuss powder auto-indexing algorithms.
(For example, although the simplest brute force grid search may obtain the highest success rate,
it takes from a few hours to a few days in monoclinic and triclinic cases.)
This paper presents a method that can resolve each of the problems listed above.
The first three are explained in more detail in the following discussion.
The next two are resolved using the estimated error ${\rm Err}[\abs{l}^2]$ of $\abs{l}^2$ (see ($A$\ref{item:Tolerance level}) in Section \ref{Precise formulation of powder auto-indexing problem and our strategy} and the beginning of Section \ref{Powder indexing algorithm}).
The final problem is also resolved naturally in our algorithm (Section \ref{Problems on quality of powder diffraction patterns}).

Our three main results are as follows. 
First, we formulate powder auto-indexing as a mathematical problem,
and summarize mathematical results on the cardinality of solutions in powder auto-indexing.
This is considered the most important foundation for the following discussion;
in general, a lattice $L$ is not determined uniquely in $N \geq 2$,
even if the rank $N$ of $L$ and all the lengths $l \in L$ are given
(Appendix \ref{Summary of known theorems on multiple solutions}).
On the other hand, the cardinality of $L_2$ satisfying $\Lambda_L = \Lambda_{L_2}$ is always finite in $N \leq 4$,
therefore it is possible to enumerate all such $L_2$ algorithmically (Appendix \ref{Appendix:lattice determination algorithm from complete set of the lengths of lattice vectors}).  
However, it is not certain whether the number of solutions is actually finite in powder auto-indexing,
owing to systematic absences and the finite observational range.
Fortunately, using both physical constraints and mathematical theorems,
we obtain a necessary condition for assuming the finiteness of solutions, except for events with zero probability (Section \ref{Precise formulation of powder auto-indexing problem and our strategy}).
As a result, it is only necessary to enumerate finitely many solutions for powder auto-indexing.

Our second result is a new algorithm to enumerate powder auto-indexing solutions
(Sections \ref{Lattice determination from weighted theta series in N = 2} and \ref{Lattice determination from weighted theta series in N = 3}).
We prove that it works regardless of the type of systematic absences. 
Our basic idea is to use the graph whose edges are associated with a set of squares of lengths $\abs{l_1^*}^2$, $\abs{l_2^*}^2$, $\abs{l_1^*+l_2^*}$, $\abs{l_1^*-l_2^*}^2$ for some lattice vectors $l_1^*, l_2^* \in L^*$.

In Ito's method,
the parallelogram law $2( \abs{l_1^*}^2 + \abs{l_2^*}^2 ) = \abs{l_1^*+l_2^*}^2 + \abs{l_1^*-l_2^*}^2$
is used to obtain Gram matrices of sublattices of rank 2 (called \textit{zones}).
However,
simple use of the parallelogram law does not always provide successful results,
as introduced in Fact \ref{fact:no Ito's equation} of Section \ref{Results in N = 3}.
Therefore, instead of just enumerating a set of lengths satisfying the parallelogram law, our algorithm utilizes graphs whose edges are associated with $\abs{l_1^*}^2$, $\abs{l_2^*}^2$, $\abs{l_1^*+l_2^*}$, and $\abs{l_1^*-l_2^*}^2$ as a network of lattice vector lengths.
This provides a comprehensive method of analyzing how 
the two sets $\{ \abs{l_1^*}^2, \abs{l_2^*}^2, \abs{l_1^*+l_2^*}^2, \abs{l_1^*-l_2^*}^2 \}$ and $\{ \abs{k_1^*}^2, \abs{k_2^*}^2, \abs{k_1^*+k_2^*}^2, \abs{k_1^*-k_2^*}^2 \}$ are related to each other.

For lattices of rank 2, such a graph has already been defined by Conway, who called it a topograph \cite{Conway97}.  
As explained in section \ref{Definition for lattices of general rank},
a graph having the required property is constructed for general $N$ by using Voronoi's second reduction theory.
We also call this a topograph.
Basic properties of topographs for lattices of rank $N = 2, 3$ are explained in Section \ref{Topographs for low-dimensional lattices}.

In Section \ref{Main theorems}, we shall see how primitive vectors of $l^* \in L^*$ belonging to $\Gamma_{ext}(G, H)$ 
are distributed on a topograph.
Although the following theorem for the case of $N = 2$ has not been described explicitly,
it provides a theoretical reason for the parallelogram law working appropriately for $N = 2$.

\begin{theorem}\label{fact:two-dimensional extinction rules}
Let $(G, H)$ be a type of systematic absence in $N = 2$,
and let $L^*$ be the reciprocal lattice of $L$, a lattice consisting of all the translations in $G$.
Then, for any primitive vector $l^*$ of $L^*$ belonging to $\Gamma_{ext}(G, H)$, 
there exists $l_2^* \in L^*$ such that $l_1^* \cdot l_2^* = 0$ and 
$l_1^*, l_2^*$ make a basis of $L^*$.
\end{theorem}

In brief, Theorem \ref{fact:two-dimensional extinction rules} claims
that $l^* \in \Gamma_{ext}(G, H)$ only appear in the gray area in Figure \ref{Lattice vectors contained in some Lambda_{ext}(G, H, x) (Case of N = 2)}, regardless of the type $(G, H)$.
With regard to the case of $H = \{ 1 \}$, 
we shall introduce a proof using topographs in Section \ref{Results in N = 2}.
The case $H \supsetneq L$ is not proved here
since it is verified easily by checking lists in \cite{Hahn83}.

\begin{figure}
\begin{center}
\scalebox{0.8}{\includegraphics{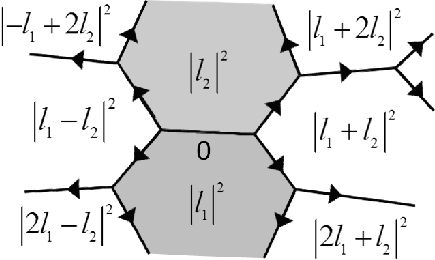}}
\end{center}
\caption{Reciprocal lattice vectors that are allowed to correspond to systematic absences.}
\label{Lattice vectors contained in some Lambda_{ext}(G, H, x) (Case of N = 2)}
\end{figure}

By Theorem \ref{fact:two-dimensional extinction rules},
a subgraph of a topograph with infinitely many edges 
is formed by unifying substructures as in Figure \ref{Substructure of a topograph corresponding to the parallologram law}
associated with the lengths $\abs{l_1^*}^2, \abs{l_2^*}^2, \abs{l_1^*+l_2^*}^2, \abs{l_1^*-l_2^*}^2$ of 
$l_1^*, l_2^* \in L^* \setminus \Gamma_{ext}(G, H)$ that are easily extracted from an average theta series.

\begin{figure}
\begin{center}
\scalebox{0.6}{\includegraphics{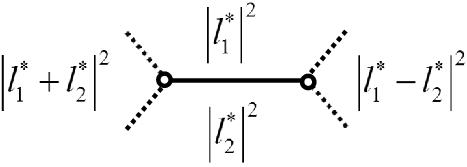}}
\end{center}
\caption{Substructure of a topograph corresponding to the parallelogram law.}
\label{Substructure of a topograph corresponding to the parallologram law}
\end{figure}

For $N = 3$, similar statements are proved in Theorems \ref{thm:distribution rules on CT_{2, S_2}} and \ref{thm:distribution rules on CT_{3, S_3}}.
In this case, it is necessary to use the following formula instead of the parallelogram law:
\begin{eqnarray}
	3 |l_1^*|^2 + |l_1^* + 2 l_2^*|^2 = 3 |l_2^*|^2 + |2 l_1^* + l_2^*|^2.
\end{eqnarray}

As another consequence of Theorems \ref{thm:distribution rules on CT_{2, S_2}} and \ref{thm:distribution rules on CT_{3, S_3}},
our algorithm enumerates
the Gram matrices $(l_i^* \cdot l_j^*)_{1 \leq i, j \leq 3}$
for multiple bases $l_1^*, l_2^*, l_3^*$ of the true solution $L^*$.
This makes
the enumeration procedure robust against missing or false elements in the set of extracted lengths,
as explained in Section \ref{Problems on quality of powder diffraction patterns}.

For the third result, we provide a method to speed up the enumeration process in Section \ref{Speed-up using topographs}.
In order not to reduce the rate to acquire the true solution,
Theorems \ref{thm:distribution rules on CT_{2, S_2}} and \ref{thm:distribution rules on CT_{3, S_3}}
are used again here.
In the test using actual powder diffraction patterns in Section \ref{Results}, 
it is demonstrated that the improvement makes the enumeration speed 32--250 times faster.
As a result, the enumeration process is executed in a few minutes at most.

The novel algorithm for lattices of rank $N=3$ is implemented in the powder auto-indexing software Conograph.
In section \ref{Computational results of Conograph}, we introduce the default parameters and results of Conograph.
The default parameters are selected so that less-experienced users of the software can obtain good results without modifying them. 
We prepare 30 sets of test data, including difficult cases such as samples from Structure Determination by Powder Diffractometry Round Robin-2 (SDPDRR-2).
Results for rather difficult cases are explained in Examples \ref{exam:Non-unique solutions}--\ref{exam:bad quality}.
The total time for powder auto-indexing did not exceed several minutes.
The Conograph software is scheduled to be distributed in the near future from http://sourceforge.jp/projects/conograph/.

\section*{Notation and symbols}
\label{Notations and symbols}

The notation and symbols used in this paper are summarized in this section.
The inner product of the Euclidean space $\RealField^N$ is denoted by $u \cdot v$,
and the Euclidean norm $u \cdot u$ is denoted by $|u|^2$.  
The standard basis $\tr{(0, \ldots, 0, \overset{i}{\check{1}}, 0, \ldots, 0)}$ 
of $\IntegerRing^N$ is denoted by $\mathbf{e}_i$ ($1 \leq i \leq N)$.

A \textit{lattice} $L$ of rank $N$ is 
a discrete and cocompact subgroup of $\RealField^N$.
For any lattice $L$, there are linearly independent vectors $v_1, \ldots, v_N \in \RealField^N$ over $\RealField$ such that 
$v_1, \ldots, v_N$ generate $L$ as a $\IntegerRing$-module.
In this case, $v_1, \ldots, v_N$ are called a \textit{basis} of $L$,
and the matrix $(v_i \cdot v_j)_{1 \leq i, j \leq N}$ is called a \textit{Gram matrix} of $L$.
The \textit{reciprocal lattice} $L^*$ of $L$ is defined as $L^* := \{ l^* \in \RealField^N : l \cdot l^* \in \IntegerRing \text{ for all } l \in L \}$.

If $\{ v_1, \ldots, v_i \} \subset L$ ($1 \leq i < N$) is extended to a basis of $L$, it is called \textit{a primitive set} of $L$.  
In particular, $v \in L$ is a \textit{primitive vector} of $L$ if and only if $\{ v \}$ is a primitive set of $L$.
$P_n(L)$ is the set consisting of all the primitive sets of $L$ of cardinality $n$.

A symmetric matrix $(s_{ij})_{1 \leq i, j \leq N}$
is always identified with a \textit{quadratic form} $Q(x_1, \ldots, x_N) := \sum_{i=1}^N \sum_{j=1}^N s_{ij}x_i x_j$.
The linear space consisting of $N \times N$ symmetric matrices with real entries is denoted by ${\mathcal S}^N$.
${\mathcal S}^N_{\succ 0}$ (resp. ${\mathcal S}^N_{\succeq 0}$) is the subset of ${\mathcal S}^N$ consisting of all the positive-definite (resp. semidefinite) symmetric matrices.
$S_1, S_2 \in {\mathcal S}^N$ are \textit{equivalent} 
if and only if there exists $g \in GL_N(\IntegerRing)$ such that $S_2 = g S_1 \tr{g}$.

For any $S \in {\mathcal S}^N$,
elements of $\Lambda_S$
are called \textit{representations} of $S$ over $\IntegerRing$.
\begin{eqnarray}
	\Lambda_S &:=& \{ \tr{v} S v : 0 \neq v \in \IntegerRing^N \}.
\end{eqnarray}

In crystallography, a crystal lattice in three-dimensional Euclidean space
is represented by a set of lattice parameters $a$, $b$, $c$, $\alpha$, $\beta$, and $\gamma$ as in Figure \ref{Unitcell and lattice parameters}.
\begin{figure}
\begin{center}
\scalebox{1.0}{\includegraphics{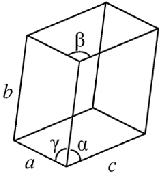}}
\end{center}
\caption{Lattice parameters and the unit cell.}
\label{Unitcell and lattice parameters}
\end{figure}
This parameterization is transformed into a $3 \times 3$ matrix $S := (s_{ij})_{1 \leq i, j \leq 3}$ as follows,
and $S$ is the Gram matrix of the \textit{Bravais lattice} of the crystal.
\begin{eqnarray}\label{eq:3-by-3 Gram matrix}
	s_{11} = a^2, & s_{22} = b^2, & s_{33} = c^2, \\
	s_{12} = a b\ cos \gamma,\ & s_{13} = a c\ cos \beta,\ & s_{23} = b c\ cos \alpha. \nonumber
\end{eqnarray}

In general, for any lattice $L \subset \RealField^N$, its automorphism group is defined by $\{ g \in GL(3, \IntegerRing) : g S \tr{g} = S \}$,
and $L$ is categorized into its Bravais type by the conjugacy class of the group in $GL(3, \IntegerRing)$. 
In $N = 3$, it is known that there exist 14 Bravais types.
In crystallography, selection of the Bravais lattice $L_2 \subset L$ and its basis $l_1, l_2, l_3$  
is standardized according to the Bravais type of $L$.  $S$ in (\ref{eq:3-by-3 Gram matrix}) is the Gram matrix defined for the basis of $L_2$.
When $L$ belongs the category called a primitive centring,
$L = L_2$, and $l_1, l_2, l_3$ are chosen by the method called the Niggli reduction, which is very similar with the Minkowski reduction in $N = 3$ (see \cite{Hahn83}).
When $L$ belongs to the category called a face-centered (resp. body-centered) lattice,
there exist $l_1, l_2, l_3 \in L$ such that
$L$ is generated by $\frac{l_i + l_j}{2}$ ($1 \leq i < j \leq 3$) (resp. $\frac{l_1 + l_2 + l_3}{2} - l_i$ ($1 \leq i\leq 3$)),
and $l_i \cdot l_j = 0$ ($1 \leq i < j \leq 3$) holds.
Regardless of the choice of $l_1, l_2, l_3$, their Gram matrix $S$ and $L_2$ generated by these $l_1, l_2, l_3$ are determined uniquely.
With regard to the Bravais lattice, the explanation above is sufficient in order to understand our following discussions.

In the context of powder structure analysis, representations of $S^{-1}$ over $\IntegerRing$ are called \textit{$q$-values} of diffraction peaks.

\section{Outline of the powder auto-indexing problem}
\label{Powder indexing problem}

A function $\wp$ on $\RealField^N$ is said to be \textit{periodic} 
if $L := \{ l \in \RealField^N : \wp(x+l) = \wp(x) \text{ for any } x \in \RealField^N \}$ is a lattice.
We call $L$ the \textit{period lattice} of $\wp$.
For example, an electron (or nucleus) density in a crystal (\CF the left figure in Figure \ref{Powder diffraction pattern}) 
is a periodic function of $N = 3$.

The following $\wp$ provides its standard model.
\begin{eqnarray}
	\wp(x) &=& \sum_{i = 1}^m \sum_{k=1}^{d_i} \sum_{l \in L} p_{i}(x - x_{ik} - l), \label{eq:distribution in crystal}
\end{eqnarray}
where
\begin{eqnarray*}
	m & : & \text{number of different elements in a crystal}, \\
	d_i & : & \text{number of the $i$th element in the primitive cell $\RealField^3 / L$}, \\
	p_{i}(x) & : & \text{rapidly decreasing function on $\RealField^3$ that represents the electron distribution of respective atoms.}
\end{eqnarray*}

According to diffraction theory, if a crystal has an electron density $\wp$,
the diffraction image of its single-crystal sample equals $c f_{single}(x^*; \wp)$ for some $c > 0$, where $f_{single}(x^*; \wp)$ is 
a sum of delta functions given by
\begin{eqnarray}
	f_{single}(x^*; \wp) &:=& \sum_{l^* \in L^*} |\hat{\wp}(l^*)|^2 \delta(x^* - l^*), \\
			 \hat{\wp}(l^*) &:=& \int_{\RealField^N / L} \wp(x) e^{-2 \pi \sqrt{-1} x \cdot l^*} dx.
\end{eqnarray}

A powder sample is an ensemble of a very large number of randomly oriented crystallites.
As a result, its diffraction image is proportional to the integration of $f_{single}(x^*; \wp)$ on a sphere of radius $\sqrt{q}$:
\begin{eqnarray}\label{eq:powder diffraction pattern}
	f_{powder}(q; \wp)
	&:=&
		\int_{|x^*|^2 = q} f_{single}(x^*; \wp) d x^*
	=
		2 \sqrt{q} \sum_{l^* \in L^*} |\hat{\wp}(l^*)|^2 \delta(q - |l^*|^2).
\end{eqnarray}
The $\hat{\wp}(l^*)$ is called a \textit{structure factor} in crystallography.

The right figure in Figure \ref{Powder diffraction pattern} presents an actual powder diffraction pattern.  It is obtained by replacing every delta function in (\ref{eq:powder diffraction pattern}) with some kind of peak-shape model function $g$ that is close to a Gaussian distribution and satisfies $\int_{\RealField} g(q) dq = 1$.

\begin{figure}[htbp]
\begin{minipage}{\textwidth}
\begin{center}
\scalebox{0.6}{\includegraphics{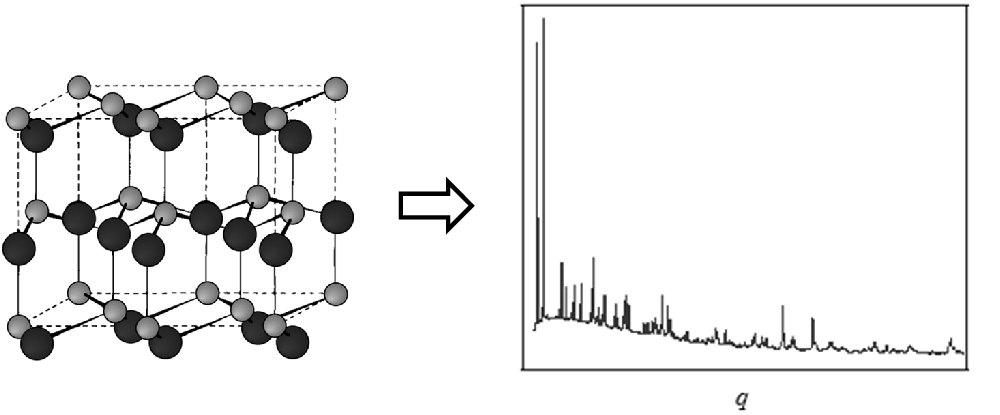}}
\end{center}
\end{minipage}
\caption{Powder diffraction pattern.}
\label{Powder diffraction pattern}
\end{figure}

Ab-initio powder crystal structure determination
retrieves the electron density $\wp$ from the powder diffraction pattern
under the assumption that the following additional information is available.  
	\begin{itemize}
		\item chemical formula (\IE the ratio $[d_1:\cdots:d_m]$),
		\item density of a single crystal,
		\item rapidly decreasing functions $p_{i}$ ($1 \leq i \leq m$).
	\end{itemize}

Powder auto-indexing is the initial stage of ab-initio powder crystal structure determination, and aims 
to find the period lattice $L$ of $\wp$.
As the positions of the delta functions,
elements of $\Lambda_\wp$ are extracted from a powder diffraction pattern $f_{powder}(q; \wp)$.
\begin{eqnarray}
	\Lambda_{\wp} &:=& \{ \abs{l^*}^2 : l^* \in L^*,\ F_\wp(\abs{l^*}^2) \neq 0  \}. \label{item:peak-positions}
\end{eqnarray}

$L$ is normally determined from elements of $\Lambda_{\wp}$ in powder auto-indexing.
After $L$ is obtained,
using the coefficients $F_\wp(q)$, powder crystal structure determination is carried out.
\begin{eqnarray}
	F_\wp(q) &:=& \displaystyle \sum_{l^* \in L^*, |l^*|^2=q} \abs{ \hat{ \wp}(l^*) }^2. \label{item:intensities} 
\end{eqnarray}

\section{Formulation of the powder auto-indexing problem}
\label{Precise formulation of powder auto-indexing problem and our strategy}

In this section, we formulate the problem explicitly.
In powder auto-indexing, extracted $\Lambda^{obs}$ is different from the set $\Lambda^{\wp}$ of true $\abs{l^*}^2$ ($l^* \in L^*$)
owing to observational problems.
Consequently, the Gram matrix $S$ of the period lattice $L$ of $\wp$ must be retrieved from $\Lambda^{obs}$
under the following assumptions.
\begin{enumerate}[(\text{$A$}1)]
	\item \label{item:observed range} 
		The observed range of a powder diffraction pattern is contained in a finite interval $[q_{min}, q_{max}] \subset (0, \infty)$.
		Consequently, only information about $\Lambda_\wp \cap [q_{min}, q_{max}]$ is available.

	\item \label{item:observation error} 
		Every $q^{obs} \in \Lambda^{obs}$ has some observation error.
		It may be assumed that the threshold ${\rm Err}[q^{obs}]$ on the error $\abs{ q^{obs} - q^{cal} }$ is given.
		(A method to compute ${\rm Err}[q^{obs}]$ is introduced in \cite{Tomiyasu2012}.)

	\item \label{item:error}
		Owing to probabilistic mistakes and errors in acquisition of the peak-positions (\IE peak-search),
		$\Lambda^{obs}$ differs from the true
		$\Lambda_\wp \cap [q_{min}, q_{max}]$.
		There exist small $\epsilon_1 > 0 $ and $\epsilon_2 > 0$ satisfying the following conditions: 
	\begin{enumerate}[(i)]
		\item 
			For arbitrarily fixed $q \in \Lambda_\wp \cap [q_{min}, q_{max}]$, the following occurs with probability $\epsilon_1$:
			\begin{eqnarray}
				q \notin \bigcup_{q^{obs} \in \Lambda^{obs}} [q^{obs} - {\rm Err}[q^{obs}], q^{obs} + {\rm Err}[q^{obs}]]. 
			\end{eqnarray}
		\item 
			For arbitrarily fixed $q^{obs} \in \Lambda^{obs}$, the following occurs with probability $\epsilon_2$:
			\begin{eqnarray}
				\Lambda_\wp \cap [q^{obs} - {\rm Err}[q^{obs}], q^{obs} + {\rm Err}[q^{obs}]] = \emptyset. 
			\end{eqnarray}
	\end{enumerate}

	\item \label{item:extinction rules}
	When the distribution $\wp$ is constrained by group symmetry,
	infinitely many $\hat{\wp}(l^*)$ and $F_\wp(q)$ become zero
	owing to \textit{systematic absences} deterministically (Section \ref{Crystallographic group and extinction rules}).
\end{enumerate}

As assumed in (\ref{item:error}),
$\Lambda_{obs}$ extracted from a powder diffraction pattern
has missing or false elements. These are caused by observational problems
including background noise or false peaks due to sample impurity. 

With regard to (\ref{item:extinction rules}), sometimes $F_\wp(q) = 0$ holds due to a special arrangement of atom positions $x_{ik}$, rather than systematic absences (\CF \cite{Engel84}, \cite{Templeton56}).
However, the probability is zero 
if every $x_{ik}$ is distributed uniformly in $\RealField^N / L$. 
(The only known exceptions are systematic absences.)

For the special arrangement with zero probability,
we replace $\Lambda_{\wp}$ by $\Lambda_{ext}(\wp)$ and consider ($\tilde{A}$\ref{item:error2}) instead of ($A$\ref{item:error}):
	\begin{eqnarray}\label{eq:definition of Lambda_{ext}(wp)}
		\Lambda_{ext}(\wp) := \Lambda_{\wp} \cup \left\{ \abs{l^*}^2 : 
								\begin{matrix}
								l^* \in L^*, F_\wp(q) = 0 \text{ owing to reasons} \\ \text{other than systematic absences}
								\end{matrix}
							\right\}. \hspace{10mm}
	\end{eqnarray}

\begin{enumerate}[($\tilde{A}$1)]
	\setcounter{enumi}{2}
	\item \label{item:error2}
	Assume that
	\begin{enumerate}[(i)]
		\item \label{item:q in Lambda_{ext}(wp) setminus Lambda^{obs}}
			For any $q \in \Lambda_{ext}(\wp) \cap [q_{min}, q_{max}]$,
			$q \notin \bigcup_{q^{obs} \in \Lambda^{obs}} [q^{obs} - {\rm Err}[q^{obs}], q^{obs} + {\rm Err}[q^{obs}]]$ 
			occurs with probability $\epsilon_1$.
		\item \label{item:q in Lambda^{obs} setminus Lambda_{ext}(wp)}
			For any $q^{obs} \in \Lambda^{obs}$, $\Lambda_{ext}(\wp) \cap [q^{obs} - {\rm Err}[q^{obs}], q^{obs} + {\rm Err}[q^{obs}]] = \emptyset$ occurs with probability $\epsilon_2$. 
	\end{enumerate}

\end{enumerate}


Under the conditions ($A$\ref{item:observed range})--($A$\ref{item:extinction rules}),
infinitely many solutions may exist in some cases 
(for example, consider the case of very small $[q_{min}, q_{max}]$).
Hence, additional assumptions are necessary in order to guarantee a finite number of solutions.
\begin{enumerate}[(\text{$A$}1)]
	\setcounter{enumi}{4}
	\item \label{item:size of cells}
		Owing to repulsive force between atoms,
		we may assume $\min \{ |l|^2 : 0 \neq l \in L \} \geq d^2$
		for a positive constant $d \approx 2 \AA$.
		Then, by the inequalities on successive minima of $L$ and its reciprocal lattice $L^*$ proved by Lagarias \textit{et al.\ }\cite{Lagarias90}, 
		in $N = 2, 3$, the maximum diagonal entry $D_N$ of a Minkowski-reduced(defined in Appendix \ref{Appendix:lattice determination algorithm from complete set of the lengths of lattice vectors}) Gram matrix $S$ of $L^*$ satisfies
		\begin{eqnarray}\label{eq:formula by Lagarias}
			D_N \leq \frac{N+3}{4 d^2} \max \{ \gamma_i : 1 \leq i \leq N \},
		\end{eqnarray}
		where $\gamma_i$ is the Hermite constant:
		\begin{eqnarray}
			\gamma_i := \sup \left\{ \min \{ \tr{v} S v : 0 \neq v \in \IntegerRing^N \}: S \in {\mathcal S}^i_{\succ 0},\ {\rm det} S = 1 \right\}.
		\end{eqnarray}
		In particular, $D_2 \leq \frac{5}{3} d^{-2}$ and $D_3 \leq 3 \cdot 2^{-1/3} d^{-2}$ follow from $\gamma_1 = 1$, $\gamma_2^2 = \frac{4}{3}$, $\gamma_3^3 = 2$.  (In $N = 2$, the estimation is improved up to $D_2 \leq \frac{4}{3} d^{-2}$ easily.)
		
	\item \label{item:q_{max} - q_{min}} The length of the interval $\sqrt{q_{max}} - \sqrt{q_{min}}$ is sufficiently greater than $\sqrt{D_N}$
			that there exists $l_1^*$, $l_2^*$, $l_3^*$ is a basis of $L^*$ such that $\Lambda^{obs}$ includes 
			$\abs{l_1^* \pm l_2^*}^2$, $\abs{l_1^* \pm l_3^*}^2$, $\abs{\pm l_1^*+l_2^*+l_3^*}^2$,
			and at least one of the following for both $i= 2, 3$:  
			\begin{enumerate}[(i)]
				\item \label{item:abs{l_1^*}^2 and abs{l_i^*}^2} $\abs{l_1^*}^2$, $\abs{l_i^*}^2$,
				\item $\abs{l_1^*}^2, \abs{l_1^* - 2 l_i^*}^2$,
				\item $\abs{l_1^*}^2, \abs{l_1^* + 2 l_i^*}^2$,
				\item $\abs{l_i^*}^2, \abs{2l_1^* - l_i^*}^2$.
				\item \label{item:abs{l_i^*}^2 and abs{2l_1^* + l_i^*}^2} $\abs{l_i^*}^2, \abs{2l_1^* + l_i^*}^2$.
			\end{enumerate}
\end{enumerate}

In ($A$\ref{item:size of cells}), 
$2 \AA$ is selected as the minimum distance between the two closest latttice points of $L$ satisfied by any existing crystals.
($A$\ref{item:q_{max} - q_{min}})
looks rather artificial.
This assumption is necessary for our algorithm in Table \ref{Enumeration of the two dimensional lattices enumTwoDimLattices}.
By Theorems \ref{thm:distribution rules on CT_{2, S_2}} and \ref{thm:distribution rules on CT_{3, S_3}}, 
($A$\ref{item:q_{max} - q_{min}}) holds except for events with zero probability
if a sufficiently large $q_{max}$ is chosen, regardless of the type of systematic absences.
Under assumptions ($A$\ref{item:size of cells}) and ($A$\ref{item:q_{max} - q_{min}}),
the number of solutions is always finite, 
because $\Lambda^{obs} \subset [q_{min}, q_{max}]$
contains only finite elements 
and the Gram matrix of $L^*$ is computed from a combination of elements of 
$\Lambda^{obs}$ due to assumption ($A$\ref{item:q_{max} - q_{min}}).

Here, it is still non-trivial how to select $q_{max}$.
At least, it is clear $q_{max} \geq D_3$ is required;
otherwise $\Lambda_{ext}(\wp) \cap [q_{min}, q_{max}]$ contain only $\abs{l^*}^2$ of $l^* \in L_2^*$ in some cases,
where $L_2^* \subset L^*$ is a sublattice of rank less than 3.
On the other hand, as seen in (\ref{item:Number of used q-values}) in Section \ref{Computational results of Conograph},
too many $q$-values are frequently extracted from 
the interval $[q_{min}, D_3]$ when we set $d = 2 \AA$ and $D_3 := 3 \cdot 2^{-1/3} d^{-2}$, 
nevertheless powder auto-indexing is very frequently successful even with a smaller interval.
Since the time of our enumeration algorithm is roughly proportional to the fourth power of the number of elements of $\Lambda^{obs}$ (\CF Section \ref{Speed-up using topographs}),
the time can be significantly decreased by minimizing the range $[q_{min}, q_{max}]$.
Considering the current accuracy of diffractometers and the power of personal computers,
$q_{max}$ should be chosen empirically to some degree, in addition to the theoretical estimation above.
See (\ref{item:Number of used q-values}) in Section \ref{Parameters for calculation conditions} for a more detailed approach 
to this issue.





\section{Summary of crystallographic groups and systematic absences}
\label{Crystallographic group and extinction rules}

We now give some definitions for crystallographic groups
and systematic absences.
In the following, 
we represent the elements $x \in \RealField^N / L$ as a row vector and 
 $l^* \in L^*$ as a column vector.
Furthermore, any group action on $\RealField^N / L$ (resp. $L^*$) is represented as a right action (resp. left action).

Any congruent transformation of the Euclidean space $\RealField^N$ is represented as a composition of the orthogonal group $O(N)$ and 
a translation; if $\sigma$ is a congruent transformation of $\RealField^N$, 
there exist $\tau \in O(N)$ and $\nu \in \RealField^N$ such that:
\begin{eqnarray}
	x^\sigma = x^\tau + \nu \text{ for any } x \in \RealField^N.
\end{eqnarray}
Such a $\sigma$ is denoted by $\{ \tau | \nu \}$.
The group consisting of all congruent transformations of $\RealField^N$
is the semidirect group $O(N) \ltimes \RealField^N$.  
By expanding the composition $(x^{ \{ \tau_1 | \nu_1 \} })^{ \{ \tau_2 | \nu_2 \} }$, it is seen that the group multiplication is given by:
\begin{eqnarray}
	\{ \tau_1 | \nu_1 \} \cdot \{ \tau_2 | \nu_2 \} = \{ \tau_1 \tau_2 | \nu_1^{\tau_2} + \nu_2 \}.
\end{eqnarray}

\begin{definition}
A \textit{crystallographic group} is a discrete and cocompact subgroup of $O(N) \ltimes \RealField^N$.
\end{definition}
A crystallographic group is also called a \textit{wallpaper group} in $N = 2$, and a \textit{space group} in $N = 3$.

For a crystallographic group $G$, 
two groups $R_G$ and $L$ are defined by: 
\begin{eqnarray}
		R_G &:=& \{ \tau \in O(N) : \{ \tau | \nu \} \text{ for some } \nu \in \RealField^N \}, \\
		L &:=& \{ \nu \in \RealField^N : \{ 1_N | \nu \} \in G \},
\end{eqnarray}
where $1_N$ is the identity of $O(N)$.
$L$ is a lattice and $R_G$ is a finite subgroup of $O(N)$ consisting of $\tau$ 
that maps any elements of $L$ to $L$.
$R_G$ is called a \textit{point group} of $G$.
$G$ is a group extension of $L$ by $R_G$.
From the definition of $L$, 
for any $\{ \tau | \nu_\tau \} \in G$,
the class $\nu_\tau + L \in \RealField^N / L$ is uniquely determined.
Furthermore, the map $R_G \longrightarrow \RealField^N / L$ : $\tau \mapsto \nu_\tau + L$ is a $1$-cocycle, \IE it satisfies
\begin{eqnarray}
	\nu_{\sigma \tau} \equiv \nu_{\sigma}^\tau + \nu_{\tau} \text{ mod } L.
\end{eqnarray}

We now proceed to the definition of systematic absences.
Let $G$ be a crystallographic group with the point group $R_G$ and the translation group $L$,
and consider the following periodic function $\wp$:
\begin{eqnarray}\label{eq:representation of wp}
		\wp(x) = \sum_{i=1}^m \sum_{\sigma \in G} p_{i}( x - x_{i}^\sigma ).
\end{eqnarray}
Note that the density model function $\wp$ in (\ref{eq:distribution in crystal}),
is represented as in (\ref{eq:representation of wp}) for some space group $G$.

Let $L^2(\RealField^N)$ be the $L^2$-space, \IE the set of all measurable functions $f$ on $\RealField^N$
with a finite $L^2$-norm $|| f ||_2 := (\int_{\RealField^N} \abs{f(x)}^2 dx)^{1/2} < \infty$. 
In order to compute the same list as \cite{Hahn83}, let us assume that

\vspace{3mm}
	{\bf (Isotropy condition)} $p_{i}$ belongs to $L^2(\RealField^N)^{R_G}$,
		\IE $p_{i}(x^\tau) = p_{i}(x)$ holds for any $\tau \in R_G$.
\vspace{3mm}

Then, $\wp(x^\sigma) = \wp(x)$ holds for any $\sigma \in G$,
and $\wp$ has the Fourier coefficient 
\begin{eqnarray}\label{eq:formula of hat{wp}(l^*)}
	\hat{\wp}(l^*) &=& \sum_{i=1}^m \sum_{\sigma \in G} \int_{\RealField^N / L} p_{i}(x - x_{i}^\sigma) e^{-2 \pi \sqrt{-1} x \cdot l^*} d x \nonumber \\
	&=& \sum_{i=1}^m \hat{p}_{i}(l^*) \sum_{\sigma \in G / L} e^{-2 \pi \sqrt{-1} x_{i}^\sigma \cdot l^*},
\end{eqnarray}
where $\hat{p}_{i}(l^*) := \int_{\RealField^N} p_{i}(x) e^{-2 \pi \sqrt{-1} x \cdot l^*} dx$.

Hence, the probability of $\hat{\wp}(l^*) = 0$ depends on the size of $W_{G, l^*} \subset \RealField^{N}/L$.
\begin{eqnarray}
	W_{G, l^*} := \left\{ x \in \RealField^{N}/L : \textstyle \sum_{\sigma \in G / L} e^{-2 \pi \sqrt{-1} x^\sigma \cdot l^* } = 0 \right\}.
\end{eqnarray}

\begin{definition}\label{def:definition of systematic absences}
For any finite subgroup $H \subset G$,
let $(\RealField^N)^H$ be the subset of $\RealField^N$ consisting of all the fixed points of $H$.
Under this notation, $\Gamma_{ext}(G)$ and $\Gamma_{ext}(G, H)$ are defined by
\begin{eqnarray}
	\Gamma_{ext}(G) &:=& \{ l^* \in L^* : W_{G, l^*} = \RealField^{N}/L \}, \\
	\Gamma_{ext}(G, H) &:=& \{ l^* \in L^* : (\RealField^N)^H / ((\RealField^N)^H \cap L) \subset W_{G, l^*} \}. \hspace{10mm} \label{eq:Gamma_{ext}(G, H, U)}
\end{eqnarray}
According to the terminology of crystallography, we say that 
$l^* \in L^*$ corresponds to a \textit{systematic absence at general positions (resp. special positions)}
if and only if $l^*$ belongs to $\Gamma_{ext}(G)$ (resp. $\Gamma_{ext}(G, H)$).
We shall call $(G, H)$ a \text{ type of systematic absences}.
\end{definition}


In the following, we shall focus on $\Gamma_{ext}(G, H)$,
since we have $\Gamma_{ext}(G) = \Gamma_{ext}(G, \{ id \})$.
From the definition, it is clear that $-l^*$ and $\tau l^*$ belong to $\Gamma_{ext}(G, H)$ for any $\tau \in R_G$ if and only if 
$l^*$ does.

For $\wp$ in (\ref{eq:representation of wp}) satisfying the isotropy condition,
it is not difficult to confirm that 
the following equivalence condition holds when $H_{i} \subset G$ is the stabilizer subgroup of $x_{i} \in \RealField^N$: 
\begin{eqnarray}\label{eq:bigcap_{i=1}^m Gamma_{ext}(G, H_{i}, x_{i} + V^{Sigma_{R_{H_i}}})}
	& & \hspace{-10mm}
	\hat{\wp}(l^*) = 0 \text{ holds constantly when} (x_{1}, \ldots, x_{m}, p_{1}, \ldots, p_{m}) \nonumber \\
	& & \hspace{-10mm}
	\text{ is perturbed in } (\RealField^N)^{H_{1}} \times \cdots \times (\RealField^N)^{H_{m}} \times ( L^2(\RealField^N)^{R_G} )^m \nonumber \\
	& \Longleftrightarrow & l^* \in \bigcap_{i=1}^m \Gamma_{ext}(G, H_{i}).
\end{eqnarray}

As proved by Bieberbach \cite{Bieberbach11}, \cite{Bieberbach12},
a homomorphism $\varphi : G_1 \longrightarrow G_2$ is an isomorphism between two crystallographic groups $G_1$ and $G_2$
			 if and only if
			there is an affine map $\alpha$ of $\RealField^N$ such that $\varphi(g) = \alpha g \alpha^{-1}$.
For such $G_1, G_2$, 
$\Gamma_{ext}(G_1, H) = \Gamma_{ext}(G_2, \varphi(H))$ clearly holds.
In general, we also have $\Gamma_{ext}(G, H) = \Gamma_{ext}(G, \sigma H \sigma^{-1})$ for any $\sigma \in G$.

\begin{proposition}\label{prop:basic property of extinction rules1}
		All types of systematic absences are classified by pairs $(G, H)$,
		where $G$, $H$ range respectively in 
		\begin{enumerate}[(a)]
			\item \label{item:isomorphism class of crystallographic group} isomorphism classes of a crystallographic group $G$,

			\item \label{item:subgroup H} conjugacy classes of finite subgroups $H \subset G$ in $G$.
			We also assume $H = \{ \sigma \in G : x^\sigma = x \}$ for some $x \in \RealField^N$.

		\end{enumerate}
		As a result, there are only finitely many types of systematic absences for each $N > 0$.
\end{proposition}

\begin{proof}
As proved by Bieberbach \cite{Bieberbach11}, \cite{Bieberbach12},
there are only finitely many isomorphism classes of crystallographic groups in each $N$.
Hence we shall only show that the number of the conjugacy classes is finite. 
For any finite subgroup $H \subset G$,
the natural epimorphism $\varphi : G \twoheadrightarrow G/L$ induces an injective map $H \hookrightarrow G/L$. 
Thus it is sufficient for the proof if 
we can prove that for any fixed subgroup $\tilde{H} \subset G / L$,
the set $S(\tilde{H}) := \{ H \subset G: \varphi(H) = \tilde{H}, H = \{ \sigma \in G : x^\sigma = x \} (\exists x \in \RealField^N) \}$
contains finitely many conjugacy classes.
Let $(\RealField^N / L)^{\tilde{H}}$ be the subset of $\RealField^N / L$ consisting of all the fixed points of $\tilde{H}$,
and $\pi$ be the natural onto map $\RealField^N \longrightarrow \RealField^N / L$.
For any $\bar{x} \in (\RealField^N / L)^{\tilde{H}}$ and $x \in \RealField^N$ with $\pi(x) = \bar{x}$,
let $H_{\bar{x}} \subset G$ be the subgroup consisting of all $\sigma \in G$ with $x^\sigma = x$.
The conjugacy class of $H_{\bar{x}}$ in $G$ clearly depends only on $\bar{x}$.
It is also straightforward to check that any elements of $S(\tilde{H})$ can be represented as $H_{\bar{x}}$ for some $\bar{x} \in (\RealField^N / L)^{\tilde{H}}$.
Furthermore, if $\bar{x}$ and $\bar{y}$ are contained in the same connected component of $(\RealField^N / L)^{\tilde{H}}$,
there is a path $w : [0, 1] \rightarrow (\RealField^N / L)^{\tilde{H}}$ such that $w(0) = \bar{x}$ and $w(1) = \bar{y}$,
and therefore $H_{\bar{x}} = H_{\bar{y}}$ must hold since $\varphi(H_{w(t)}) = \tilde{H}$ holds for any $t \in [0, 1]$ from the assumption.
Since $(\RealField^N / L)^{\tilde{H}}$ is compact, it contains only finitely many connected components.
As a result, $S(\tilde{H})$ contains only finitely many conjugacy classes.
\end{proof}

$\Gamma_{ext}(G, H)$ is computed using 
the following proposition.

\begin{proposition}\label{prop:basic property of extinction rules2}
			For fixed $l^* \in L^*$, the equivalence relation among the right cosets $R_H \backslash R_G$ is defined by: 
			\begin{eqnarray}
				R_H \tau_1 {\stackrel{l^*}{\sim}} R_H \tau_2
					\underset{def}{\Longleftrightarrow} \sum_{\tau \in R_H \tau_1} \tau l^* = \sum_{\tau \in R_H \tau_2} \tau l^*.
			\end{eqnarray}
			Then, for any $x \in (\RealField^N)^H$, 
			$l^* \in L^*$ belongs to $\Gamma_{ext}(G, H)$ if and only if the following holds:
			\begin{eqnarray}\label{eq:equivalent condition of extinction rules}
			\sum_{ R_H \tau_2 {\stackrel{l^*}{\sim}} R_H \tau_1 } e^{2 \pi \sqrt{-1} x^{ \{ \tau_2 | \nu_{\tau_2} \} } \cdot l^* } = 0 \text{ for every } R_H \tau_1 \in R_H \backslash R.		
			\end{eqnarray}
\end{proposition}

\begin{proof}
	We fix $x \in \RealField^N$ stabilized by $H$ arbitrarily.
	In this case, $l^* \in \Gamma_{ext}(G, H)$ holds if and only if
	\begin{eqnarray}\label{eq:extinction rule2}
		\sum_{ R_H \tau_1 \in R_H \backslash R} e^{2 \pi i ( (x + \delta x)^{\tau_1} + \nu_{\tau_1}) \cdot l^* } = 0
		\text{ for any } \delta x \in \RealField^N \text{ stabilized by } R_H.
	\end{eqnarray}
	Furthermore,
	\begin{eqnarray}\label{eq:equivalent condition}
		& & \hspace{-10mm}
		\delta x^{\tau_1} \cdot l^* = \delta x^{\tau_2} \cdot l^* \text{ for any } \delta x \in \RealField^N \text{ stabilized by } R_H \nonumber \\
		& \Longleftrightarrow & \sum_{\tau \in R_H} (\delta \tilde{x}^{\tau \tau_1} - \delta \tilde{x}^{\tau \tau_2}) \cdot l^* = 0 \text{ for any } \delta \tilde{x} \in \RealField^N \nonumber \\
		& \Longleftrightarrow & \sum_{\tau \in R_H} \delta \tilde{x} \cdot ({\tau \tau_1} l^* - {\tau \tau_2} l^*) = 0 \text{ for any } \delta \tilde{x} \in \RealField^N \nonumber \\
		& \Longleftrightarrow & \sum_{\tau \in R_H} \tau ({\tau_1}l^* - {\tau_2} l^*)  = 0
        \Longleftrightarrow R_H \tau_1 {\stackrel{l^*}{\sim}} R_H \tau_2. \hspace{10mm}
	\end{eqnarray}
	Hence, (\ref{eq:extinction rule2}) holds if and only if the following does for any $\delta x$ stabilized by $R_H$:
	\begin{eqnarray}\label{eq:extinction rule3}
			\sum_{ [\tau_1] \in (R_H \backslash R) / \stackrel{l^*}{\sim}} e^{2 \pi \sqrt{-1} \delta x^{\tau_1} \cdot l^* }
				\sum_{ R_H \tau_2 \stackrel{l^*}{\sim} R_H \tau_1 } e^{2 \pi \sqrt{-1} ( x^{\tau_2} + \nu_{\tau_2}) \cdot l^* } = 0,
	\end{eqnarray}
	which leads to the statement.
\end{proof}

\begin{corollary}\label{cor:definition of Omega}
Let $M$ be the order of $R_G$, and
${\mathcal H}_{G, H} \subset \RealField^N$ be the following union of finite linear subspaces of dimension less than $N$:
\begin{eqnarray}\label{eq:definition of {mathcal H}_{G, H}}
	{\mathcal H}_{G, H} := \bigcup_{ { R_H \tau_1, R_H \tau_2 \in R_H \backslash R_G,}\atop{ \sum_{\tau \in R_H} \tau (\tau_1 - \tau_2) \neq 0 } } \left\{ x^* \in \RealField^N : \textstyle \sum_{\tau \in R_H} \tau (\tau_1 - \tau_2) x^* = 0 \right\}.
\end{eqnarray}
There then exists $\Omega \subset L^* / M L^*$
such that $l^* \in \Gamma_{ext}(G, H) \Longleftrightarrow l^* + M L^* \in \Omega$
holds for any $l^* \in L^* \setminus {\mathcal H}_{G, H}$.
\end{corollary}
\begin{proof}
If $x$ is stabilized by $H$, $\nu_\tau \equiv x - x^\tau$ mod $L$ holds for any $\tau \in H$.
Hence,
$[\nu_\tau] \in H^1(R_G, \RealField^N / L)$ is mapped to 0 by the natural map $H^1(R_G, \RealField^N / L) \longrightarrow H^1(R_H, \RealField^N / L)$. As a result, it is also mapped to 0 
by $H^1(R_G, \RealField^N / L) \stackrel{\times M/m }{\longrightarrow} H^1(R_G, \RealField^N / L)$, where $m$ is the order of $R_H$ (\CF Proposition 6 in Chap.\ VII of Serre (1968)).
Thus, for any $\tau \in R_G$, there exist $y \in \RealField^N$ and $\mu_\tau \in (M/m)^{-1} L$
such that $\nu_\tau \equiv y - y^{\tau} + \mu_\tau$ mod $L$.
Hence, $\frac{M}{m} (x - y) \equiv \frac{M}{m} (x - y)^\sigma$ holds for any $\sigma \in R_H$.
Therefore $M (x- y) \equiv \sum_{\sigma \in R_H} \frac{M}{m} (x - y)^{\sigma}$.
If $u \in \RealField^N$ with $m u = x - y$ is fixed,
we have $(x- y) - \sum_{\sigma \in R_H} u^{\sigma} \in M^{-1} L$,
hence, for some $\xi_\tau \in M^{-1} L$,  $\nu_\tau$ is represented as follows:
\begin{eqnarray}
	\nu_\tau \equiv x - x^{\tau} - \sum_{\sigma \in R_H} u^{\sigma} + \sum_{\sigma \in R_H} u^{\sigma \tau} + \xi_\tau.
\end{eqnarray}
From Lemma \ref{prop:basic property of extinction rules2}, $l^* \in L^*$
belongs to $\Gamma_{ext}(G, H)$ if and only if the following holds for any $R_H \tau_1 \in R_H \backslash R_G$:
\begin{eqnarray}
\sum_{  R_H \tau_2 {\stackrel{l^*}{\sim}} R_H \tau_1 } e^{2 \pi i (x^{\tau_2} + \nu_{\tau_2}) \cdot l^* }
&=& \sum_{ R_H \tau_2 {\stackrel{l^*}{\sim}} R_H \tau_1 } e^{ 2 \pi i ( x - \sum_{\sigma \in R_H} u^{\sigma} + \sum_{\sigma \in R_H} u^{\sigma \tau_2} + \xi_{\tau_2} ) \cdot l^* } \nonumber \\
&=& 
e^{ 2 \pi i ( x - \sum_{\sigma \in R_H} u^{\sigma} ) \cdot l^* + 2 \pi i u \cdot \sum_{\sigma \in R_H} \sigma \tau_1 l^* } 
\sum_{ R_H \tau_2 {\stackrel{l^*}{\sim}} R_H \tau_1 } e^{ 2 \pi i \xi_{\tau_2} \cdot l^* } = 0. \hspace{10mm}
\end{eqnarray}
This is impossible if $l^*$ belongs to $M L^*$.
\end{proof}

There exists a simple condition equivalent to $l^* \in \Gamma_{ext}(G)$:
\begin{corollary}\label{cor:case of L=H}
For any given crystallographic group $G$,
	\begin{eqnarray}
		l^* \in L^* \text{ belongs to } \Gamma_{ext}(G)
		& \Longleftrightarrow & 
		\exists \tau \in R_G \text{ such that } \tau l^* = l^* \text{ and } \nu_{\tau} \cdot l^* \notin \IntegerRing. \hspace{10mm}
	\end{eqnarray}
\end{corollary}

\begin{proof}
Let $R_{G, l^*} \subset R_G$ be the stabilizer subgroup of $l^*$.
Then, $R_L \tau_2 {\stackrel{l^*}{\sim}} R_L \tau_1$ if and only if $\tau_2 \in \tau_1 R_{G, l^*}$.
Hence, 
\begin{eqnarray}
	\sum_{ R_L \tau_2 {\stackrel{l^*}{\sim}} R_L \tau_1 } e^{2 \pi \sqrt{-1} x^{ \{ \tau_2 | \nu_{\tau_2} \} } \cdot l^* } 	
	&=& \sum_{ \tau \in R_{G, l^*} } e^{2 \pi \sqrt{-1} ( x^{\tau_1 \tau} + \nu_{\tau_1 \tau} ) \cdot l^* } \nonumber \\
	&=& e^{2 \pi \sqrt{-1} (x^{\tau_1} + \nu_{\tau_1} ) \cdot l^* } \sum_{ \tau \in R_{G, l^*} } e^{2 \pi \sqrt{-1} \nu_{\tau} \cdot l^* }.
\end{eqnarray}
The statement follows 
from the fact that $\tau \mapsto e^{2 \pi \sqrt{-1} \nu_{\tau} \cdot l^*}$ is a homomorphism on $R_{G, l^*}$.
\end{proof}

Now that we have defined systematic absences, 
it is possible to formulate $\Lambda_{ext}(\wp)$ in (\ref{eq:definition of Lambda_{ext}(wp)}) precisely;
for any type $(G, H)$ of systematic absences,
we define 
\begin{eqnarray}
	\Lambda_{ext}(G, H) &:=& \left\{ |l^*|^2 : 0 \neq l^* \in L^* \setminus \Gamma_{ext}(G, H) \right\}, \\ 
	\Lambda_{ext}(G) &:=& \Lambda_{ext}(G, \{ id \}).
\end{eqnarray}

From the equivalence condition (\ref{eq:bigcap_{i=1}^m Gamma_{ext}(G, H_{i}, x_{i} + V^{Sigma_{R_{H_i}}})}),
for $\wp$ in (\ref{eq:representation of wp}),
\begin{eqnarray}
	\Lambda_{ext}(\wp) := \bigcup_{i=1}^m \Lambda_{ext}(G, H_i).
\end{eqnarray}
As a result, 
it is only necessary to consider the case of $\Lambda_{ext}(\wp) = \Lambda_{ext}(G, H)$ for all the types $(G, H)$ of systematic absences,
so as to discuss the problem ($A$\ref{item:extinction rules}) of systematic absences.

\section{$C$-type domains and Conway's topographs}
\label{Voronoi's second reduction theory and Conway topograph}

The reduction theory deals with the problem of specifying a domain $D \subset {\mathcal S}^N_{\succ 0}$
satisfying:
\begin{enumerate}[($R$1)]
	\item \label{item:reduction condition 0} 
			When we put $D[g] := \{ g S \tr{g} : S \in D \}$ for any subset $D \subset {\mathcal S}^N_{\succ 0}$ and $g \in GL_N(\IntegerRing)$, the subgroup of $GL_N(\IntegerRing)$
			consisting of all $g \in GL_N(\IntegerRing)$ satisfying $D = D[g]$ has only finite elements.
	\item \label{item:reduction condition 1} For any $g \in GL_N(\IntegerRing)$, $D$ and $D[g]$ do not share interior points,
			except for $g \in H$. 
	\item \label{item:reduction condition 2} ${\mathcal S}^N_{\succ 0}$ is decomposed as follows:
		\begin{eqnarray}\label{eq:cell docomposition}
			{\mathcal S}^N_{\succ 0} = \bigcup_{ g H \in GL_N(\IntegerRing) / H } D[g].
		\end{eqnarray}
\end{enumerate}

In \cite{Conway97}, a topograph was defined from the Selling reduction with $N = 2$.
So we shall recall the Selling reduction first;
let $A_N := (a_{ij})_{1 \leq i, j \leq N} \in {\mathcal S}^N_{\succ 0}$ 
be the Gram matrix of the root lattice ${\mathbb A}_N$ having entries as follows:
	\begin{eqnarray}
		a_{ij} &=&
			\begin{cases}
				2 & \text{if } i=j, \\
				1 & \text{otherwise}.
			\end{cases}
	\end{eqnarray}

Using $A_N$ and the inner-product $\langle S, T \rangle := {\rm Trace}(S T) = \sum_{i=1}^N \sum_{j=1}^N s_{ij} t_{ij}$ on ${\mathcal S}^N$,
the domain ${\mathcal D}^N_{Sel} \subset {\mathcal S}^N_{\succ 0}$ is defined by
\begin{eqnarray}\label{eq:Selling reduction domain}
		{\mathcal D}^N_{Sel} := \left\{ S \in {\mathcal S}^N_{\succ 0} : \langle S, A_N \rangle \leq \langle g S \tr{g}, A_N \rangle \text{ for any } g \in GL_N(\IntegerRing) \right\}.
\end{eqnarray}

If the subgroup of all $g \in GL_N(\IntegerRing)$ satisfying $\tr{g} A_N g = A_N$
is denoted by $H(A_N)$,
${\mathcal S}^N_{\succ 0}$ is partitioned using ${\mathcal D}^N_{Sel}$ by the Selling reduction \cite{Selling1874}:
\begin{eqnarray}\label{eq:partitioning of {mathcal S}^N_{succ 0}}
	{\mathcal S}^N_{\succ 0} = \bigcup_{ g \in GL_N(\IntegerRing) / H(A_N) } {\mathcal D}^N_{Sel}[g],
\end{eqnarray}

In order to generalize the definition of topographs for general $N$,
we'd like to note that the tessellation of (\ref{eq:partitioning of {mathcal S}^N_{succ 0}}) 
coincides with the one given by Voronoi's two reduction theories \cite{Voronoi07}, \cite{Voronoi08} in $N = 2, 3$.
Hence ${\mathcal D}^N_{Sel}$ ($N = 2, 3$) 
is same as the \textit{principal domain of the first type} defined in the two reduction theories.
In the first reduction theory, the principal domain
is defined as the convex cone expanded by $v \tr{v}$ of all the minimal vectors $v$ of $A_N$.  
In the second reduction theory, 
the same domain is represented as follows: 
\begin{eqnarray}
		{\mathcal V}(\Phi) &:=& \{ S \in {\mathcal S}_{\succ 0}^N : \tr{v} S u \leq \tr{u} S u \text{ for any } v \in \Phi \text{ and } u \in \IntegerRing^N \},\\
 		\Phi_0^N &:=& \left\{ \pm \textstyle \sum_{k=1}^N i_k {\mathbf e}_k : i_k = 0, 1 \right\}. \label{eq:definition of Phi_0^N}
\end{eqnarray}

Different from ${\mathcal D}^N_{Sel}$,
in the tessellation of ${\mathcal S}^N_{\succ 0}$ given by the first reduction theory,
every domain is associated explicitly with the set of integral vectors formed by minimal vectors of a perfect form.
Even in the second reduction theory,
such an association is provided  
by using $C$-type domains (a union of finite $L$-type domains), instead of $L$-type domains.

In the following, we choose the association by the second reduction theory,
because of Proposition \ref{prop:facet and parallelogram law}
on the association of facets of primitive $C$-type domains and the parallelogram law.

\subsection{Topographs for lattices of general rank}
\label{Definition for lattices of general rank}

In this section, $C$-type domains and topographs are defined for the general dimension $N$.
We also refer to \cite{Ryshkov76} for more detailed information about $C$-type domains and its connection with covering problems.
In \cite{Conway97}, a topograph was defined to explain the Selling reduction with $N = 2$.
In contrast, vonorms and conorms were used for the case of $N > 2$.

The idea of vonorms and conorms as invariants of a lattice
seems to have originated from the Voronoi vectors defined in Voronoi's second reduction theory;
for a fixed $S \in {\mathcal S}^N_{\succeq 0}$,
if $v \in \IntegerRing^N$ satisfies $\tr{v} S v = {\rm vo}_S(v + 2 \IntegerRing^N)$,
$v$ is called a \textit{Voronoi vector} of $S$.
The \textit{vonorm map} of $S$ is defined as a map from $v + 2 \IntegerRing^N \in \IntegerRing^N / 2 \IntegerRing^N$
to the representation by the Voronoi vector corresponding to $v + 2 \IntegerRing^N$.
\begin{eqnarray}
{\rm vo}_S(v + 2 \IntegerRing^N) := \min \{ \tr{w} S w : w \in v + 2 \IntegerRing^N \}.
\end{eqnarray}


Conorms are the Fourier transform of vonorms: 
when $\chi$ is any character on $\IntegerRing^N / 2 \IntegerRing^N$,
the \textit{conorm} map of $S$ is defined by: 
\begin{eqnarray}
{\rm co}_S(\chi) := - \frac{1}{2^{N-1}} \sum_{v + 2 \IntegerRing^N \in \IntegerRing^N / 2 \IntegerRing^N} {\rm vo}_S(v + 2 \IntegerRing^N) \chi(v).
\end{eqnarray}

In the following, we use vonorm maps in the definition of $C$-type domains for clarity.
Before introducing $C$-type domains, we shall recall the definition of $L$-type domains (also called secondary cones \CF \cite{Vallentin2003});
The Dirichlet--Voronoi polytope of $S$ is defined by: 
\begin{eqnarray}
	{\rm DV}(S) := \{ x \in \RealField^N : \tr{x} S x \leq \tr{(x + l)} S (x + l) \text{ for any } l \in \IntegerRing^N \}.
\end{eqnarray}

From the definition, ${\rm DV}(S)$ is the intersection of half-spaces:
\begin{eqnarray}
	{\rm DV}(S) &=& \bigcap_{0 \neq v \in \IntegerRing^N}
				\{ x \in \RealField^N : \tr{x} S x \leq \tr{(x + v)} S (x + v) \}.
\end{eqnarray}

As proved in \cite{Voronoi08} (\CF \cite{Conway92}), 
$v \in \IntegerRing^N$ is a Voronoi vector
if and only if the hyperplane 
$H_{S, v} := \left\{ x \in \RealField^N : \tr{x} S x = \tr{(x + v)} S (x + v) \right\}$
intersects ${\rm DV}(S)$.

A tiling of $\RealField^N$ is given by the Dirichlet--Voronoi polytopes:
\begin{eqnarray}\label{eq:tiling by Diriclet-Voronoi polytopes}
	\RealField^N = \bigcup_{l \in \IntegerRing^N} ({\rm DV}(S) + l).
\end{eqnarray}

The Delone subdivision is a dual tiling of (\ref{eq:tiling by Diriclet-Voronoi polytopes}).
If we let $P_S$ be the set of extreme points of ${\rm DV}(S)$,
and denote the set of all Voronoi vectors $v$ satisfying $p \in H_{S, v}$ by $\Psi_p$ for any $p \in P_S$, then 
every $v \in \Psi_p$
satisfies $\tr{p} S p = \tr{(p + v)} S (p + v)$.
Therefore, an ellipsoid $\{ x \in \RealField^N : \tr{(x + p)} S (x + p) = \tr{p} S p \}$
passes through all the elements of $\{ 0 \} \cup \Psi_p \subset \IntegerRing^N$.
If $L_p$ is the convex hull of $\{ 0 \} \cup \Psi_p$, then the Delone subdivision of $\RealField^N$ 
is given by:
\begin{eqnarray}
	{\rm Del}(S) := \bigcup_{ p \in P_S, l \in \IntegerRing^N } L_p + l.
\end{eqnarray}

In this case, a \textit{$L$-type domain} containing $S$ is defined by
$\{ S_2 \in {\mathcal S}^N : {\rm Del}(S) = {\rm Del}(S_2) \}$.

For any subset $D \subset {\mathcal S}^N_{\succ 0}$ and $g \in GL_N(\IntegerRing)$,
define $D[g] := \{ g S \tr{g} : S \in D \}$.
Two domains $D_1, D_2 \subset {\mathcal S}^N_{\succ 0}$ are said to be \textit{equivalent}
if and only if $D_1[g] = D_2$ holds for some $g \in GL_N(\IntegerRing)$.
Voronoi proved that the number of equivalence classes of $L$-type domains of dimension $\frac{N(N+1)}{2}$ 
is finite, and provided an algorithm to gain all the equivalence classes \cite{Voronoi08}.
This is the outline of Voronoi's second reduction theory.

For fixed $S \in {\mathcal S}^N_{\succ 0}$,
define $\Phi_S := \{ [v] : v \in \IntegerRing^N \text{ is a Voronoi vector of } S \}$,
where $[v]$ represents the class of $v \in \IntegerRing^N$ when $v$ and $-v$ are identified.
A \textit{$C$-type domain} containing $S$ is defined by
${\mathcal V}(\Phi) := \{ S_2 \in {\mathcal S}_{\succ 0}^N : \Phi_S = \Phi_{S_2} \}$, \IE a set of elements of ${\mathcal S}_{\succ 0}^N$ having the same Voronoi vectors.
Then, ${\mathcal V}(\Phi)$
is a union of finite $L$-type domains,
because a set of Voronoi vectors of $S$
is decomposed into $\bigcup_{p \in P_S} \Psi_p$ in different ways, depending on $S$.
From the definition, the Voronoi map is a linear function on ${\mathcal V}(\Phi_S)$ for any $S \in {\mathcal S}^N_{\succ 0}$.

More generally, we shall define $C$-type domains for any set $\Phi := \{ [v_1], \ldots, [v_{n}] \}$ of arbitrary size $n$ 
with $v_1, \ldots, v_n \in \IntegerRing^N$: 
\begin{eqnarray}
	{\mathcal V}(\Phi) := \{ S \in {\mathcal S}^N_{\succ 0} : \tr{v} S v = {\rm vo}_S(v + 2 \IntegerRing^N)  \text{ for any } [v] \in \Phi \}.
\end{eqnarray}

This is well defined, because both $\tr{v} S v$ and $v + 2 \IntegerRing$ are invariant if $v$ is replaced by $-v$.
From the definition, ${\mathcal V}(\Phi)$ is an intersection of the following half-spaces.
\begin{eqnarray}
	{\mathcal V}(\Phi) &=& \bigcap_{ [v] \in \Phi } \bigcap_{ u + 2 \IntegerRing^N = v + 2 \IntegerRing^N} H^{\geq 0}(u, v), \\
	H^{\geq 0}(u, v) &:=& \{ S \in {\mathcal S}^N_{\succ 0} : \tr{u} S u \geq \tr{v} S v \}.
\end{eqnarray}

Any $S \in {\mathcal S}^N_{\succ 0}$ is contained in ${\mathcal V}(\Phi_S)$.
$S$ is said to be in a \textit{general position} 
if $\Phi_S$ has exactly $2^N$ elements.  
If $S$ is in a general position,
${\mathcal V}(\Phi_S)$ includes an open neighbor of $S$.
Such $C$-type domains are called \text{primitive}.
Otherwise, there exist $[v] \neq [u] \in \IntegerRing^N$ such that $u + 2 \IntegerRing^N = v + 2 \IntegerRing^N$,
and $S$ belongs to the following hyperplanes:
\begin{eqnarray}
	H( u, v ) := \{ S \in {\mathcal S}^N : \tr{u} S u = \tr{v} S v \}.
\end{eqnarray}

Even for such $S$, by perturbing the entries of $S$, 
$\tilde{S}$ in a general position satisfying $\Phi_S \subset \Phi_{\tilde{S}}$ 
is obtained.
As a result, the following partitioning of ${\mathcal S}^N_{\succ 0}$ is obtained.
\begin{eqnarray}\label{eq:decomposition of S_{> 0}^N}
	{\mathcal S}^N_{\succ 0} &=& \bigcup_{{\mathcal V} \in P_N} {\mathcal V}, \\
	P_N &:=& \left\{ {\mathcal V}(\Phi_S) : S \in {\mathcal S}_{\succ 0}^N, {\mathcal V}(\Phi_S) \text{ is primitive} \right\}.
\end{eqnarray}

This tessellation is coarser than that given by $L$-type domains.
Hence, similarly with $L$-type domains,
the number of equivalence classes of $C$-type domains of dimension $\frac{N(N+1)}{2}$ is finite.

Table \ref{extreme rays and facets of all the representatives of equivalent V-type domains}
lists all the representatives of equivalence classes of $C$-type domains for $1 \leq N \leq 4$.
Each domain ${\mathcal V}(\Phi_i^k)$ in Table \ref{extreme rays and facets of all the representatives of equivalent V-type domains} 
has a set of extreme rays provided by: 
	\begin{eqnarray}
		M(\Phi_0^k) &:=& \left\{ v \tr{v} : v = {\mathbf e}_i\ (1 \leq i \leq N), {\mathbf e}_i - {\mathbf e}_j\ (1 \leq i < j \leq N) \right\}, \\
M(\Phi_1^4) &:=& \left( M(\Phi_0^4) \cup \left\{ D_4 \right\} \right)
				\setminus \left\{ v \tr{v} : v = {\mathbf e}_1 - {\mathbf e}_2 \right\}, \\
M(\Phi_2^4) &:=& \left( M(\Phi_0^4) \cup \left\{ v \tr{v} : v = {\mathbf e}_1 + {\mathbf e}_2 - {\mathbf e}_3 - {\mathbf e}_4 \right\}
				\cup \left\{ D_4, D_{4,2} \right\} \right)
			\setminus \left\{ v \tr{v} : v = {\mathbf e}_1 - {\mathbf e}_2, {\mathbf e}_3 - {\mathbf e}_4 \right\}, \hspace{10mm}
	\end{eqnarray}
where 
$D_4$ and $D_{4,2}$ are the equivalent perfect forms corresponding to the root lattice ${\mathbb D}_4$:
	\begin{eqnarray}\label{def:definition of D_4}
		D_4 &:=& 
			\begin{pmatrix}
				2 &  1 & -1 & -1 \\ 
			    1 &  2 & -1 & -1 \\ 
			   -1 & -1 &  2 &  0 \\ 
			   -1 & -1 &  0 &  2 
			\end{pmatrix},\
		D_{4,2} :=
			\begin{pmatrix}
				2 &  0 & -1 & -1 \\ 
			    0 &  2 & -1 & -1 \\ 
			   -1 & -1 &  2 &  1 \\ 
			   -1 & -1 &  1 &  2 
			\end{pmatrix}.
	\end{eqnarray}

\begin{table}
	\caption{Equivalence classes of primitive $C$-type domains}
	\label{extreme rays and facets of all the representatives of equivalent V-type domains}
	\begin{minipage}{\textwidth}
	\begin{tabular}{lll}
	\hline
	& \textbf{Voronoi vectors} & \textbf{Primitive $C$-type domain} \\
	\textbf{$N = 1$}
	& $\Phi_0^1 := \left\{ [0], [{\mathbf e}_1] \right\}$
	& ${\mathcal V}(\Phi_0^1) = {\mathcal S}^1_{\geq 0}$. \\
	\\
	\textbf{$N = 2$}
	& $\Phi_0^2 := \left\{ [0], [{\mathbf e}_1], [{\mathbf e}_2], [{\mathbf e}_1 + {\mathbf e}_2] \right\}$
	& ${\mathcal V}(\Phi_0^2) = \bigcap_{1 \leq i < j \leq 3} H^{\geq 0}( {\mathbf e}_i - {\mathbf e}_j, {\mathbf e}_i + {\mathbf e}_j )$ \\
	\\
	\textbf{$N = 3$}
	& $\Phi_0^3 := \left\{ \left[ \sum_{k=1}^3 i_k {\mathbf e}_k \right] : i_k = 0 \text{ or } 1 \right\}$
	& ${\mathcal V}(\Phi_0^3) = \bigcap_{1 \leq i < j \leq 4} H^{\geq 0}( {\mathbf e}_i - {\mathbf e}_j, {\mathbf e}_i + {\mathbf e}_j )$ \\
	\\
	\textbf{$N = 4$}
	& $\Phi_0^4 := \left\{ \left[ \sum_{k=1}^4 i_k {\mathbf e}_k \right] : i_k = 0 \text{ or } 1 \right\}$
	& ${\mathcal V}(\Phi_0^4) = \bigcap_{1 \leq i < j \leq 5} H^{\geq 0}( {\mathbf e}_i - {\mathbf e}_j,  {\mathbf e}_i + {\mathbf e}_j )$ \\
	\\
	& $\begin{matrix} \Phi_1^4 := \Phi_0^4 \cup \left\{ [{\mathbf e}_1 - {\mathbf e}_2] \right\} \setminus \left\{ [{\mathbf e}_1 + {\mathbf e}_2] \right\} \end{matrix}$ &
	${\mathcal V}(\Phi_1^4) = H^{\geq 0} \left( {\mathbf e}_1 + {\mathbf e}_2, {\mathbf e}_1 - {\mathbf e}_2 \right)$ \\
	& & \hspace{15mm} $\cap \left( \bigcap_{3 \leq i < j \leq 5} H^{\geq 0} \left( {\mathbf e}_i - {\mathbf e}_j, {\mathbf e}_i + {\mathbf e}_j \right) \right)$ \\
	& & \hspace{15mm} $\cap \left( \bigcap_{{3 \leq i < j \leq 5,}\atop{k = 1, 2}} H^{\geq 0} \left( {\mathbf e}_i + {\mathbf e}_j + 2 {\mathbf e}_k, {\mathbf e}_i + {\mathbf e}_j \right) \right)$ 
	\\
	& $\begin{matrix} \Phi_2^4 := \Phi_0^4 \cup \left\{ [{\mathbf e}_1 - {\mathbf e}_2], [{\mathbf e}_3 - {\mathbf e}_4] \right\} \\ \hspace{5mm} \setminus \left\{ [{\mathbf e}_1 + {\mathbf e}_2], [{\mathbf e}_3 + {\mathbf e}_4] \right\} \end{matrix}$ 
	& ${\mathcal V}(\Phi_2^4)
		= H^{\geq 0} \left( {\mathbf e}_1 + {\mathbf e}_2, {\mathbf e}_1 - {\mathbf e}_2 \right)
			\cap H^{\geq 0} \left( {\mathbf e}_3 + {\mathbf e}_4, {\mathbf e}_3 - {\mathbf e}_4 \right)$ \\
	& & \hspace{15mm} $\cap \left( \bigcap_{k = 3, 4} H^{\geq 0} \left( {\mathbf e}_1 + {\mathbf e}_2 + 2 {\mathbf e}_k, {\mathbf e}_1 - {\mathbf e}_2 \right) \right)$ \\
	& & \hspace{15mm} $\cap \left( \bigcap_{k = 1, 2} H^{\geq 0} \left( {\mathbf e}_3 + {\mathbf e}_4 + 2 {\mathbf e}_k, {\mathbf e}_3 - {\mathbf e}_4 \right) \right)$ \\
	& & \hspace{15mm} $\cap \left( \bigcap_{i = 1}^4 H^{\geq 0} \left( {\mathbf e}_i - {\mathbf e}_5, {\mathbf e}_i + {\mathbf e}_5 \right) \right)$ \\
	& & \hspace{15mm} $\cap \left( \bigcap_{ {i = 1, 2, k = 3, 4}\atop{\text{or } i = 3, 4, k = 1, 2} } H^{\geq 0} \left( {\mathbf e}_i + {\mathbf e}_5 + 2 {\mathbf e}_k, {\mathbf e}_i + {\mathbf e}_5 \right) \right)$ \\
	\end{tabular}
	\footnotetext[1]{Here, we put ${\mathbf e}_{N+1} := -\sum_{i=1}^N {\mathbf e}_{i}$ in every $N$-dimensional case.}
	\footnotetext[2]{Among the above domains, only ${\mathcal V}(\Phi_2^4)$ is not an $L$-type domain, but is a union of two $L$-type domains.  $L$-type domains for $N = 4$ are listed in \cite{Vallentin2003}. 
	}
\end{minipage}
\end{table}

The following proposition claims that every facet of a primitive $C$-type domain
is associated with a set of four vectors satisfying the parallelogram law.
Although this is the most important property in our discussion, 
we could not find references mentioning this explicitly.

\begin{proposition}\label{prop:facet and parallelogram law}
Suppose that two primitive $C$-type domains ${\mathcal V}(\Phi_{S_1}) \neq {\mathcal V}(\Phi_{S_2})$
have an  $(\frac{N(N+1)}{2} - 1)$-dimensional cone as their intersection.
Then $\Phi_{S_1} \cap \Phi_{S_2}$ contains exactly $2^N - 1$ elements.
Hence, there exist $u, v \in \IntegerRing^N$ such that 
$\Phi_{S_1} \setminus \Phi_{S_2} = \{ [ u ] \}$ and $\Phi_{S_2} \setminus \Phi_{S_1} = \{ [ v ] \}$.
${\mathcal V}(\Phi_{S_1} \cup \Phi_{S_2}) \subset H( u, v )$ is 
a common facet of ${\mathcal V}(\Phi_{S_1})$ and ${\mathcal V}(\Phi_{S_2})$.
Furthermore,
$\{ \frac{v + u}{2}, \frac{v - u}{2} \}$ is a primitive set of $\IntegerRing^N$,
and $[ \frac{v + u}{2} ]$, $[ \frac{v - u}{2} ]$ are elements of $\Phi_{S_1} \cap \Phi_{S_2}$.
\end{proposition}

\begin{proof}
From ${\mathcal V}(\Phi_{S_1}) \neq {\mathcal V}(\Phi_{S_2})$,
there exist $u_1, v_1$ such that $u_1 + 2 \IntegerRing^N = v_1 + 2 \IntegerRing^N$,
$[u_1] \in \Phi_{S_1} \setminus \Phi_{S_2}$, and $[v_1] \in \Phi_{S_2} \setminus \Phi_{S_1}$.
Since $u_1$, $v_1$ are Voronoi vectors of any $S_3 \in {\mathcal V}(\Phi_{S_1}) \cap {\mathcal V}(\Phi_{S_2})$,
we have $\tr{u_1} S_3 l \leq \tr{l} S l$ and 
$\tr{v_1} S_3 l \leq \tr{l} S l$.
Hence, $\frac{u_1 + v_1}{2}$ and $\frac{u_1 - v_1}{2}$ are also Voronoi vectors of $S_3$.
Replacing $u_i, v_i$ with $g u_i, g v_i$ ($g \in GL_N(\IntegerRing)$) if necessary,
we may assume $\frac{u_1 - v_1}{2} = {\mathbf e}_1$ and $\frac{u_1 + v_1}{2} = m{\mathbf e}_1 + n{\mathbf e}_2$ for some $m, n \in \IntegerRing$.
Then $u_1, v_1$ and $\frac{u_1 \pm v_1}{2}$
are Voronoi vectors of $S_4 := (\tr{{\mathbf e}_i} S_3 {\mathbf e}_j)_{1 \leq i, j \leq 2} \in {\mathcal S}^2_{\succ 0}$.
Hence $S_4$ belongs to ${\mathcal V}(\{ [\frac{u_1 \pm v_1}{2}], u_1 \})$, 
which is equivalent to ${\mathcal V}(\Phi_0^2)$ in Table \ref{extreme rays and facets of all the representatives of equivalent V-type domains}. 
As a result, there exists $g \in GL_N(\IntegerRing)$ such that 
$g (\frac{u_1 - v_1}{2}) = {\mathbf e}_1$ and $g (\frac{u_1 + v_1}{2}) = {\mathbf e}_2$.
Hence, it is concluded that $\{ \frac{u_1 + v_1}{2}, \frac{u_1 - v_1}{2} \}$ is a primitive set of $\IntegerRing^N$.
Suppose that there exists another $(u_2, v_2) \neq (u_1, v_1)$
satisfying $u_2 + 2 \IntegerRing^N = v_2 + 2 \IntegerRing^N$,
$[u_2] \in \Phi_{S_1} \setminus \Phi_{S_2}$, and $[v_2] \in \Phi_{S_2} \setminus \Phi_{S_1}$.
From the dimension of ${\mathcal V}(\Phi_{S_1}) \cap {\mathcal V}(\Phi_{S_2})$,
the following must hold for any $S_3 \in {\mathcal S}^N$:
\begin{eqnarray}\label{eq:equivalence condition from H(u, v)}
	\tr{(\frac{u_1 + v_1}{2})} S_3 (\frac{u_1 - v_1}{2}) = 0 \Longleftrightarrow \tr{(\frac{u_2 + v_2}{2})} S_3 (\frac{u_2 - v_2}{2}) = 0.
\end{eqnarray}
We may now assume $\frac{u_1 - v_1}{2} = {\mathbf e}_1$ and $\frac{u_1 + v_1}{2} = {\mathbf e}_2$.
Then, $\{ [u_1], [v_1] \} = \{ [u_2], [v_2] \}$
 is easily obtained.
Therefore,
$\Phi_{S_1} \setminus \Phi_{S_2}$ and $\Phi_{S_2} \setminus \Phi_{S_1}$
consist of only one element. 
Now the remaining statements follow immediately. 
\end{proof}

By Proposition \ref{prop:facet and parallelogram law}, it is proved that the decomposition (\ref{eq:decomposition of S_{> 0}^N})
is a facet-to-facet tessellation.
Hence, (\ref{eq:decomposition of S_{> 0}^N}) is also a face-to-face tessellation 
by a theorem of Gruber and Ryshkov \cite{Gruber89}. 

Using the $V$-partitioning (\ref{eq:decomposition of S_{> 0}^N}), a topograph is defined for general $N$.
\begin{definition}\label{def:definition of topographs}
Define $CT_{N}$ as the graph
which has $P_N$, $E_N$ 
as its sets of nodes and edges, respectively.
\begin{eqnarray}
	P_N &:=& \left\{ {\mathcal V}(\Phi_S) : S \in {\mathcal S}^N_{\succ 0},\ \Phi_S \textit{ is primitive} \right\}, \\
	E_N &:=& \left\{ e_{ {\mathcal V}(\Phi_{1}), {\mathcal V}(\Phi_{2}) } : {\mathcal V}(\Phi_{1}) \neq {\mathcal V}(\Phi_{2}) \in P_N \text{ share a facet} \right\},
\end{eqnarray}
where $e_{ {\mathcal V}(\Phi_{1}), {\mathcal V}(\Phi_{2}) }$ connects two nodes ${\mathcal V}(\Phi_{1})$, ${\mathcal V}(\Phi_{2}) \in P_N$.
When $S_0 \in {\mathcal S}^N$ is fixed arbitrarily,
edges in $E_N$ are associated with two representations of $S_0$ over $\IntegerRing$ by the map:
\begin{eqnarray}
	f_{S_0}( e_{ {\mathcal V}(\Phi_{1}), {\mathcal V}(\Phi_{2}) } ) := \{ \tr{u} S_0 u, \tr{v} S_0 v \},
\end{eqnarray}
where $u, v \in \IntegerRing^N$ are taken so that 
$\{ [u+v] \} = \Phi_{S_1} \setminus \Phi_{S_2}$ and $\{ [u-v] \} = \Phi_{S_2} \setminus \Phi_{S_1}$.
Such an edge is represented in Figure \ref{Edge and associated vectors}.
Furthermore, the direction of the edge is defined as in Figure \ref{Direction of edges in topographs}.  
Assuming every edge is oriented by this,
we call the pair $CT_{N, S_0} := (CT_N, f_{S_0})$ a \textit{topograph} of $S_0$.
\end{definition}

\begin{figure}
\begin{minipage}{\textwidth}
\begin{center}
\scalebox{0.64}{\includegraphics{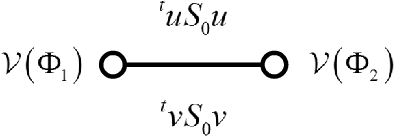}}
\end{center}
\end{minipage}
\caption{Representations associated with an edge of a topograph.}
\label{Edge and associated vectors}
\end{figure}

\begin{figure}
\begin{minipage}{\textwidth}
\begin{center}
\scalebox{0.64}{\includegraphics{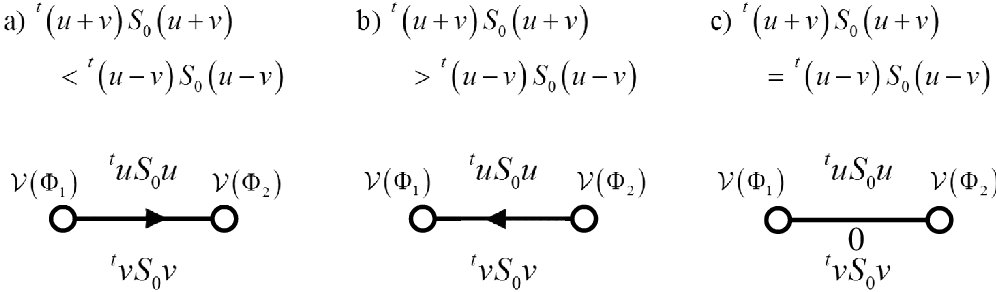}}
\end{center}
\footnotetext{
When ${\mathcal  V}(\Phi_1)$, ${\mathcal V}(\Phi_2) \in P_N$ 
have a common facet, there exists $u, v \in \IntegerRing^N$
such that $\Phi_1 \setminus \Phi_2 = \{ \pm (u + v) \}$,
$\Phi_2 \setminus \Phi_1 = \{ \pm (u - v) \}$.
In this case, the direction of the edge connecting ${\mathcal V}(\Phi_1)$ and ${\mathcal V}(\Phi_2)$
is defined as in \cite{Conway97}.}
\end{minipage}
\caption{Direction of edges of a topograph $CT_{N, S_0}$.}
\label{Direction of edges in topographs}
\end{figure}

\subsection{Topographs for low-dimensional lattices}
\label{Topographs for low-dimensional lattices}

In this section, the structures of topographs for lattices of rank $N = 2, 3$ are explained.
Since the same topic is also discussed in \cite{Conway97},
we mention only basic facts necessary in the following sections.
In $N = 2, 3$, the structures are also determined from the partitioning (\ref{eq:partitioning of {mathcal S}^N_{succ 0}})
of the Selling reduction.
In particular, the set of nodes is provided by
\begin{eqnarray}
	P_N &:=& \left\{ {\mathcal D}^N_{Sel}[g] = {\mathcal V}(\tr{g}^{-1} \Phi_0^N) : g \in GL_N(\IntegerRing) \right\}.
\end{eqnarray}
By Voronoi's second reduction theory,
every node ${\mathcal D}^N_{Sel}[g]$ is associated with $\tr{g}^{-1} \Phi_0^N$.

In the following, a lattice $L$ of rank $N$, a basis $b_1, \ldots, b_N$ of $L$ are fixed,
and a matrix $\begin{pmatrix} b_1 & \cdots & b_N \end{pmatrix}$ is denoted by $B$.
In order to clarify $\tr{B} B$ is the Gram matrix of $(L, B)$,
we utilize the following notation, instead of $f_{\tr{B} B}$, $CT_{N, \tr{B} B}$.
\begin{eqnarray}\label{eq:definition of f_L} 
	f_{(L, B)}( e_{ {\mathcal V}(\Phi_1), {\mathcal V}(\Phi_2) } ) &:=& f_{\tr{B} B}( e_{ {\mathcal V}(\Phi_1), {\mathcal V}(\Phi_2) } ) = \{ \abs{B u}^2, \abs{B v}^2 \}, \\
	CT_{N, (L, B)} &:=& CT_{N, \tr{B} B}.
\end{eqnarray}
where $u, v \in \IntegerRing^N$ are chosen as in Definition \ref{def:definition of topographs}.

Basic properties of $CT_{2, (L, B)}$ and $CT_{3, (L, B)}$ are explained in the following examples.

\begin{example}\label{Topographs in N = 2.}{\bf Case of $CT_{2, (L, B)}$.}
${\mathcal D}_{Sel}^2 = {\mathcal V}(\Phi_0^2)$ is a polyhedral cone surrounded by the three hyperplanes
 in Table \ref{extreme rays and facets of all the representatives of equivalent V-type domains}.
Hence, a node of $CT_2$ is adjacent to three nodes, as in Figure \ref{Local structure of a topograph CT_2}.

\begin{figure}[htbp]
\begin{minipage}{\textwidth}
\begin{center}
\scalebox{0.64}{\includegraphics{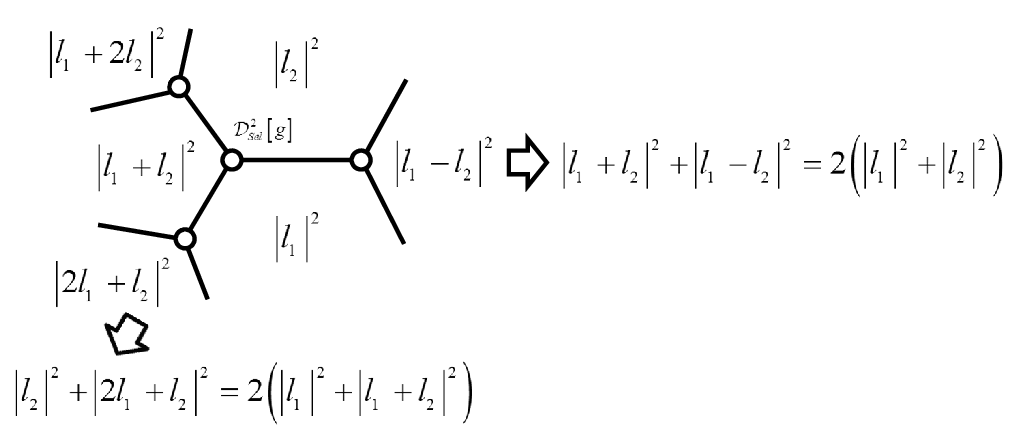}}
\end{center}
\footnotetext{
Let ${\mathbf e}_3 := - {\mathbf e}_1 - {\mathbf e}_2$, and
$\tau^{(2)}_{ij}$ be the $2 \times 2$ matrix satisfying
\begin{eqnarray}
	\tau_{ij}^{(2)} {\mathbf e}_i  = - {\mathbf e}_j,\
	\tau_{ij}^{(2)} {\mathbf e}_j  = {\mathbf e}_i.
\end{eqnarray}
When we put $(l_1\ l_2) := (b_1\ b_2) \tr{g}^{-1}$ and $l_3 := -l_1 - l_2$ for the fixed basis $b_1$, $b_2$ of $L$ and $g \in GL_2(\IntegerRing)$,
every node ${\mathcal D}_{Sel}^2[g] = {\mathcal V}(\tr{g}^{-1} \Phi_0^2)$ is an end point 
of three edges $e_{ {\mathcal V}(\tr{g}^{-1} \Phi_0^2), {\mathcal V}(\tr{g}^{-1} \tau_{ij}^{(2)} \Phi_0^2) }$ ($1 \leq i < j \leq 3$)
that are associated with $\{ |l_i|^2, |l_j|^2 \}$.
(In this figure, any edge between two domains labeled $\abs{ k_1 }^2$ and $\abs{ k_2 }^2$
is considered to be associated with $\{ |k_1|^2, |k_2|^2 \}$.  
Such a labeling of domains is achieved by embedding the topograph in the upper half plane as in 
Figure \ref{Tree with action PGL(2, IntegerRing)}, and labeling every domain $D \subset {\mathbb H}$ 
surrounded by topograph edges with $\abs{l}^2$ 
of the minimal vector $l$ of some $S \in \iota^{-1}(D)$.  This is well-defined,
because $l$ is the only minimal vector of all $S \in \iota^{-1}(D)$.)
}
\end{minipage}
\caption{Local structure of topograph $CT_{2, (L, (b_1\ b_2))}$.} 
\label{Local structure of a topograph CT_2}
\end{figure}

It can be proved that $CT_2$ is a tree as in Figure \ref{Tree with action PGL(2, IntegerRing)}.
The tree is embedded in ${\mathcal S}^2_{\succ 0} / \RealField_{> 0} \simeq {\mathbb H}$
by mapping a node ${\mathcal V}(\tr{g}^{-1} \Phi_0^2)$ to the perfect form $g A_2^{-1} \tr{g}$,
and an edge between ${\mathcal V}(\tr{g_1}^{-1} \Phi_0^2)$ and ${\mathcal V}(\tr{g_2}^{-1} \Phi_0^2)$
to the geodesic connecting $g_1 A_2^{-1} \tr{g_1}$, $g_2 A_2^{-1} \tr{g_2}$.

\begin{figure}[htbp]
\begin{minipage}{\textwidth}
\begin{center}
\scalebox{0.6}{\includegraphics{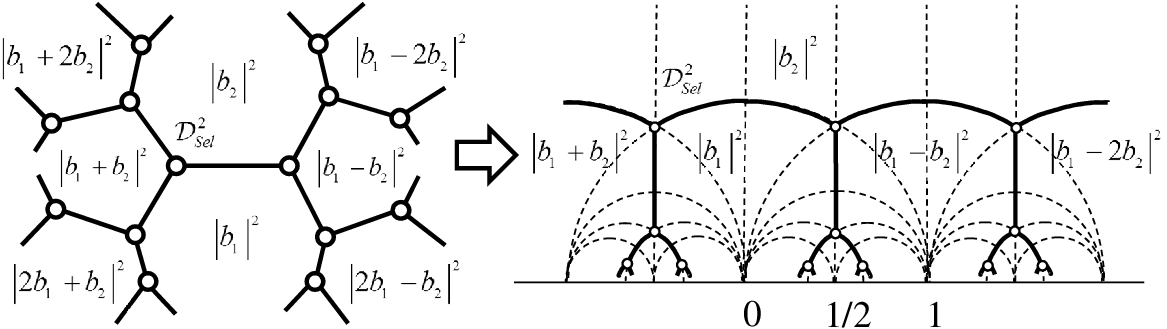}}
\end{center}
\footnotetext{
A one-to-one correspondence between the points of the upper half-plane ${\mathbb H} := \{ z \in \ComplexField: {\rm Im}(z) > 0 \}$ and ${\mathcal S}^2_{\succ 0} / \RealField_{> 0}$ is given by the map:  
	\begin{eqnarray}
		\iota :
		x + \sqrt{-1} y \mapsto
		\begin{bmatrix}
			x^2 + y^2 & x \\
			x & 1
		\end{bmatrix}. \label{eq:map to upper half-plane}
	\end{eqnarray}
 	Under this identification of ${\mathbb H}$ and ${\mathcal S}^2_{\succ 0} / \RealField_{> 0}$,
 	the action of $g \in GL_2(\IntegerRing)$ on the latter set by $S \mapsto g S \tr{g}$ coincides with the action of $GL_2(\IntegerRing)$ on ${\mathbb H}$ by:
 	\begin{eqnarray}
 		\begin{pmatrix}
 			a & b \\
 			c & d
 		\end{pmatrix}
 		\cdot z
 		:= 
 		\begin{cases}
 			\frac{a z + b}{c z + d} & \text{if } a d - b c = 1, \\
 			\frac{a \bar{z} + b}{c \bar{z} + d} & \text{if } a d - b c = -1,
 		\end{cases}
	\end{eqnarray}
 	where $\bar{z}$ is the complex conjugate of $z$.
}
\end{minipage}
\caption{Embedding of a topograph in the upper half plane.}
\label{Tree with action PGL(2, IntegerRing)}
\end{figure}

\end{example}

\begin{example}\label{Topographs in N = 3.}{\bf Case of $CT_{3, (L, B)}$.}
${\mathcal D}_{Sel}^3 = {\mathcal V}(\Phi_0^3)$ is a polyhedral cone surrounded by the six hyperplanes in Table \ref{extreme rays and facets of all the representatives of equivalent V-type domains}.
Hence, a node of $CT_3$ is adjacent to six nodes, as in Figure \ref{Local structure of a topograph CT_3}.
All the adjacent nodes of ${\mathcal V}(\tr{g}^{-1} \Phi_0^2)$ are given as ${\mathcal V}(\tr{g}^{-1} \tau_{ij}^{(3)} \Phi_0^2)$ ($1 \leq i < j \leq 4$), where
${\mathbf e}_4 := -\sum_{i=1}^3 {\mathbf e}_i$ and 
$\tau_{ij}^{(3)}$ is the $3 \times 3$ matrix satisfying:
\begin{eqnarray}
	\tau_{ij}^{(3)} {\mathbf e}_i = - {\mathbf e}_i,\
	\tau_{ij}^{(3)} {\mathbf e}_j = {\mathbf e}_j,\
	k \neq i, j \Longrightarrow \tau_{ij}^{(3)} {\mathbf e}_k = {\mathbf e}_i + {\mathbf e}_k.
\end{eqnarray}

As explained in Figure \ref{Local structure of a topograph CT_3},
$CT_3$ contains two kinds of circuits of length 3 and 6,
which correspond to the following fundamental relations of $GL_3(\IntegerRing)$:
\begin{eqnarray}\label{eq:fundamental relations}
	\tau^{(3)}_{ij} \tau^{(3)}_{km} \tau^{(3)}_{ij} = \sigma_{ik, jm} \in H(A_3),\
	\left( \tau^{(3)}_{ij} \tau^{(3)}_{ik} \tau^{(3)}_{im} \right)^2 = 1,
\end{eqnarray}
where $i, j, k, m$ are integers satisfying $ \{ i, j, k, m \} = \{ 1, 2, 3, 4 \}$,
and $\sigma_{ik, jm} \in H(A_3)$ is the $3 \times 3$ matrix satisfying: 
\begin{eqnarray}
	\sigma_{ik, jm} {\mathbf e}_p = - {\mathbf e}_q
	\text{ for any } (p, q) = (i, k), (k, i), (j, m), (m, j). 
\end{eqnarray}


\begin{figure}[htbp]
\begin{minipage}{\textwidth}
\begin{center}
\scalebox{0.64}{\includegraphics{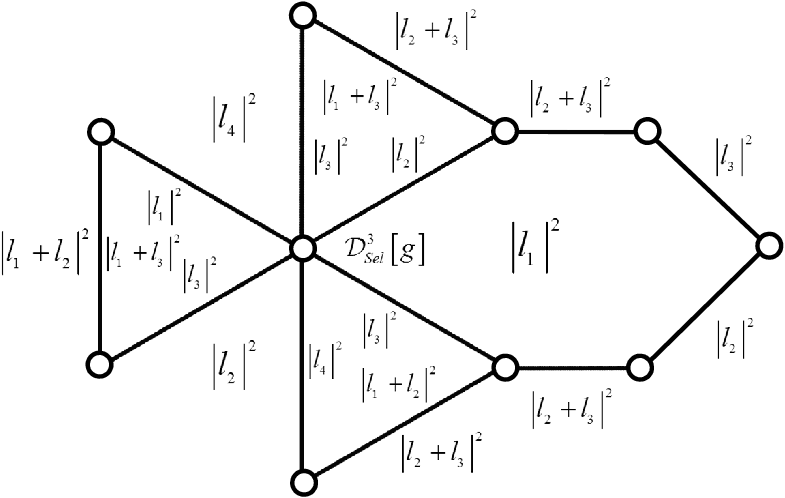}}
\end{center}
\footnotetext{
When we put $(l_1\ l_2\ l_3) := (b_1\ b_2\ b_3) \tr{g}^{-1}$ and $l_4 := - \sum_{i=1}^3 l_i$ for the fixed basis $b_1, b_2, b_3$ and $g \in GL_3(\IntegerRing)$,
each node ${\mathcal D}_{Sel}^3[g] = {\mathcal V}(\tr{g}^{-1} \Phi_0^3)$ is an end point 
of six edges associated with $\{ |l_i|^2, |l_j|^2 \}$ ($1 \leq i < j \leq 4$).
For any $\{ i, j, k, m \} = \{ 1, 2, 3, 4 \}$,
two edges associated with $\{ |l_i|^2, |l_j|^2 \}$, $\{ |l_k|^2, |l_m|^2 \}$
are contained in a circuit of length 3.
On the other hand,
two edges associated with $\{ |l_i|^2, |l_j|^2 \}$, $\{ |l_i|^2, |l_k|^2 \}$
are contained in a circuit of length 6.
}
\end{minipage}
\caption{Local structure of a topograph $CT_{3, (L, (b_1\ b_2\ b_3))}$.}
\label{Local structure of a topograph CT_3}
\end{figure}

\end{example}

\section{Main results for the distribution rules of systematic absences}
\label{Main theorems}

We now discuss on the distribution rules of systematic absences on a topograph.
Those used in our algorithm are described as theorems, whereas other properties important for powder auto-indexing algorithms are mentioned as facts.
As proved in Proposition \ref{prop:basic property of extinction rules1},
there are only a finite number of types of systematic absences,
and they are classified by the types $(G, H)$.

It is not difficult to prove our theorems if $\Gamma_{ext}(G, H)$
is contained in ${\mathcal H}_{G, H}$ in (\ref{eq:definition of {mathcal H}_{G, H}}).
(This always holds for lattices of rank 2.)
Because too many case-by-case considerations are required otherwise,
the most difficult part of the theorems
is confirmed by direct computation, using the International Tables ($N = 2$) and executing a program ($N = 3$).
For confirmation of the case $N = 3$,
we verify that the program outputs exactly the same list as the International Tables.

The most important property of $\Gamma_{ext}(G, H)$ is that $L^*$ is generated by elements of $L^* \setminus \Gamma_{ext}(G, H)$,
under the assumption that $L$ is the period lattice of the periodic function $\wp$.
Therefore, if $L^* \setminus \Gamma_{ext}(G, H)$ holds, the type $(G, H)$ may be regarded to be invalid,
and removed from the following consideration.

\subsection{Cases of rank 2}
\label{Results in N = 2}

According to the International Tables,  
there are 17 wallpaper groups and 72 types of systematic absences, including invalid ones.
By direct computation (or by seeing tables in \cite{Hahn83}), it is verified that $\Gamma_{ext}(G, H) = \Gamma_{ext}(G, \{ id \})$ holds in all the valid cases.

In the following,
we introduce a short proof of Theorem \ref{fact:two-dimensional extinction rules} in the case $H = L$ 
using a topograph.
As a result,  
Theorem \ref{fact:two-dimensional extinction rules} follows for general cases.

\begin{lemma}\label{lem:tau l^* = l^*}
Let $L^*$ be a lattice of rank $2$, and $R \subset O(N)$ be the automorphism group of $L^*$.
If $\tau l_1^* = l_1^*$ holds for some $1 \neq \tau \in R$ and primitive vector of $l_1^* \in L^*$,
either of the following holds:
\begin{enumerate}[(a)]
\item \label{item: l^* cdot l_2^* = 0} 
there exists $l_2^* \in L^*$ such that $l_1^* \cdot l_2^* = 0$ and 
$l_1^*, l_2^*$ is a basis of $L^*$, or

\item \label{item: l^* = l_2^* + tau_1 l_2^*}
$\tau^2 = 1$
and there exists $l_2^* \in L^*$ such that $l_1^* = l_2^* + \tau_1 l_2^*$ and $l_2^*$, $\tau_1 l_2^*$ is a basis of $L^*$.
\end{enumerate}
\end{lemma}
\begin{proof}
Fix a Gram matrix $S_0$ of $L^*$.
Let $l_2^* \in L^*$ be the vector satisfying $\{ l_1^*, l_2^* \} \in P_2(L^*)$, $l_1^* \cdot l_2^* \geq 0$ and $\abs{l_2^*}^2 = \min \{ \abs{l_3^*}^2 : \{ l_1^*, l_3^* \} \in P_2(L^*) \}$.
The edges of $CT_{2, S_0}$ then have the direction presented in Figure \ref{Lattice vector fixed by a point group}.
\begin{figure}[htbp]
\begin{center}
\scalebox{0.5}{\includegraphics{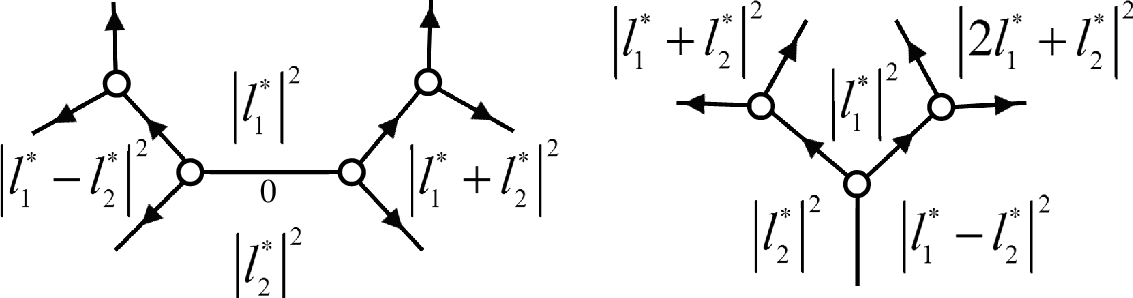}}
\end{center}
\caption{Directions of edges of a topograph associated with $\abs{l_1^*}^2$.}
\label{Lattice vector fixed by a point group}
\end{figure}
The left figure corresponds to the case (\ref{item: l^* cdot l_2^* = 0}).
In the case of the right figure, 
$\tau$ fixes the node surrounded
by $\abs{l_1}^2$, $\abs{l_2}^2$, and $\abs{l_1 - l_2}^2$.
If we recall that the stabilizer of ${\mathcal V}(\Phi_0^3)$ is given by $H(A_3)$,
the action of $\tau$ or $-\tau$ on $\{ l_1^*, -l_2^*, -l_1^* + l_2^* \}$ 
coincides with a permutation of $l_1^*$, $-l_2^*$, $-l_1^* + l_2^*$.
From $\tau \neq 1$ and $\tau l_1^* = l_1^*$, we obtain $\tau l_2^* = l_1^* - l_2^*$.
This is equivalent to the case (\ref{item: l^* = l_2^* + tau_1 l_2^*}).
\end{proof}

\begin{proof}[Proof of Theorem \ref{fact:two-dimensional extinction rules} in the case of $H = L$]
From Corollary \ref{cor:case of L=H}, for any primitive vector $l_1^* \in \Gamma_{ext}(G)$,
there exists $\tau \in R$ with $\tau l_1^* = l_1^*$
such that $\nu_\tau \cdot l_1^* \notin \IntegerRing$.
If (\ref{item: l^* = l_2^* + tau_1 l_2^*}) of Lemma \ref{lem:tau l^* = l^*} holds, then:
\begin{eqnarray}
\nu_\tau \cdot l_1^* = \nu_\tau^{1 + \tau} \cdot l_2^* = \nu_{\tau^2} \cdot l_2^* = \nu_{1} \cdot l_2^* \in \IntegerRing.
\end{eqnarray}
Hence, the statement of the theorem follows.
\end{proof}


\subsection{Cases of rank 3}
\label{Results in N = 3}

For $N = 3$, we consider the following cases separately: 
\begin{enumerate}[(i)]
	\item $\Gamma_{ext}(G, H) \cap P_1(L^*)$ is contained in ${\mathcal H}_{G, H}$
	in (\ref{eq:definition of {mathcal H}_{G, H}});
	\item $\Gamma_{ext}(G, H) \cap P_1(L^*)$ is not contained in ${\mathcal H}_{G, H}$.
\end{enumerate}

Table \ref{Lambda_{ext}(G, H; x) not contained in a union of hyperplanes}
lists all valid types of systematic absences corresponding to the latter case.
For each type,
$\Omega \subset L^* / M L^*$ in Corollary \ref{cor:definition of Omega}
is given in Table \ref{Necessary and sufficient condition for (h, k, l) in Lambda_{ext}(G, H, x)}.

\begin{table}[htbp]
\caption{Types of systematic absences having $\Gamma_{ext}(G, H) \not\subset {\mathcal H}_{G, H}$.}
\label{Lambda_{ext}(G, H; x) not contained in a union of hyperplanes}
\begin{minipage}{\textwidth}
\begin{footnotesize}
\begin{tabular}{p{28mm}p{5mm}p{15mm}p{23mm}p{5mm}p{15mm}p{22mm}p{5mm}p{18mm}}
	\hline
Space group $G$ (No.${}^a$) & $H{}^b$ & $(x, y, z){}^c$ \\
\multicolumn{3}{l}{{\bf A (Face-centered lattice)}}					&	\multicolumn{3}{l}{{\bf B (Body-centered lattice)}}					&	$P\ \bar{4}\ 3\ n$ (218)	&	$C_2$	&	$(x, 0, \frac{1}{2})$	\\
$F\ d\ d\ 2$  (43)	&	$C_2$	&	$(0, 0, z)$	&	$I\ 4_1/a$  (88)	&	$C_i$	&	$(0, \frac{1}{4}, \frac{1}{8})$	&	$P\ m\ \bar{3}\ n$ (223)	&	$C_2$	&	$(\frac{1}{4}, y, y+\frac{1}{2})$	\\
$F\ d\ d\ d$  (70)	&	$C_2$	&	$(x, 0, 0)$	&	$I\ 4_1/a$  (88)	&	$C_i$	&	$(\frac{1}{4}, 0, \frac{3}{8})$	&	$P\ m\ \bar{3}\ n$ (223)	&	$C_{2v}$	&	$(x, \frac{1}{2}, 0)$	\\
$F\ d\ d\ d$  (70)	&	$D_2$	&	$(0, 0, 0)$	&	$I\ 4_1/a\ m\ d$ (141)	&	$C_{2h}$	&	$(0, \frac{1}{4}, \frac{1}{8})$	&	$P\ m\ \bar{3}\ n$ (223)	&	$C_{2v}$	&	$(x, 0, \frac{1}{2})$	\\
$F\ d\ d\ d$  (70)	&	$D_2$	&	$(\frac{1}{2}, \frac{1}{2}, \frac{1}{2})$	&	$I\ 4_1/a\ m\ d$ (141)	&	$C_{2h}$	&	$(0, \frac{1}{4}, \frac{5}{8})$	&	\multicolumn{3}{l}{{\bf F (Body-centered)}}	\\				
$F\ d\ \bar{3}$ (203)	&	$C_2$	&	$(x, 0, 0)$	&	\multicolumn{3}{l}{{\bf C}}					&	$I\ \bar{4}\ 3\ d$ (220)	&	$C_3$	&	$(x, x, x)$	\\
$F\ d\ \bar{3}$ (203)	&	$T$	&	$(0, 0, 0)$	&	$I\ 4_1/a\ m\ d$ (141)	&	$C_2$	&	$(x, \frac{1}{4}, \frac{1}{8})$	&	$I\ a\ \bar{3}\ d$ (230)	&	$C_3$	&	$(x, x, x)$	\\
$F\ d\ \bar{3}$ (203)	&	$T$	&	$(\frac{1}{2}, \frac{1}{2}, \frac{1}{2})$	&	$I\ 4_1/a\ c\ d$ (142)	&	$C_2$	&	$(\frac{1}{4}, y, \frac{1}{8})$	&	{\bf G}	\\				
$F\ 4_1\ 3\ 2$ (210)	&	$C_2$	&	$(x, 0, 0)$	&	\multicolumn{3}{l}{{\bf D}}					&	$P\ 4_2\ 3\ 2$ (208)	&	$D_2$	&	$(\frac{1}{4}, 0, \frac{1}{2})$	\\
$F\ 4_1\ 3\ 2$ (210)	&	$T$	&	$(0, 0, 0)$	&	$P\ 3\ 1\ c$ (159)	&	$C_3$	&	$(\frac{1}{3}, \frac{2}{3}, z)$	&	$P\ 4_2\ 3\ 2$ (208)	&	$D_2$	&	$(\frac{1}{4}, \frac{1}{2}, 0)$	\\
$F\ 4_1\ 3\ 2$ (210)	&	$T$	&	$(\frac{1}{2}, \frac{1}{2}, \frac{1}{2})$	&	$P\ \bar{3}\ 1\ c$ (163)	&	$C_3$	&	$(\frac{1}{3}, \frac{2}{3}, z)$	&	$P\ \bar{4}\ 3\ n$ (218)	&	$S_4$	&	$(\frac{1}{4}, 0, \frac{1}{2})$	\\
$F\ d\ \bar{3}\ m$ (227)	&	$C_{2v}$	&	$(x, 0, 0)$	&	$P\ \bar{3}\ 1\ c$ (163)	&	$D_3$	&	$(\frac{2}{3}, \frac{1}{3}, \frac{1}{4})$	&	$P\ \bar{4}\ 3\ n$ (218)	&	$S_4$	&	$(\frac{1}{4}, \frac{1}{2}, 0)$	\\
$F\ d\ \bar{3}\ m$ (227)	&	$T_d$	&	$(0, 0, 0)$	&	$P\ \bar{3}\ 1\ c$ (163)	&	$D_3$	&	$(\frac{1}{3}, \frac{2}{3}, \frac{1}{4})$	&	$P\ m\ \bar{3}\ n$ (223)	&	$D_{2d}$	&	$(\frac{1}{4}, 0, \frac{1}{2})$	\\
$F\ d\ \bar{3}\ m$ (227)	&	$T_d$	&	$(\frac{1}{2}, \frac{1}{2}, \frac{1}{2})$	&	$P\ 6_3$ (173)	&	$C_3$	&	$(\frac{1}{3}, \frac{2}{3}, z)$	&	$P\ m\ \bar{3}\ n$ (223)	&	$D_{2d}$	&	$(\frac{1}{4}, \frac{1}{2}, 0)$	\\
\multicolumn{3}{l}{{\bf A (Body-centered lattice)}}					&	$P\ 6_3/m$ (176)	&	$C_3$	&	$(\frac{1}{3}, \frac{2}{3}, z)$	&	{\bf H}	\\				
$I\ 4_1$  (80)	&	$C_2$	&	$(0, 0, z)$	&	$P\ 6_3/m$ (176)	&	$C_{3h}$	&	$(\frac{2}{3}, \frac{1}{3}, \frac{1}{4})$	&	$P\ 4_3\ 3\ 2$ (212)	&	$D_3$	&	$(\frac{1}{8}, \frac{1}{8}, \frac{1}{8})$	\\
$I\ 4_1/a$  (88)	&	$C_2$	&	$(0, 0, z)$	&	$P\ 6_3/m$ (176)	&	$C_{3h}$	&	$(\frac{1}{3}, \frac{2}{3}, \frac{1}{4})$	&	$P\ 4_3\ 3\ 2$ (212)	&	$D_3$	&	$(\frac{5}{8}, \frac{5}{8}, \frac{5}{8})$	\\
$I\ 4_1/a$  (88)	&	$S_4$	&	$(0, 0, 0)$	&	$P\ 6_3\ 2\ 2$ (182)	&	$C_3$	&	$(\frac{1}{3}, \frac{2}{3}, z)$	&	$P\ 4_1\ 3\ 2$ (213)	&	$D_3$	&	$(\frac{3}{8}, \frac{3}{8}, \frac{3}{8})$	\\
$I\ 4_1/a$  (88)	&	$S_4$	&	$(0, 0, \frac{1}{2})$	&	$P\ 6_3\ 2\ 2$ (182)	&	$D_3$	&	$(\frac{2}{3}, \frac{1}{3}, \frac{1}{4})$	&	$P\ 4_1\ 3\ 2$ (213)	&	$D_3$	&	$(\frac{7}{8}, \frac{7}{8}, \frac{7}{8})$	\\
$I\ 4_1\ 2\ 2$  (98)	&	$C_2$	&	$(0, 0, z)$	&	$P\ 6_3\ 2\ 2$ (182)	&	$D_3$	&	$(\frac{1}{3}, \frac{2}{3}, \frac{1}{4})$	&	{\bf I}	\\				
$I\ 4_1\ 2\ 2$  (98)	&	$D_2$	&	$(0, 0, 0)$	&	$P\ 6_3\ m\ c$ (186)	&	$C_{3v}$	&	$(\frac{1}{3}, \frac{2}{3}, z)$	&	$I\ 4_1\ 3\ 2$ (214)	&	$D_3$	&	$(\frac{1}{8}, \frac{1}{8}, \frac{1}{8})$	\\
$I\ 4_1\ 2\ 2$  (98)	&	$D_2$	&	$(0, 0, \frac{1}{2})$	&	$P\ \bar{6}\ 2\ c$ (190)	&	$C_3$	&	$(\frac{1}{3}, \frac{2}{3}, z)$	&	$I\ 4_1\ 3\ 2$ (214)	&	$D_3$	&	$(\frac{7}{8}, \frac{7}{8}, \frac{7}{8})$	\\
$I\ 4_1\ m\ d$ (109)	&	$C_{2v}$	&	$(0, 0, z)$	&	$P\ \bar{6}\ 2\ c$ (190)	&	$C_{3h}$	&	$(\frac{2}{3}, \frac{1}{3}, \frac{1}{4})$
	&	{\bf J1}	\\				
$I\ \bar{4}\ 2\ d$ (122)	&	$C_2$	&	$(0, 0, z)$	&	$P\ \bar{6}\ 2\ c$ (190)	&	$C_{3h}$	&	$(\frac{1}{3}, \frac{2}{3}, \frac{1}{4})$
	&	$I\ 4_1\ 3\ 2$ (214)	&	$D_2$	&	$(\frac{1}{8}, 0, \frac{1}{4})$	\\
$I\ \bar{4}\ 2\ d$ (122)	&	$S_4$	&	$(0, 0, 0)$	&	$P\ 6_3/m\ m\ c$ (194)	&	$C_{3v}$	&	$(\frac{1}{3}, \frac{2}{3}, z)$
	&	$I\ 4_1\ 3\ 2$ (214)	&	$D_2$	&	$(\frac{5}{8}, 0, \frac{1}{4})$	\\
$I\ \bar{4}\ 2\ d$ (122)	&	$S_4$	&	$(0, 0, \frac{1}{2})$	&	$P\ 6_3/m\ m\ c$ (194)	&	$D_{3h}$	&	$(\frac{2}{3}, \frac{1}{3}, \frac{1}{4})$
	&	{\bf J2}	\\
$I\ 4_1/a\ m\ d$ (141)	&	$C_2$	&	$(x, x, 0)$	&	$P\ 6_3/m\ m\ c$ (194)	&	$D_{3h}$	&	$(\frac{1}{3}, \frac{2}{3}, \frac{1}{4})$
	&	$I\ \bar{4}\ 3\ d$ (220)	&	$S_4$	&	$(\frac{7}{8}, 0, \frac{1}{4})$	\\
$I\ 4_1/a\ m\ d$ (141)	&	$C_{2v}$	&	$(0, 0, z)$	&	\multicolumn{3}{l}{{\bf E}}				
	&	$I\ \bar{4}\ 3\ d$ (220)	&	$S_4$	&	$(\frac{3}{8}, 0, \frac{1}{4})$	\\
$I\ 4_1/a\ m\ d$ (141)	&	$D_{2d}$	&	$(0, 0, 0)$	&	$P\ 6_2$ (171)	&	$C_2$	&	$(\frac{1}{2}, \frac{1}{2}, z)$
	&	{\bf K}	\\				
$I\ 4_1/a\ m\ d$ (141)	&	$D_{2d}$	&	$(0, 0, \frac{1}{2})$	&	$P\ 6_4$ (172)	&	$C_2$	&	$(\frac{1}{2}, \frac{1}{2}, z)$
	&	$I\ 4_1\ 3\ 2$ (214)	&	$C_2$	&	$(x, 0, \frac{1}{4})$	\\
$I\ 4_1/a\ c\ d$ (142)	&	$C_2$	&	$(x, x, \frac{1}{4})$	&	$P\ 6_2\ 2\ 2$ (180)	&	$C_2$	&	$(\frac{1}{2}, 0, z)$
	&	$I\ \bar{4}\ 3\ d$ (220)	&	$C_2$	&	$(x, 0, \frac{1}{4})$	\\
\multicolumn{3}{l}{{\bf B (Face-centered lattice)}}					&	$P\ 6_2\ 2\ 2$ (180)	&	$D_2$	&	$(\frac{1}{2}, 0, 0)$
	&	$I\ a\ \bar{3}\ d$ (230)	&	$C_2$	&	$(\frac{1}{8}, y, -y+\frac{1}{4})$	\\
$F\ d\ d\ d$  (70)	&	$C_i$	&	$(\frac{1}{8}, \frac{1}{8}, \frac{1}{8})$	&	$P\ 6_2\ 2\ 2$ (180)	&	$D_2$	&	$(\frac{1}{2}, 0, \frac{1}{2})$
	&	{\bf L}	\\				
$F\ d\ d\ d$  (70)	&	$C_i$	&	$(\frac{5}{8}, \frac{5}{8}, \frac{5}{8})$	&	$P\ 6_4\ 2\ 2$ (181)	&	$C_2$	&	$(\frac{1}{2}, 0, z)$
	&	$I\ a\ \bar{3}\ d$ (230)	&	$C_2$	&	$(x, 0, \frac{1}{4})$	\\
$F\ d\ \bar{3}$ (203)	&	$C_{3i}$	&	$(\frac{1}{8}, \frac{1}{8}, \frac{1}{8})$	&	$P\ 6_4\ 2\ 2$ (181)	&	$D_2$	&	$(\frac{1}{2}, 0, 0)$
	&	{\bf M}	\\				
$F\ d\ \bar{3}$ (203)	&	$C_{3i}$	&	$(\frac{5}{8}, \frac{5}{8}, \frac{5}{8})$	&	$P\ 6_4\ 2\ 2$ (181)	&	$D_2$	&	$(\frac{1}{2}, 0, \frac{1}{2})$
	&	$I\ a\ \bar{3}\ d$ (230)	&	$D_2$	&	$(\frac{1}{8}, 0, \frac{1}{4})$	\\
$F\ 4_1\ 3\ 2$ (210)	&	$D_3$	&	$(\frac{1}{8}, \frac{1}{8}, \frac{1}{8})$	&	\multicolumn{3}{l}{{\bf F (Primitive)}}				
	&	$I\ a\ \bar{3}\ d$ (230)	&	$S_4$	&	$(\frac{3}{8}, 0, \frac{1}{4})$	\\
$F\ 4_1\ 3\ 2$ (210)	&	$D_3$	&	$(\frac{5}{8}, \frac{5}{8}, \frac{5}{8})$	&	$P\ 4_2\ 3\ 2$ (208)	&	$C_2$	&	$(x, \frac{1}{2}, 0)$
	&	{\bf N}	\\				
$F\ d\ \bar{3}\ m$ (227)	&	$D_{3d}$	&	$(\frac{1}{8}, \frac{1}{8}, \frac{1}{8})$
	&	$P\ 4_2\ 3\ 2$ (208)	&	$C_2$	&	$(x, 0, \frac{1}{2})$	
	&	$I\ a\ \bar{3}\ d$ (230)	&	$D_3$	&	$(\frac{1}{8}, \frac{1}{8}, \frac{1}{8})$	\\
$F\ d\ \bar{3}\ m$ (227)	&	$D_{3d}$	&	$(\frac{5}{8}, \frac{5}{8}, \frac{5}{8})$
	&	$P\ \bar{4}\ 3\ n$ (218)	&	$C_2$	&	$(x, \frac{1}{2}, 0)$
\end{tabular}
\end{footnotesize}
\footnotetext[1]{Number assigned to every space group in \cite{Hahn83}.}
\footnotetext[2]{The isomorphism class of $H$}
\footnotetext[3]{$(x, y, z) \in \RealField^N/L$
corresponding to all the elements of $(\RealField^3)^H$.
}
\end{minipage}
\end{table}

\begin{table}[htbp]
\caption{Necessary and sufficient condition$^a$ for $\sum_{i=1}^3 u_i l_i^* \notin {\mathcal H}_{G, H}$ to belong to $\Gamma_{ext}(G, H)$.}
\label{Necessary and sufficient condition for (h, k, l) in Lambda_{ext}(G, H, x)}
\begin{minipage}{\textwidth}
\begin{footnotesize}
\begin{center}
\begin{tabular}{cl}
\hline
Type & Necessary and sufficient condition \\
A & \begin{tabular}{lcl}
		face-centered & $\Longrightarrow$ & $u_1 + u_2 + u_3 \equiv 2$ mod 4 \\
		body-centered & $\Longrightarrow$ & $2 u_2 + u_3 \equiv 2$ mod 4
	\end{tabular} \\
B & \begin{tabular}{lcl}
		face-centered & $\Longrightarrow$ & $\{ u_1, u_2, u_3 \} \equiv \{ 0, 0, 2 \}, \{ 0, 2, 2 \}$ mod 4 \\
		body-centered & $\Longrightarrow$ & 
		$\{ u_1, u_2, u_3 \} \equiv \{ 0, 0, 2\}, \{ 2, 2, 2\}$ mod 4 or
		$( u_1, u_2, u_3 ) \equiv (1, 1, 0)$ mod 2
	\end{tabular} \\
C & $(u_1, u_2, u_3) \equiv (1, 1, 0)$ mod 2.  \\
D & $u_1 \equiv u_2$ mod 3, $u_3 \equiv 1$ mod 2. \\
E & $u_1, u_2 \equiv 0$ mod 2, $u_3 \not\equiv 0$ mod 3. \\
F & \begin{tabular}{lcl}
		primitive & $\Longrightarrow$ & $\{ u_1, u_2, u_3 \} \equiv \{ 1, 1, 1 \}$ mod 2 \\
		body-centered & $\Longrightarrow$ & $\{ u_1, u_2, u_3 \} \equiv \{ 0, 0, 2 \}, \{ 2, 2, 2 \}$ mod 4
	\end{tabular} \\
G & $\{ u_1, u_2, u_3 \}
		\begin{cases}
			\equiv \{ 0, 0, 1 \}, \{ 1, 1, 1 \} \text{ mod } 2, \text{ and} \\
			\not\equiv \{ 0, 1, 2 \}, \{ 0, 2, 3 \} \text{ mod } 4.
		\end{cases}$ \\
H & $\{ u_1, u_2, u_3 \}
		\begin{cases}
			\equiv  \{ 0, 0, 0 \}, \{ 0, 0, 1 \} \text{ mod } 2, \text{ and} \\
			\not\equiv \{ 0, 1, 2 \}, \{ 0, 2, 3 \}, \{ 0, 0, 0 \}, \{ 2, 2, 2 \} \text{ mod } 4.
		\end{cases}$ \\
I & $\{ u_1, u_2, u_3 \} \equiv \{ 0, 0, 2 \}, \{ 0, 2, 2 \} \text{ mod } 4$. \\
J1 & $\{ u_1, u_2, u_3 \}
		\begin{cases}
			\not\equiv \{ 0, 0, 0 \}, \{ 0, 2, 2 \}, \{ 1, 1, 2 \}, \{ 1, 2, 3 \}, \{ 2, 3, 3 \} \text{ mod } 4, \text{ and} \\
			\not\equiv \{ 0, 2, 4 \}, \{ 0, 4, 6 \}, \{ 0, 1, 1 \}, \{ 1, 1, 4 \}, \{ 0, 3, 3 \},
			\{ 3, 3, 4 \}, \{ 0, 1, 7 \}, \{ 1, 4, 7 \}, \{ 0, 3, 5 \}, \{ 3, 4, 5 \} \text{ mod } 8.
		\end{cases}$ \\
J2 & $\{ u_1, u_2, u_3 \}
		\begin{cases}
			\not\equiv \{ 0, 0, 0 \}, \{ 0, 2, 2 \}, \{ 1, 1, 2 \}, \{ 1, 2, 3 \}, \{ 2, 3, 3 \} \text{ mod } 4, \text{ and} \\
			\not\equiv \{ 0, 2, 4 \}, \{ 0, 4, 6 \}, \{ 0, 1, 3 \}, \{ 1, 3, 4 \}, \{ 0, 1, 5 \},
			\{ 1, 4, 5 \}, \{ 0, 3, 7 \}, \{ 3, 4, 7 \}, \{ 0, 5, 7 \}, \{ 4, 5, 7 \} \text{ mod } 8.
		\end{cases}$ \\
K & $\{ u_1, u_2, u_3 \} \equiv \{ 2, 2, 2 \} \text{ mod } 4$. \\
L & $\{ u_1, u_2, u_3 \} \equiv \{ 0, 1, 1 \}, \{ 0, 1, 3 \}, \{ 0, 3, 3 \}, \{ 2, 2, 2 \} \text{ mod } 4$. \\
M & $\{ u_1, u_2, u_3 \} 
		\begin{cases}
			\not\equiv \{ 0, 0, 0 \}, \{ 0, 2, 2 \}, \{ 1, 1, 2 \}, \{ 1, 2, 3 \}, \{ 2, 3, 3 \} \text{ mod } 4, \text{ and} \\
			\not\equiv \{ 0, 2, 4 \}, \{ 0, 4, 6 \} \text{ mod } 8.
		\end{cases}$ \\
N & $\{ u_1, u_2, u_3 \}
			\not\equiv \{ 0, 0, 0 \}, \{ 1, 1, 2 \}, \{ 1, 2, 3 \}, \{ 2, 3, 3 \} \text{ mod } 4$.
\end{tabular}
\end{center}
\end{footnotesize}
\footnotetext[1]{$\langle \l_1^*, l_2^*, l_3^* \rangle$ is the reciprocal basis of 
the basis $l_1, l_2, l_3$ of $\RealField^3$ taken as in the footnote of Table \ref{Lambda_{ext}(G, H; x) not contained in a union of hyperplanes}.
}
\end{minipage}
\end{table}

First, we shall explain some more details of Ito's method, and determine why it does not work appropriately
for some types of systematic absences. 
It was Ito \cite{Ito49}
who first proposed using the parallelogram law for powder auto-indexing;
if the observed $q$-values $q_1, q_2, q_3, q_4 \in \Lambda^{obs}$ satisfy $2(q_1 + q_2) = q_3 + q_4$,
the method assumes that 
there exist $l_1^*, l_2^* \in L^*$ satisfying the following:
\begin{eqnarray}\label{eq:definition of zone}
	q_1 = \abs{l_1^*}^2,\ 
	q_2 = \abs{l_2^*}^2,\
	q_3 = \abs{l_1^*+l_2^*}^2,\ 
	q_4 = \abs{l_1^*-l_2^*}^2.
\end{eqnarray}
If (\ref{eq:definition of zone}) is true, the Gram matrix of the sublattice expanded by $l_1^*, l_2^*$ is determined.
In order to obtain candidate solutions for $L^*$,
it is necessary to construct a $3 \times 3$ Gram matrix
from combinations of these sublattices.
In order to simplify the combination procedure,
it is very desirable that $\{ l_1^*, l_2^* \}$ in (\ref{eq:definition of zone}) is a primitive set of $L^*$.
Otherwise, it is necessary 
to check whether $L^*$ has rank 3 sublattices that are more plausible as a solution once $L^*$ is obtained.  
The program of Visser \cite{Visser69} which adopted Ito's method also implicitly requires $\{ l_1^*, l_2^* \} \in P_2(L^*)$ \cite{Wolff58}.

However, according to the following fact, in some types of systematic absences,
$\{ l_1^*, l_2^* \} \subset L^*$ is never a primitive set of $L^*$,
if $\{ l_1^*, l_2^* \}$ satisfies (\ref{eq:definition of zone}). 
\begin{fact}\label{fact:no Ito's equation}
	If $(G, H)$ is of the category B or N,
	there exists no primitive set $\{ l_1^*, l_2^* \}$ of $L^*$ 
	such that none of $l_1^*$, $l_2^*$, $l_1^* \pm l_2^*$ belong to $\Gamma_{ext}(G, H)$.
\end{fact}

To remove adverse effects of systematic absences from Ito's algorithm, it has been proposed that a formula other than the parallelogram law 
should be used \cite{Wolff57}:
\begin{eqnarray}
	\abs{l_1^* + m l_2^*}^2 - \abs{l_1^* - m l_2^*}^2 = m (\abs{l_1^* + l_2^*}^2 - \abs{l_1^* - l_2^*}^2).
\end{eqnarray}
However, it has not been ascertained whether the equation works appropriately for all types of systematic absences.
As another example, the following was proposed to obtain a rank 3 solution directly.
\begin{eqnarray} 
	\abs{l_1^*}^2 + \abs{l_2^*}^2 + \abs{l_3^*}^2 + \abs{l_1^* + l_2^* + l_3^*}^2
	=
	\abs{l_1^* + l_2^*}^2 + \abs{l_1^* + l_3^*}^2 + \abs{l_2^* + l_3^*}^2.
\end{eqnarray} 

The above equation has a similar property to the parallelogram law:
\begin{fact}
	If $(G, H)$ is of the category B, C, F, G, or N,
	there exists no basis $\langle l_1^*, l_2^*, l_3^* \rangle$ of $L^*$ 
	such that none of $l_1^*$, $l_2^*$, $l_3^*$, $l_1^*+l_2^*$, $l_1^*+l_3^*$, $l_2^*+l_3^*$, 
	$l_1^*+l_2^*+l_3^*$ belong to $\Gamma_{ext}(G, H)$.
\end{fact}

The following two theorems are intended to resolve any this kind of problems caused by systematic absences.
There, it is proved that 
there exist ``infinitely many'' elements of $P_2(L^*)$ or $P_3(L^*)$ satisfying some common properties,
and connected subgraphs of a topograph containing infinitely many nodes are formed from the lengths of lattice vectors belonging to such elements.
Existence of infinitely many elements is required not to fail to gain the true solution,
because only a finite subset $\Lambda^{obs} \subset \Gamma_{ext}(G, H)$ is available in computation.
From the latter claim, it is guaranteed that rather large subgraph is formed for the true solution,
even from finite elements belonging to $\Lambda^{obs}$.
This is useful to select better candidate solutions and reduce computation time.

We shall explain the meaning of Theorem \ref{thm:distribution rules on CT_{2, S_2}} before describing the statements;
in short, it claims 
that the equation $3 \abs{ l_1^* }^2 + \abs{ l_1^* + 2 l_2^* }^2 = \abs{ 2 l_1^* + l_2^* }^2 + 3 \abs{ l_2^* }^2$
works appropriately regardless of the type of systematic absences. 
When $(L^*, B^*)$ represents the reciprocal lattice of the lattice to obtain as a solution and a matrix
whose columns are a basis of $L^*$, 
this new equation corresponds to the subgraph of $CT_{2, (L^*, B^*)}$ consisting of 
three nodes and two edges in the left-hand of Figure \ref{Edges associated with abs{l_1^* + l_2^*}^2}.
A subgraph of $CT_{2, (L^*, B^*)}$ containing infinitely many nodes as the right-hand of Figure \ref{Edges associated with abs{l_1^* + l_2^*}^2} is formed by linking such a subgraph obtained from 
$q_r, q_t, q_s, q_u \in \Lambda_{ext}(G, H)$ satisfying $3 q_r + q_t = 3 q_s + q_u$.

\begin{figure}
\begin{center}
\scalebox{0.64}{\includegraphics{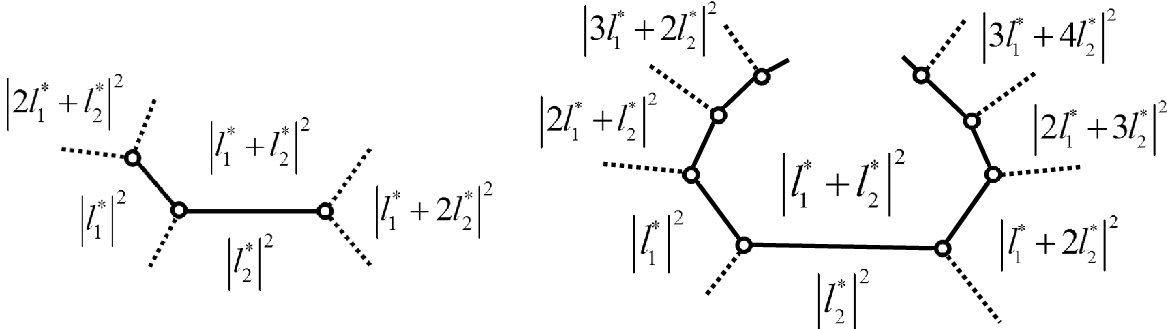}}
\end{center}
\caption{A subgraph of a topograph corresponding to the equation $3 \abs{ l_1^* }^2 + \abs{ l_1^* + 2 l_2^* }^2 = \abs{ 2 l_1^* + l_2^* }^2 + 3 \abs{ l_2^* }^2$ and a chain formed from such subgraphs.}
\label{Edges associated with abs{l_1^* + l_2^*}^2}
\end{figure}

\begin{theorem}\label{thm:distribution rules on CT_{2, S_2}}
	For any crystallographic group $G$ with the translation group $L$ and the type of systematic absences $(G, H)$, 
	the following subset of $\tilde{P}_2(L^*)$ includes infinitely many elements.
	\begin{eqnarray}
		\tilde{P}_2(L^*)
		:= \left\{ \{ l_1^*, l_2^* \} \in P_2(L^*) : 
				m l_1^* + (m-1) l_2^* \notin \Gamma_{ext}(G, H) \text{ for any } m \in \IntegerRing 
		\right\}. \hspace{10mm}
	\end{eqnarray}
	More precisely, 
	for any open convex cone $U \subset \RealField^N$ satisfying $U \cap {\mathcal H}_{G, H} = \emptyset$,
	the following holds:
	\begin{enumerate}[(1)]
		\item \label{item:infinitely many { l_1^*, l_2^* } in P_2(L^*) 3} 
			There exists $\{ l_1^*, l_2^* \} \in P_2(L^*)$
			such that $\{ m l_1^* + n l_2^* : m, n \in \IntegerRing_{\geq 0} \} \subset U$.
			If $\Gamma_{ext}(G, H) \subset {\mathcal H}_{G, H}$, 
			such $\{ l_1^*, l_2^* \} \in P_2(L^*)$ belongs to $\tilde{P}_2(L^*)$.

		\item \label{item:infinitely many { l_1^*, l_2^* } in P_2(L^*) 2}
			Assume that $(G, H)$ is one of the types
			in Table \ref{Lambda_{ext}(G, H; x) not contained in a union of hyperplanes},
			and $\{ l_1^*, l_2^* \} \in P_2(L^*)$ satisfy 
			\begin{enumerate}[(a)]
				\item $\{ l_1^*, l_2^* \} \subset U$, and
				\item $l_1^*$, $l_2^*$, $2l_1^*+ l_2^*$, $l_1^*+ 2l_2^*$, $2 l_1^* + 3 l_2^*$, $3 l_1^*+ 2 l_2^*$, $3 l_1^*+ 4 l_2^*$ 
(\IE all the vectors other than 
			$l_1^* + l_2^*$ in the right-hand of Figure \ref{Edges associated with abs{l_1^* + l_2^*}^2}) do not belong to $\Gamma_{ext}(G, H)$.
			\end{enumerate}
			Then, $\{ l_1^*, l_2^* \}$ is an element of $\tilde{P}_2(L^*)$.
			Furthermore, $\{ l_1^*, l_2^* \} \in P_2(L^*)$ satisfying the two properties exists.
	\end{enumerate}
Furthermore, there exist infinitely many 2-rank sublattices $L_2^* \subset L^*$ 
such that $L_2^*$ is expanded by a primitive set in $\tilde{P}_2(L^*)$.
\end{theorem}
\begin{proof}
If $\tilde{P}_2(L^*)$ is not empty, it includes infinitely many elements 
because, for any $n \in \IntegerRing$,
$\{ -n l_1^* - (n-1) l_2^*, (n+1) l_1^* + n l_2^* \} \in P_2(L^*)$
belongs to $\tilde{P}_2(L^*)$ if and only if $\{ l_1^*, l_2^* \} \in \tilde{P}_2(L^*)$.
Hence, the first statement 
 follows immediately from
(\ref{item:infinitely many { l_1^*, l_2^* } in P_2(L^*) 3}) and 
(\ref{item:infinitely many { l_1^*, l_2^* } in P_2(L^*) 2}).

(\ref{item:infinitely many { l_1^*, l_2^* } in P_2(L^*) 3}) is almost straightforward.
The existence of $\{ l_1^*, l_2^*, l_3^* \} \in P_2(L^*)$
satisfying $\{ l_1^*, l_2^*, l_3^* \} \subset U$
is proved by Lemma \ref{lem:open cone and lattice}.
Thus, $\sum_{i=1}^3 m_i l_i^* \in U$ holds for any integers $m_i \geq 0$.
In this case, 
$\{ l_1^*, l_2^* + k l_3^* \} \in P_2(L^*)$ is a subset of $U$ for any integer $k \geq 0$,
and their expanding sublattices are different from each other.
From the assumption $\Gamma_{ext}(G, H) \subset {\mathcal H}_{G, H}$,
$\pm ( m_1 l_1^* + m_2 (l_2^* + k l_3^*) ) \notin \Gamma_{ext}(G, H)$ follows.
As a result, $\{ l_1^*, l_2^* + k l_3^* \} \in \tilde{P}_2(L^*)$ is obtained.
(Here, $l^* \in \Gamma_{ext} \Leftrightarrow -l^* \in \Gamma_{ext}$ was used.)

In order to prove (\ref{item:infinitely many { l_1^*, l_2^* } in P_2(L^*) 2}),
let $M$ be the order of $R_G$, 
and $\Omega$ be the following set defined for the type $(G, H)$ in Corollary \ref{cor:definition of Omega}:
\begin{eqnarray}
	\Omega := \left\{ l^* + M L^* : l^* \in \Gamma_{ext}(G, H) \setminus {\mathcal H}_{G, H} \right\}.
\end{eqnarray}

The following set is defined using $\Omega$:
\begin{eqnarray}
		P_{2, M}(L^*) &:=& \left\{ \{ l_1^* + M L^*, l_2^* + M L^* \} : \{ l_1^*, l_2^* \} \in P_2(L^*) \right\}, \label{eq:definition of P_{2, M}(L^*)} \\
	\tilde{P}_{2, M}(L^*) &:=& \left\{ \{ l_1^* + M L^*, l_2^* + M L^* \} \in P_{2, M}(L^*): 
			m l_1^* + (m-1) l_2^* + M L^* \notin \Omega \text{ for any } m \in \IntegerRing
	\right\}. \hspace{10mm}
\end{eqnarray}

By direct computation,
it is verified that $\tilde{P}_{2, M}(L^*)$ equals the following set:
\begin{eqnarray}
	\left\{ \{ l_1^* + M L^*, l_2^* + M L^* \} \in P_{2, M}(L^*) : 
		\begin{matrix}
		m l_1^* + (m-1) l_2^* + M L^* \notin \Omega\ (k \leq m \leq k+6) \text{ for some } k \in \IntegerRing
		\end{matrix}
	\right\}.
\end{eqnarray}
Furthermore, $\tilde{P}_{2, M}(L^*) \neq \emptyset$ holds for any type of systematic absences presented in Table \ref{Lambda_{ext}(G, H; x) not contained in a union of hyperplanes} (see Table \ref{Cardinality of Gamma_{ext}(G, H, x) in L^*}).
Consequently, some 
$\{ l_1^*, l_2^*, l_3^* \} \in P_3(L^*)$ 
satisfies the assumptions (a) and (b), as a result of Lemma \ref{lem:open cone and lattice}.
In this case, 
$\{ l_1^*, l_2^* + k M l_3^* \} \in \tilde{P}_{2}(L^*)$ holds for any $k \geq 0$,
because of 
$\{ l_1^* \text{ mod } M, l_2^*  \text{ mod } M \} \in \tilde{P}_{2, M}(L^*)$ and $U \cap {\mathcal H}_{G, H} = \emptyset$.
The lattices expanded 
by $\{ l_1^*, l_2^* + k M l_3^* \}$ are different, depending on $k \geq 0$.
Hence all the statements were shown.
\end{proof}

The following lemma is used in the proof of Theorem \ref{thm:distribution rules on CT_{2, S_2}}.  
Although it is straightforward, we provide a proof.
\begin{lemma}\label{lem:open cone and lattice}
	Let $L \subset \RealField^N$ be a lattice of rank $N$, and $1 \leq m \leq N$ be an integer. 
	Then, any open cone $U \subset \RealField^N$ 
	contains a primitive set $\{ l_1, \ldots, l_m \} \in P_m(L)$.
	Furthermore, when $M > 0$ is a positive integer and $\{ k_1, \ldots, k_m \} \in P_m(L)$,
	$U$ contains infinitely many $\{ l_1, \ldots, l_m \} \in P_m(L)$
	satisfying $l_i - k_i \in M L$ for any $1 \leq i \leq m$.
\end{lemma}
\begin{proof}
	We prove the first statement by induction.
	Since $\{ a l : 0 \neq a \in \RationalField, l \in L \} = \{ a l : a \in \RationalField, l \in P_1(L) \}$ is dense in $\RealField^N$,
	there exist $0 \neq a \in \RationalField$ and $l \in P_1(L)$ such that $a l \in U$.
	Hence, $l \in U$ is obtained.
	Next suppose that $m < N$ and there is $T \in P_m$ contained in $U$.
	Then there exists $l \in L$ such that $T \cup \{ l \} \in P_{m+1}$.
	For any arbitrarily fixed $l_2 \in T$,
	there is $\epsilon > 0$ such that $U_2 := \{ x \in \RealField^N : (1 - \epsilon) \abs{x}^2 \abs{l_2}^2 \leq (x \cdot l_2)^2 \}$ is contained in $U$.  In this case, $l + s l_2 \in U_2$ holds for sufficiently large integer $s > 0$.
	As a result, $T \cup \{ l + s l_2 \}$ is a subset of $U$ and primitive.
	In order to prove the second statement, it is sufficient if
	some $\{ l_1, \ldots, l_m \} \in P_m(L)$ satisfies the desired property.
	We fix a basis $l_1, \ldots, l_N \in U$ of $L$
	and $g \in GL(\IntegerRing)$
	satisfying $k_i = g l_i$ for any $1 \leq i \leq m$.
	When the subgroup of $GL(\IntegerRing)$ with positive entries is denoted by $GL_{+}(\IntegerRing)$, 
	the natural map $GL_{+}(\IntegerRing) \longrightarrow G := \{ g \in GL(\IntegerRing / M \IntegerRing) : \det g = \pm 1 \text{ mod } M \}$ is an epimorphism.
	Let $g_0 \in GL_{+}(\IntegerRing)$ be an element belonging to the inverse image of $g$ mod $M$.
	Then $g_0 l_1, \ldots, g_0 l_N$ are all contained in $U$ and satisfy $g_0 l_i - k_i \in M L$ ($1 \leq i \leq m$).
\end{proof}

\begin{table}[htbp]
\caption{Density of $L^* \setminus \Gamma_{ext}(G, H)$ in $L^*$.}
\label{Cardinality of Gamma_{ext}(G, H, x) in L^*}
\begin{minipage}{\textwidth}
\begin{center}
\begin{tabular}{lll}
Type & $\tilde{P}_{2, M}(L^*)$ in $P_{2, M}(L^*)$ & $\tilde{P}_{3, M}(L^*)$ in $P_{3, M}(L^*)$ \\
\hline
A	 & 	0.321 	 & 	0.286 	\\
B	 & 	0.286 	 & 	0.143 	\\
C	 & 	0.476 	 & 	0.190 	\\
D	 & 	0.341 	 & 	0.209 	\\
E	 & 	0.736 	 & 	0.604 	\\
F	 & 	0.714 	 & 	0.571 	\\
G	 & 	0.214 	 & 	0.027 	\\
H	 & 	0.429 	 & 	0.058 	\\
I	 & 	0.714 	 & 	0.571 	\\
J1 \& J2	 & 	0.071 	 & 	0.022 	\\
K	 & 	0.857 	 & 	0.786 	\\
L	 & 	0.107 	 & 	0.004 	\\
M	 & 	0.036 	 & 	0.004 	\\
N	 & 	0.036 	 & 	0.004 	
\end{tabular}
\end{center}
\footnotetext[1]{$M$ is the cardinality of $R_G$.  The densities are computed by dividing 
the cardinality of $\tilde{P}_{i, M}(L^*)$ by that of $P_{i, M}(L^*)$ ($i = 2, 3$),
where $P_{i, M}(L^*)$ is defined by (\ref{eq:definition of P_{2, M}(L^*)}) and (\ref{eq:definition of P_{3, M}(L^*)}).
}
\footnotetext[2]{
The densities of G, H, J, L, M, N, which consist of only primitive and body-centered cubic lattices, are rather small.
We were not able to find another combination of $q$-values
with larger densities.
As a practical measure against small densities,
the parallelogram law is used in the enumeration procedure of Table \ref{Enumeration of the three dimensional lattices enumTwoDimLattices},
in addition to $3 \abs{ l_1^* }^2 + \abs{ l_1^* + 2 l_2^* }^2 = \abs{ 2 l_1^* + l_2^* }^2 + 3 \abs{ l_2^* }^2$.
(This is effective except for the category N, owing to Fact \ref{fact:no Ito's equation}.)
Furthermore, the condition (\ref{item:long subgraph}) of Theorem \ref{thm:distribution rules on CT_{3, S_3}}
is not required to hold for infinitely many $m$ in the actual algorithm, and subgraphs with relatively many edges are given priority.
}
\end{minipage}
\end{table}

Theorem \ref{thm:distribution rules on CT_{3, S_3}} claims that 
a solution of rank 3 is obtained from the combination of 
two subgraphs of a topograph for lattices expanded by $l_1^*$ and $l_i^*$ ($i = 2, 3$) as in the right-hand of Figure \ref{Edges associated with abs{l_1^* + l_2^*}^2}
and $\abs{\pm l_1^*+l_2^*+l_3^*}^2 \in \Lambda_{ext}(G, H)$.
By using such subgraphs with infinitely many edges for the purpose,
the ``sort criterion for zones'' proposed in Section \ref{Speed-up using topographs} 
is provided a theoretical foundation.

\begin{theorem}\label{thm:distribution rules on CT_{3, S_3}}
	Using the same notation as Theorem \ref{thm:distribution rules on CT_{2, S_2}}, 
	let $\tilde{P}_3(L^*)$ be the set of $\{ l_1^*, l_2^*, l_3^* \} \in P_3(L^*)$ 
			satisfying (a) and (b) with both $i= 2, 3$.
			\begin{enumerate}[(a)]
				\item $\pm l_1^* + l_2^* + l_3^* \notin \Gamma_{ext}(G, H)$.
				\item \label{item:long subgraph} 
					$m l_1^* + (m-1)(-l_1^* + l_i^*) \notin \Gamma_{ext}(G, H)$ for any $m \in \IntegerRing$,
					or 
					$m l_i^* + (m-1)(l_1^* - l_i^*) \notin \Gamma_{ext}(G, H)$ for any $m \in \IntegerRing_{\geq 0}$.
			\end{enumerate}
	$\tilde{P}_3(L^*)$ then contains infinitely many elements, regardless of the type of systematic absences.
\end{theorem}
\begin{remark}
	This theorem may be considered to refer the vectors associated with the edges
	in the circuit of length 6 in Figure \ref{Local structure of a topograph CT_3};
	the 6 edges are associated with one of the four vectors in the following parallelogram law:
	\begin{enumerate}[(a)]
		\item $2 ( \abs{l_1}^2 + \abs{l_2 + l_3}^2 ) = \abs{l_1 + l_2 + l_3}^2 + \abs{-l_1 + l_2 + l_3}^2$,
		\item $2 ( \abs{l_1}^2 + \abs{l_i}^2 ) = \abs{l_1 + l_i}^2 + \abs{-l_1 + l_i}^2$ ($i=2$ or $3$).
	\end{enumerate}
\end{remark}

\begin{figure}
\begin{minipage}{\textwidth}
\begin{center}
\scalebox{0.64}{\includegraphics{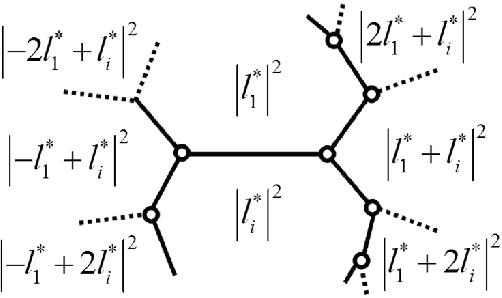}}
\end{center}
\end{minipage}
\caption{Connected subgraph of a topograph consisting of edges associated with 
$\left\{ \{ l_i, m l_1^* + (m-1)(-l_1 + l_i^*) \} : m \in \IntegerRing \right\}$ or
$\left\{ \{ l_1, m l_i^* + (m-1)(l_1^* - l_i^*) \} : m \in \IntegerRing_{\geq 0} \right\}$.
}
\label{fig:GraphWithInfiniteNodes}
\end{figure}

\begin{proof}
By Lemma \ref{lem:open cone and lattice},
any open convex cone $U \subset \RealField^N \setminus {\mathcal H}_{G, H}$
contains some $\{ l_1^*, -l_1^*+l_2^*, -l_1^* + l_3^* \} \in P_3(L^*)$.
In this case, $U$ also includes $\pm l_1^* + l_2^* + l_3^*$, $m l_1^* + (m-1)(-l_1 + l_i^*)$, $m l_1^* + (m+1)(-l_1 + l_i^*)$ and $m l_i^* + (m-1)(l_1^* - l_i^*) = m l_1^* + (-l_1 + l_i^*)$ for any $m \in \IntegerRing_{\geq 0}$.
Consequently, 
if $\Gamma_{ext}(G, H)$ is contained in ${\mathcal H}_{G, H}$,
the statement is obtained immediately.

If $(G, H)$ is one of the types
in Table \ref{Lambda_{ext}(G, H; x) not contained in a union of hyperplanes},
let $M$ be the order of $R_G$, and $\Omega$ and $\tilde{P}_{2, M}(L^*)$
be the sets defined in the proof of Theorem \ref{thm:distribution rules on CT_{2, S_2}}.
Furthermore, we define
\begin{small}
\begin{eqnarray}
	P_{3, M}(L^*) &:=& 
		\left\{ (l_1^* + M L^*, l_2^* + M L^*, l_3^* + M L^*) : \{ l_1^*, l_2^*, l_3^* \} \in P_3(L^*) \right\}, \label{eq:definition of P_{3, M}(L^*)} \\
	\tilde{P}_{3, M}(L^*) &:=& \left\{ ( l_1^* + M L^*, l_2^* + M L^*, l_3^* + M L^* ) \in P_{3, M}(L^*)
		: \begin{matrix}
			\pm l_1^* + l_2^* + l_3^* + M L^* \in \Omega, \\
			\{ l_1^* + M L^*, - l_1^* - l_i^* + M L^* \} \\
			\text{ or } \{ l_i^* + M L^*, - l_1^* - l_i^* + M L^* \} \\
			\text{belongs to } \tilde{P}_{2, M}(L^*) \text{ for both } i = 2, 3
		  \end{matrix}
	\right\}. \hspace{10mm}
\end{eqnarray}
\end{small}
By direct calculation, it is verified that
$\tilde{P}_{3, M}(L^*) \neq \emptyset$, regardless of the type of systematic absences (See Table \ref{Cardinality of Gamma_{ext}(G, H, x) in L^*}).
From Lemma \ref{lem:open cone and lattice}, there exist infinitely many $\{ l_1^*, l_2^*, l_3^* \} \in P_{3}(L^*)$ such that $\{ l_1^* + M L^*, l_2^* + M L^*, l_3^* + M L^* \} \in \tilde{P}_{3, M}(L^*)$ holds
and  $\{ l_1^*, -l_1^* + l_2^*, -l_1^* + l_3^* \}$ is a subset of $C$.
In this case, $\{ l_1^*, l_2^*, l_3^* \} \in \tilde{P}_{3}(L^*)$ also holds.
\end{proof}

\section{Algorithm for powder auto-indexing}
\label{Powder indexing algorithm}

In order to elaborate our new algorithm,
we first define a data structure that is assumed to be implemented in the program.
The set $\Lambda^{obs}$ of observed lattice vector lengths 
is input as 
an array $\langle (q_i[0], q_i[1]) : 1 \leq i \leq M \rangle$ 
consisting of pairs of a $q$-value $q_i[0]$ and its estimated error $q_i[1] = {\rm Err}[q_i[0]]$.
In our algorithm, every candidate for a Gram matrix of $L^*$ 
is provided as a matrix with entries $\sum_{i=1}^M n_i q_i[0]$ having coefficients $n_i \in \frac{1}{4} \IntegerRing$.
At various stages of powder auto-indexing,
the propagated errors of the entries are useful for making statistical judgments
and strengthening the algorithm against observation errors in the $q$-values.
For this purpose, a data structure for formal sums $\sum_{i=1}^{N_{peak}} c_i q_i$ of elements of $\Lambda^{obs}$
with coefficients $c_i \in \frac{1}{4 N} \IntegerRing$
is implemented in Conograph ($N$ is a fixed positive integer). 
The data structure is equipped with the order $<$ and functions ${\rm getTerms}$, ${\rm Val}$, ${\rm Err}$ as follows:
\begin{eqnarray}
 \sum_{i=1}^M a_i q_i < \sum_{i=1}^M b_i q_i & \underset{def}{\Leftrightarrow} & \sum_{i=1}^M a_i q_i[0] < \sum_{i=1}^M b_i q_i[0], \label{eq:order of q} \\
{\rm getTerms}\left( \sum_{i=1}^M n_i q_i \right) &:=& \{ q_i : 1 \leq i \leq M, n_i \neq 0 \}, \label{eq:definition of getTerm} \\
{\rm Val}\left( \sum_{i=1}^M n_i q_i \right) &:=& \sum_{i=1}^M n_i q_i[0], \label{eq:definition of val} \\
{\rm Err}\left( \sum_{i=1}^M n_i q_i \right) &:=& \left( \sum_{i=1}^M n_i^2 (q_i[1])^2 \right)^{1/2}. \label{eq:definition of Err}
\end{eqnarray}

In particular, 
if ${\rm Val}$ and ${\rm Err}$ are called with the argument $\sum_{i=1}^M n_i q_i$,
they return the value and the propagated error of $\sum_{i=1}^M n_i q_i[0]$, respectively.

\subsection{Algorithm for $N = 2$}
\label{Lattice determination from weighted theta series in N = 2}

In this case, according to Theorem \ref{fact:two-dimensional extinction rules}, the method using the parallelogram law 
works sufficiently.

\begin{small}
\begin{table}
	\caption{Procedures for enumeration of four $q$-values satisfying the parallelogram law.}
	\label{Enumeration of the two dimensional lattices enumTwoDimLattices}
	\begin{minipage}{\textwidth}
	\begin{tabular}{lccp{125mm}}
	\hline
	\multicolumn{4}{c}{{\bf void enumerateItoSolutions($\Lambda^{obs}, c, Ans$)}} \\
	(Input) & $\Lambda^{obs}$ & : & array of $N_{peak}$ pairs of a $q$-value $q_i[0]$ and its approximated error $q_i[1]$.\\
	 & $c > 0$ & :  & parameter setting error tolerance level. \\
	(Output) & $Ans$ & :  & array of a sequence $( \{ q_r, q_s\}, \{ q_t, q_u \} )$, where $q_r, q_s, q_t, q_u$ are elements of $\Lambda^{obs}$ satisfying 
			\begin{eqnarray*}
				\hspace{-20mm}
				\begin{cases}
				\abs{ {\rm Val}( 2q_r + 2q_s - q_t - q_u) } \leq c\ {\rm min}\{ 2 {\rm Err}(q_r + q_s), {\rm Err}(q_t + q_u) \} & \text{(Parallelogram law)}, \\
				\left( \sqrt{q_r[0]}-\sqrt{q_s[0]} \right)^2 \leq q_t[0], q_u[0] \leq \left( \sqrt{q_r[0]}+\sqrt{q_s[0]} \right)^2 & \text{(Positive-definite condition)}.
				\end{cases}
			\end{eqnarray*}
	\end{tabular}
	\begin{tabular}{p{5mm}p{10mm}p{125mm}}
 	1: & (Start)  & Set a sorted sequence $S := \langle q_i + q_j : 1 \leq i \leq j \leq N_{peak} \rangle$ of formal sums. \\
 	2: & & for $i := 1$ to $\frac{1}{2} N_{peak}(N_{peak}+1)$ do \\
 	3: & & \hspace{5mm}	Let $1 \leq J_{min}, J_{max} \leq N_{peak}$ be integers satisfying \\
 	4: & & \hspace{5mm}	$J_{min} \leq j \leq J_{max} \Longleftrightarrow \abs{ {\rm Val}(2 S[i] - S[j]) } \leq 2 c {\rm Err}(S[i])$. \\
 	5:  & & \hspace{5mm} for $j := J_{min}$ to $J_{max}$ do \\
 	6: & & \hspace{10mm}	if $\abs{ {\rm Val}(2 S[i] - S[j]) } \leq c {\rm Err}(S[j])$ then \\
 	7: & & \hspace{15mm}	$\{ q_r, q_s \} := {\rm getTerms}(S[i])$, \\ 
 	8: & & \hspace{15mm}	$\{ q_t, q_u \} := {\rm getTerms}(S[j])$. \\ 
 	9: & & \hspace{15mm}	if $(\sqrt{q_r[0]}-\sqrt{q_s[0]})^2 \leq q_t[0], q_u[0] \leq (\sqrt{q_r[0]}+\sqrt{q_s[0]})^2$ then \\
 	10: & & \hspace{20mm}		insert $( \{ q_r, q_s \}, \{ q_t, q_u \} )$ in $Ans$. \\
 	11: & & \hspace{15mm}	end if \\
 	12: & & \hspace{10mm}	end if \\
 	13: & & \hspace{5mm}	end for \\
 	14: & & end for \\
	\end{tabular}
	\end{minipage}
\end{table}
\end{small}

On output of the procedure in Table \ref{Enumeration of the two dimensional lattices enumTwoDimLattices}, each entry $( \{ q_r, q_s \}, \{ q_t, q_u \} )$ in $Ans$ satisfies the parallelogram law $2(q_r + q_s) = q_t + q_u$, and 
corresponds to the $2 \times 2$ positive-definite symmetric matrix:
\begin{eqnarray}\label{eq:Gram matrix of zone}
	( \{ q_r, q_s \}, \{ q_t, q_u \} ) \mapsto
	\begin{pmatrix}
		{\rm Val}(q_r) & \frac{1}{2}{\rm Val}(q_t - q_r - q_s) \\
		\frac{1}{2}{\rm Val}(q_t - q_r - q_s) & {\rm Val}(q_s)
	\end{pmatrix}.
\end{eqnarray}
Here we used the assumption that there exist $l_1^*, l_2^* \in L^*$ such that
$q_r = \abs{l_1^*}^2$, $q_s = \abs{l_2^*}^2$, $q_t = \abs{l_1^* + l_2^*}^2$.


\subsection{Algorithm for $N = 3$}
\label{Lattice determination from weighted theta series in N = 3}

The theorems in Section \ref{Results in N = 3}
state how to construct the Gram matrices of lattices of rank 3 from elements of $\Lambda^{obs}$ satisfying the equation $3 \abs{l_1^*}^2 + \abs{l_1^* + 2 l_2^*}^2 = \abs{2 l_1^* + l_2^*}^2 + 3 \abs{l_2^*}^2$.
Considering the case in which powder diffraction patterns contain only a small number of peaks,
it is better to also use $q$-values satisfying the parallelogram law in the enumeration algorithm.
Hence, the following two computational assumptions are used
in the algorithm of Table \ref{Enumeration of the three dimensional lattices enumTwoDimLattices}:
\begin{enumerate}[(\text{$C$}1)]
	\item \label{item:2 (q_r + q_s) = q_t + q_u}
	If $q_r, q_s, q_t, q_u \in \Lambda^{obs}$ satisfy $2 (q_r + q_s) = q_t + q_u$,
	then there are $l_1^*, l_2^* \in L^*$ such that
	$q_r = \abs{l_1^*}^2, q_s = \abs{l_2^*}^2, q_t = \abs{l_1^*+l_2^*}^2, q_u = \abs{l_1^*-l_2^*}^2$.
	\item \label{item:3 q_r + q_t = 3 q_s + q_u}
	If $q_r, q_s, q_t, q_u \in \Lambda^{obs}$ satisfy $3 q_r + q_t = 3 q_s + q_u$, then
	there are $l_1^*, l_2^* \in L^*$ such that
	$q_r = \abs{l_1^*}^2, q_s = \abs{l_2^*}^2, q_t = \abs{l_1^*+2l_2^*}^2, q_u = \abs{2l_1^*+l_2^*}^2$.
\end{enumerate}

\begin{table}
	\caption{Enumeration algorithm for three-dimensional lattices constructed from $q$-values.}
	\label{Enumeration of the three dimensional lattices enumTwoDimLattices}
\begin{minipage}{\textwidth}
\begin{tabular}{lcp{85mm}}
	\hline
\multicolumn{3}{c}{{\bf void enumerateThreeDimLattices($\Lambda^{obs}, c, {\rm det}S_{min}, {\rm det}S_{max}, Ans$)}} \\
(Input) & $\Lambda^{obs}, c$& : same as in Table \ref{Enumeration of the two dimensional lattices enumTwoDimLattices}. \\
 & ${\rm det}S_{min}$, ${\rm det}S_{max}$ & : lower and upper thresholds on determinants of output matrices. \\
(Output) & $Ans$ & : array of $3 \times 3$ positive-definite symmetric matrices.
\end{tabular}
\\
\begin{enumerate}[(1)]
	\item \label{item:enumeration of Ito equations}
	By the method in Table \ref{Enumeration of the two dimensional lattices enumTwoDimLattices},
	enumerate $q_r, q_s, q_t, q_u$ of $\Lambda^{obs}$ satisfying $2(q_r + q_s) = q_t + q_u$,
	and insert $( \{ q_r, q_s \}, \{ q_t, q_u \} )$ in $A_2$.
	(Here, $A_2$ is an array of four formal sums $( \{ Q_1, Q_2 \}, \{ Q_3, Q_4 \} )$.)

	\item \label{item:enumeration of new equations}
	Enumerate $q_r, q_s, q_t, q_u$ of $\Lambda^{obs}$ satisfying 
	the equation $3 q_r + q_t = 3 q_s + q_u$.
	This is done by a similar method as in Table \ref{Enumeration of the two dimensional lattices enumTwoDimLattices}.
	Using a new formal sum $q_{-1} := \frac{- 2 q_r + q_s + q_u}{2} = \frac{q_r - 2 q_s + q_t}{2}$,
	two sets of $q$-values satisfying the parallelogram law are generated:
	\begin{eqnarray}
		2 (q_{-1} + q_r) = q_s + q_u,\ 2(q_{-1} + q_s) = q_r + q_t.
	\end{eqnarray}
	Check whether $A_2$ contains $( \{ q_w, q_r \}, \{ q_s, q_u \} )$ or $( \{ q_w, q_s \}, \{ q_r, q_t \} )$ for some $1 \leq w \leq N_{peak}$.
	If not, this suggests that $q_{-1} := \frac{- 2 q_r + q_s + q_u}{2} = \frac{q_r - 2 q_s + q_t}{2}$
	equals $\abs{l^*}^2$ for some $l^* \in L^*$ undetected owing to some observational reason.
	Insert $( \{ \frac{- 2 q_r + q_s + q_u}{2}, q_r \}, \{ q_s, q_u \} )$ and $( \{ \frac{q_r - 2 q_s + q_t}{2}, q_s \}, \{ q_r, q_t \} )$ in $A_2$.

	\item \label{item:preparation of enumeration of rank 3} 
		For every entry $(\{ Q_1, Q_2 \}, \{ Q_3, Q_4 \}) \in A_2$,  
		insert 
		$(Q_1, Q_2, Q_3, Q_4)$,
		$(Q_1, Q_2, Q_4, Q_3)$,
		$(Q_2, Q_1, Q_3, Q_4)$,
		$(Q_2, Q_1, Q_4, Q_3)$ in a new array $A_3$.

	\item \label{item:enumeration of rank 3} For every $(Q_1, Q_2, Q_3, Q_4) \in A_3$,
		search $(R_1, R_2, R_3, R_4) \in A_3$ satisfying either of the following:
		\begin{enumerate}[(a)]
			\item $Q_1 = R_1 \in \Lambda^{obs}$.
			\item $Q_1, R_1 \notin \Lambda^{obs}$ and $\abs{ {\rm Val}(Q_1 - R_1) } \leq c {\rm Err}(Q_1 - R_1)$.
		\end{enumerate}
		In addition, for every $q_k \in \Lambda^{obs}$,
		assume that there exists $l_1^*$, $l_2^*$, $l_3^*$ satisfying$^a$
		\begin{eqnarray}
			& Q_1 \approx R_1 = \abs{l_1}^2,\
			Q_2 = \abs{l_2}^2,\
			Q_3 = \abs{l_1+l_2}^2, & \nonumber \\
			& R_2 = \abs{l_3}^2,\
			R_3 = \abs{l_1+l_3}^2,\
			q_k = \abs{l_1+l_2+l_3}^2. & 
		\end{eqnarray}
		Then, the Gram matrix $S := (l_i \cdot l_j)_{1 \leq i, j \leq 3}$
		is obtained as the following $3 \times 3$ symmetric matrix:
		\begin{eqnarray}
			\begin{pmatrix}
				Q_1 & \frac{1}{2}(Q_3-Q_1-Q_2) & \frac{1}{2}(R_3-Q_1-R_2) \\
				\frac{1}{2}(Q_3-Q_1-Q_2) & Q_2 & \frac{1}{2}(Q_1 - Q_3 - R_3 + q_k) \\
				\frac{1}{2}(R_3-Q_1-R_3) & \frac{1}{2}(Q_1 - Q_3 - R_3 + q_k) & R_2
			\end{pmatrix}.
		\end{eqnarray}
		Using ${\rm Val}$ and ${\rm Err}$, the values and propagated errors of the entries are computable.
		If ${\rm det}S_{min} \leq {\rm det}S \leq {\rm det}S_{max}$, insert $S$ in $Ans$.
\end{enumerate}

\end{minipage}
\end{table}

The computation time of the procedure in Table \ref{Enumeration of the three dimensional lattices enumTwoDimLattices}
is roughly proportional to $N_{zone}^2$,
where $N_{zone}$ is the size of $A_2$ immediately after (\ref{item:enumeration of new equations}).
This is estimated as follows.
The size of $A_3$ in (\ref{item:preparation of enumeration of rank 3}) is approximately $4 N_{zone}$, if the cases $Q_1=Q_2$ or $Q_3=Q_4$ are ignored.
When $N_{peak}$ is the cardinality of $\Lambda^{obs}$,
the average number of $(R_1, R_2, R_3, R_4) \in A_2$ satisfying (a) or (b) with regard to fixed $(Q_1, Q_2, Q_3, Q_4) \in A_2$ 
is approximated as $\frac{4 N_{zone}}{N_{peak}}$.
Hence, the number of combinations of $(Q_1, Q_2, Q_3, Q_4)$, $(R_1, R_2, R_3, R_4) \in A_2$ and $q_k \in \Lambda^{obs}$
is roughly equal to $4 N_{zone} \cdot \frac{4 N_{zone}}{N_{peak}} \cdot N_{peak} = 16 N_{zone}^2$.
Steps (1)--(3) take much less time than (4).
Hence it is concluded that the time is proportional to $N_{zone}^2$. 

The enumeration is completed 
by calling the procedure in Table \ref{Enumeration of the three dimensional lattices enumTwoDimLattices}, after which 
the following procedures are required before outputting solutions.
\begin{enumerate}
	\item Transform every enumerated Gram matrix into a Niggli-reduced form \cite{Niggli28} (which is standard in crystallography).
	\item Bravais lattice determination.
	\item Sort solutions by some figures of merit (\EG de Wolff figure of merit, widely used in powder auto-indexing \cite{Wolff68}).
	\item Remove duplicate solutions.
\end{enumerate}

\subsection{Speed-up method using topographs}
\label{Speed-up using topographs}

As described in Section \ref{Lattice determination from weighted theta series in N = 3},
the computation time of the algorithm in Table \ref{Enumeration of the three dimensional lattices enumTwoDimLattices}
is proportional to the square of the size $N_{zone}$ of $A_2$ immediately after (\ref{item:enumeration of new equations}).
Thus, an effective way to speed up the algorithm is to reduce the size of $A_2$.

When reducing the size of $A_2$,
it is necessary to retain elements obtained from $q_r, q_s, q_t, q_u \in \Lambda^{obs}$ for which 
the assumption ($C$\ref{item:2 (q_r + q_s) = q_t + q_u}) or ($C$\ref{item:3 q_r + q_t = 3 q_s + q_u})
is true, because they are essential to obtain the true solution.
A new criterion to sort the elements of $A_2$ is defined for the purpose.

Every entry $(\{ R_1, R_2 \}, \{ R_3, R_4 \})$ of $A_2$ consists of four formal sums satisfying the parallelogram law $2(R_1 + R_2) = R_3 + R_4$.
Hence, it corresponds to the subgraph of a topograph
as in Figure \ref{Substructure of a topograph corresponding to the parallologram law}.
Table \ref{Algorithm to construct the subgraph of a topograph} explains how a subgraph $T$ of a topograph is constructed
from these substructures.
The elements of $A_2$ utilized to obtain the subgraph are output in $\tilde{A}_2$.

\begin{table}
	\caption{Recursive procedure to form a subgraph of a topograph from $\tilde{e}$ and elements of $A_2$.}
	\label{Algorithm to construct the subgraph of a topograph}
\begin{minipage}{\textwidth}
\begin{tabular}{lcp{100mm}}
	\hline
\multicolumn{3}{c}{{\bf void expandSubtopograph($A_2$, $\tilde{e}$, $\tilde{A}_2$, $T$)}} \\
(Input) & $A_2$ & : array of four formal sums  $(\{ Q_1, Q_2\}, \{ Q_3, Q_4 \})$ with $Q_i := \sum_{i=1}^{N_{peak}} n_{ij} q_j$.  Furthermore, it is assumed that every entry satisfies \\ \\
	& & 
${\rm Val}( 2Q_1 + 2Q_2 - Q_3 - Q_4) \leq c\ {\rm min}\{ 2 {\rm Err}(Q_1 + Q_2), {\rm Err}(Q_3 + Q_4) \}$, \\ \\
	& & for some fixed constant $c$, and $Q_3$, $Q_4$ and either of $Q_1, Q_2$ belong to $\Lambda^{obs}$ (\IE there is $q \in \Lambda^{obs}$ such that $Q_i = q$).\\
		& $\tilde{e}$ & :$= (\{ R_1, R_2\}, R_3, R_4 )$ satisfying $(\{ R_1, R_2\}, \{ R_3, R_4 \}) \in A_2$. \\ 
(Output) & $\tilde{A}_2$ & : subset of $A_2$. \\
         & $T$ & : subgraph$^a$ of a topograph composed of substructures corresponding to $(\{ Q_1, Q_2\}, \{ Q_3, Q_4 \}) \in \tilde{A}_2$.
				(If such subgraphs are not unique, $\tilde{A}_2$ containing larger number of entries is prioritized.) \\
\end{tabular}
	\begin{tabular}{p{5mm}p{10mm}p{125mm}}
 	1: & (Start) & For $(i, j) = (1, 2), (2, 1)$, let $S_{ij}$ be the set defined by \\
	2   & & $S_{ij} := 
	\begin{cases}
			\left\{ (\{ R_i, R_4 \}, R_j, q) \in A_2 : q \in \Lambda^{obs} \right\} & \text{if } R_j \in \Lambda^{obs},  \\
			\emptyset & \text{otherwise.}
	\end{cases}$ \\
 	3: & & $\tilde{T}_{12} := \emptyset$. $\tilde{A}_{12} := \emptyset$. $\tilde{T}_{21} := \emptyset$. $\tilde{A}_{21} := \emptyset$. \\
 	4: & & for $i = 1$ to $2$ do \\
 	5: & & \hspace{5mm}	Take $j \in \{ 1, 2 \} \setminus \{ i \}$. \\
 	6: & & \hspace{5mm}	for $\tilde{e}_2 \in S_{ij}$ do \\
 	7: & & \hspace{10mm}	Call expandSubtopograph($A_2$, $\tilde{e}_2$, $T_{ij}$, $A_{ij}$). \\
 	8: & & \hspace{10mm}	if $\abs{ \tilde{A}_{ij} } < \abs{ A_{ij} }$ then \\
 	9: & & \hspace{15mm}		$\tilde{T}_{ij} := T_{ij}$. \\
 	10: & & \hspace{15mm}		$\tilde{A}_{ij} := A_{ij}$. \\
 	11: & & \hspace{10mm}	end if \\
 	12: & & \hspace{5mm}	end for \\
 	13: & & end for \\
 	14: & & Set $\tilde{A}_2 := \{ (\{ R_1, R_2\}, \{ R_3, R_4 \}) \} \cup \tilde{A}_{12} \cup \tilde{A}_{21}$. \\
 	15: & & Construct $T$ by unifying $\tilde{T}_{12}$, $\tilde{T}_{21}$ and $(\{ R_1, R_2\}, \{ R_3, R_4 \})$, as in Figure \ref{Extension of a topograph1}. \\
	\end{tabular}
\footnotetext[1]{In the actual computation, 
it is not necessary to construct a subgraph $T$,
because a sort criterion for elements of $A_2$ is defined using only $\tilde{A}_2$.
However, as $T$ has a data structure of binary trees consisting of finite nodes and edges,
it is not difficult to implement $T$ in a program.
}
\end{minipage}
\end{table}

\begin{figure}
\begin{minipage}{\textwidth}
\begin{center}
\scalebox{0.5}{\includegraphics{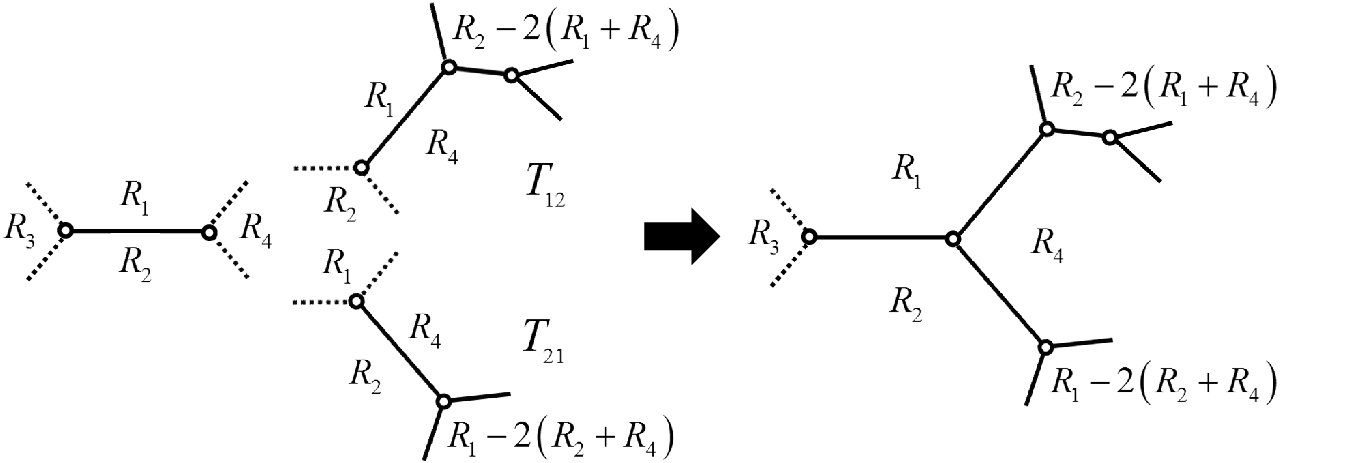}}
\end{center}
\footnotetext[1]{
Every three graphs on the left-hand side have a common 
node, which is an end point of three edges associated with $\{ R_1, R_2 \}$, $\{ R_1, R_4 \}$, and $\{ R_2, R_4 \}$.
This figure illustrates how these graphs are unified.
}
\end{minipage}
\caption{Extension of a topograph (1/2).}
\label{Extension of a topograph1}
\end{figure}

\begin{figure}
\begin{minipage}{\textwidth}
\begin{center}
\scalebox{0.5}{\includegraphics{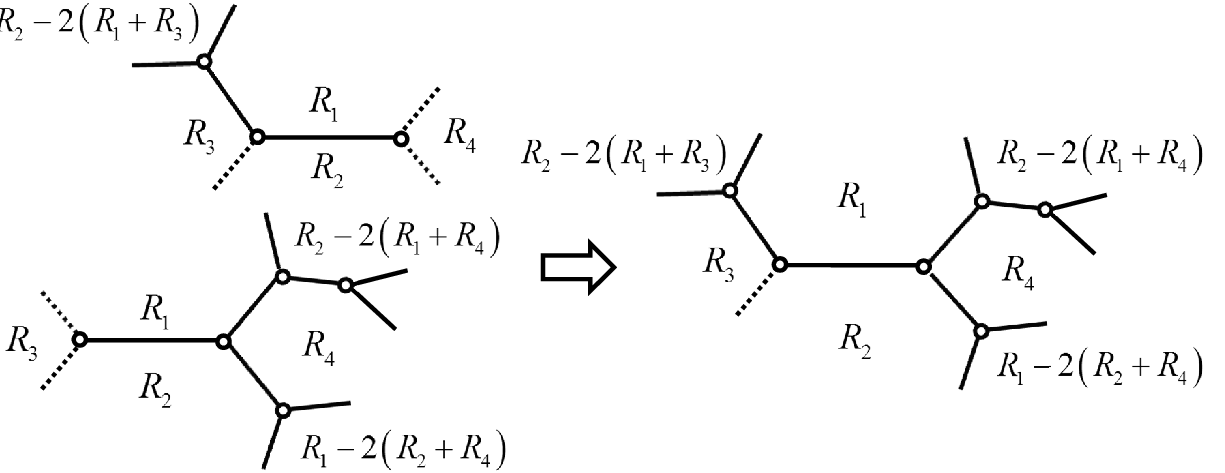}}
\end{center}
\footnotetext[1]{
The left-hand two subgraphs
contain a subgraph as in Figure \ref{Substructure of a topograph corresponding to the parallologram law} commonly. 
They are unified as above.
}
\end{minipage}
\caption{Extension of a topograph (2/2).}
\label{Extension of a topograph2}
\end{figure}

In the procedure of Table \ref{Algorithm to construct the subgraph of a topograph},
the subgraph $T$ is expanded to only one side.
The whole subgraph composed of $e := (\{ R_1, R_2 \}, \{ R_3, R_4\}) \in A_2$ and entries of $A_2$ 
is obtained by calling 
the recursive procedure twice,
setting $\tilde{e} := (\{ R_1, R_2 \}, R_3, R_4)$ and $(\{ R_1, R_2 \}, R_4, R_3)$ respectively in the second argument.
Let $C(e)$ be the cardinality of the union of the two $\tilde{A}_2$
output by calling the recursive procedure twice as above. 
This $C(e)$ is considered to quantify the size of the subgraph finally obtained.

If either ($C$\ref{item:2 (q_r + q_s) = q_t + q_u})
or ($C$\ref{item:3 q_r + q_t = 3 q_s + q_u}) is wrong for $e \in A_2$,
$C(e)$ will result in a small number, 
because the new $q$-values required to extend a new edge (\IE $q \in \Lambda^{obs}$ in line 2 of Table \ref{Algorithm to construct the subgraph of a topograph}) would rarely be found in $\Lambda^{obs}$. 
This suggests $C(e)$ provides an effective sort criterion for elements of $A_2$.
Part (b) of Theorem \ref{thm:distribution rules on CT_{2, S_2}} 
indicates that this criterion remains effective under the influence of systematic absences.

Now any $e := (\{ R_1, R_2 \}, \{ R_3, R_4 \}) \in A_2$ 
corresponds to a $2 \times 2$ Gram matrix $S(e) \in {\mathcal S}^2$ by the following map:
\begin{eqnarray}\label{eq:Gram matrix of zone 2}
	S(( \{ R_1, R_2 \}, \{ R_3, R_4 \} )) \mapsto
	\begin{pmatrix}
		{\rm Val}(R_1) & \frac{1}{2}{\rm Val}(R_3 - R_1 - R_2) \\
		\frac{1}{2}{\rm Val}(R_3 - R_1 - R_2) & {\rm Val}(R_2)
	\end{pmatrix}.
\end{eqnarray}

When $C(e)$ is used as a sort criterion,
$e \in A_2$ with smaller ${\rm det} S(e)$ is prioritized as a result,
because $\Lambda^{obs} \subset [q_{min}, q_{max}]$ contains a smaller number of $\tr{u} S(e) u$ ($u \in \IntegerRing^2$) if ${\rm det} S(e)$ has a larger value.
This property of $C(e)$ is also desirable,
because (\ref{item:enumeration of rank 3}) in Table \ref{Enumeration of the three dimensional lattices enumTwoDimLattices}
(and Theorem \ref{thm:distribution rules on CT_{3, S_3}})
requires $\{ l_1^*, l_i^* \} \in P_2(L^*)$ ($i = 2, 3$) to obtain the true solution $L^*$.
By using $C(e)$, a sublattice $L_i^*$ expanded by $\{ l_1^*, l_i^* \} \in P_2(L^*)$ is prioritized 
over any other sublattices contained in $L_i^*$.

To summarize the above discussion, 
it is only necessary to insert the following procedures between (\ref{item:enumeration of new equations}) and (\ref{item:preparation of enumeration of rank 3}) in Table \ref{Enumeration of the three dimensional lattices enumTwoDimLattices},
in order to reduce the computation time.
\begin{enumerate}[($2$i)]
	\item \label{item:speed up 1} For each element $e := (\{ R_1, R_2 \}, \{ R_3, R_4 \}) \in A_2$,
		compute $C(e)$ by calling the procedure in Table \ref{Enumeration of the three dimensional lattices enumTwoDimLattices}. 
		(In this case, any $e_2 \in \tilde{A}_2$ satisfies $C(e) = C(e_2)$.
		Hence, the number of calls of the procedure is less than the size of $A_2$.)

	\item \label{item:speed up 2} Sort $A_2$ in descending order of $C(e)$.
		For $e_1, e_2 \in A_2$ satisfying $C(e_1) = C(e_2)$,
			they may be sorted by $e_1 < e_2 \underset{def}{\Longleftrightarrow} {\rm det}S(e_1) < {\rm det}S(e_2)$.
		
	\item \label{item:speed up 3} 
		Remove the $(N_{zone}+1)$th-to-last elements of $A_2$ using a fixed integer $N_{zone} > 0$.
		(The default value of $N_{zone}$ used in Conograph is given in (\ref{eq:definition of N_{zone}}) of section \ref{Parameters for calculation conditions}.  When the default value is used,
the computation time is approximately proportional to ${N_{peak}}^4$.)
\end{enumerate}

In Section \ref{Results},
we shall see how the computation time is decreased in practice by this improvement.

\subsection{Problems with the quality of powder diffraction patterns}
\label{Problems on quality of powder diffraction patterns}

Among the problems described in ($A$1)--($A$6) and ($\tilde{A}$\ref{item:error2}) of Section \ref{Precise formulation of powder auto-indexing problem and our strategy},
we have not yet clarified how to handle ($\tilde{A}$\ref{item:error2}).
In this section, we explain 
how missing and false elements in $\Lambda^{obs}$ influence the powder auto-indexing results.
This issue is related to the quality of powder diffraction patterns as observational data.

Using a peak search program equipped with Conograph, it is not difficult to obtain the positions of all the peak heights above a given threshold automatically and rather uniformly.
Although the ability to decompose overlapping peaks depends on the peak search software,
this is not a great problem in our algorithm,
because an almost identical solution will be obtained
even if $q_1 \in \Lambda^{obs}$ is replaced with $q_2$ very close to $q_1$ 
(to some degree, ${\rm Err}[q_1]$ and ${\rm Err}[q_2]$ will absorb the difference between $q_1$ and $q_2$).

By the algorithm in Table \ref{Enumeration of the three dimensional lattices enumTwoDimLattices},
normally, multiple Gram matrices of the reciprocal lattice $L^*$ of the correct solution $L$ are generated 
from $\Lambda^{obs}$.
This is because a lattice $L^*$ has as infinitely many Gram matrices as bases of $L^*$.
Note that Gram matrices of different bases are computed 
from different $q$-values basically.
Therefore when these Gram matrices are transformed into a reduced form,
they are a bit different matrices owing to observational errors, but rather close to each other.

In order to reduce the influence of $\epsilon_1 > 0$ of ($\tilde{A}$\ref{item:error2}),
the existence of such duplicate solutions is very useful.
Suppose that $m$ Gram matrices of $L^*$ are generated from the set of $q$-values $\Lambda_{ext}(\wp) \cap [q_{min}, q_{max}]$ containing no observational errors by applying our enumeration method.
Then, when the same method is applied to $\Lambda^{obs}$ containing observational errors,
the enumeration process fails to obtain the correct solution,
only when none of approximate solutions of the $m$ Gram matrices them are generated.
The failure rate becomes very small as $m$ increases,
regardless of the magnitude of $\epsilon_1$.  
If the size of $\Lambda^{obs}$ is augmented, or equivalently the range $[q_{min}, q_{max}]$ is magnified,
$m$ increases naturally.  
As a result, the enumeration success rate increases monotonically as the size of $\Lambda^{obs}$ increases.
However, it should be noted that the time for enumeration of solutions is roughly proportional to the fourth power of the size of $\Lambda^{obs}$, if the default parameters of Conograph are used.
In addition, the success probability will not increase much once $q_{max}$ reaches some observational limit,
because larger $q$-values have larger errors owing to peak overlap and observational accuracy.

Next, we discuss the influence of $\epsilon_2$ in ($\tilde{A}$\ref{item:error2}). 
In this case, the success rate in obtaining the true solution
remains same if $\Lambda^{obs}$ is replaced with $\Lambda^{obs} \cup \{ q^{obs} \}$ which contains 
a false element $q^{obs}$.
(Of course, such $q^{obs}$ increases the time for enumeration.)

In conclusion, the enumeration process is considered to be robust to missing and false elements in $\Lambda^{obs}$, 
although such elements increase the computation time.
Indeed, the procedure to sort candidate solutions executed after the enumeration is more sensitive to missing and false elements, because figures of merit are severely affected by these elements.
In general, for a poor quality powder diffraction pattern,
it is very difficult to judge which solution is correct,  
even if the $q$-values of respective solutions are manually compared to actual peak-positions 
(see Example \ref{exam:bad quality} in Section \ref{Results}).

Example \ref{exam:two phase data} shows the case of a two-phase sample.
Both lattice parameters are acquired by Conograph.
However it seems to be almost impossible to judge which is the correct second phase parameter in this case.

\section{Implementation and results of Conograph}
\label{Computational results of Conograph}

Before introducing the input parameters and results of Conograph,
we explain the circumstances in which the program was verified. 
The Conograph source code is written in C++ and OpenMP,
and compiled with Mingw (GNU Compiler Collection for Windows). 
The computer used for the test
has an Intel i7-2620 (2.70 GHz) processor and 12 GB RAM.
The processor can execute parallel computing with eight hyper-threads,
all of which were used during the test.
Additionally, we confirmed that a computer with 4 GB RAM was able to carry out the same test.


\subsection{Input parameters of Conograph}
\label{Parameters for calculation conditions}

The input parameters used in the enumeration process
are listed in Table \ref{Parameters used in test}.
``AUTO'' is used to set parameters that are considered to depend on respective powder diffraction patterns.
The parameters required after the enumeration are not given here;
they will be introduced at another time.

\begin{table}[htbp]
\caption{Default parameters of Conograph.}
\label{Parameters used in test}
\begin{minipage}{\textwidth}
\	\begin{tabular}{lll}
\hline
		{\bf Symbol} & {\bf Meaning} & {\bf Default} \\
		$\Delta 2 \theta$ & Zero-point shift  (degree) & 0. \\
		$c$ & Tolerance level for errors in $q$-values \\
		& \hspace{10mm} Characteristic X-rays or neutron reactor sources & 1.5 \\
		& \hspace{10mm} Synchrotron X-rays or neutron spallation sources & 1. \\
		$N_{peak}$ & Number of $q$-values used & AUTO \\
		$N_{zone}$ & Threshold for the maximum number of $(\{ Q_1, Q_2 \}, \{ Q_3, Q_4 \})$ & AUTO  \\
		$N_{sol}$ & Threshold for the maximum number of solutions \\
			& \hspace{10mm} First trial & AUTO  \\
			& \hspace{10mm} When the first trial failed because of too many solutions$^a$ & $64000 \times m$ ($m \in \IntegerRing_{> 0}$)  \\
		${\rm Vol}_{min}$ & Threshold for the minimum volume of $\RealField^3 / L$ ($\AA^3$) & AUTO \\
		${\rm Vol}_{max}$ & Threshold for the maximum volume of $\RealField^3 / L$ ($\AA^3$) & AUTO
	\end{tabular}
\footnotetext[1]{
$N_{sol}$ is set to 64000 at most when AUTO is used,
in order to prevent memory allocation errors in memoryless computers. 
When the number of enumerated solutions exceeds $N_{sol}$,
Conograph removes those with a smaller unit-cell volume.
As a result, the true solution is sometimes also removed.
The simplest method of resolving this is to reduce ${\rm Vol}_{max}$ or increase $N_{sol}$ when unsatisfactory results are obtained.
This time, we chose the latter in order to measure computation time. 
However, such a trial-and-error approach is not necessary,
because an alternative method is adopted in Conograph.
This will be introduced at another time.
}
\end{minipage}
\end{table}

In the following, we provide an explanation of the parameters and formulas needed to compute AUTO.

\begin{enumerate}[(1)]
	\item {\bf Zero-point shift $\Delta 2 \theta$.}
		Depending on the type of diffractometers, 
		the $x$-axis of a powder diffraction pattern is represents a diffraction angle $2 \theta$ or a time-of-flight $t$ given by the following function of $q$-values.
		\begin{eqnarray}
			2 \theta &=& 2 \sin^{-1}\frac{\lambda \sqrt{q}}{2}  + \Delta 2 \theta, \\
			t &=& \sum_{i=1}^n c_i q^{-i/2}.
		\end{eqnarray}
		Using these equations, $x$-coordinates are transformed to $q$-values before executing powder auto-indexing.
		Among $\lambda$, $\Delta 2 \theta$ and $c_i$ ($1 \leq i \leq n$),
		only the value of $\Delta 2 \theta$ is unknown.
 		Conograph users are recommended to set $\Delta 2\theta = 0$ normally, because
 		successful results were obtained
 		even in cases of very large zero-point shift, such as $\Delta 2 \theta \approx 0.2$ degree.
 		After auto-indexing, it is possible to refine the
 		lattice parameters and zero-point shift simultaneously using a nonlinear least-squares method.

	\item \label{item:Tolerance level}{\bf Tolerance level $c$ for errors in $q$-values.}
		Table \ref{Enumeration of the two dimensional lattices enumTwoDimLattices} provides a usage example of $c$.  
		By setting a larger value of $c$,
		a wider range of combinations of $q$-values is searched. 

	\item \label{item:Number of used q-values}{\bf Number of $q$-values used.}
		This parameter determines the size of $\Lambda^{obs}$, an array of input $q$-values.
		After sorting the $q$-values in $\Lambda^{obs}$ into ascending order,
		the $(N_{peak}+1)$th-to-last parameters are removed before the powder auto-indexing commences. 
		As explained in Section \ref{Problems on quality of powder diffraction patterns}, both the computation time and success rate increase as $N_{peak}$ is magnified.
		The default value of $N_{peak}$ is calculated by:
		\begin{eqnarray}\label{eq:formula of N_{peak}}
			N_{peak} &:=& {\rm min}\left\{ \sharp\{ q < 10 / d^2 : q \in \Lambda^{obs} \}, 48 \right\},
		\end{eqnarray}
		where $\sharp T$ is the cardinality of the set $T$, 
		and $d$ is the lower threshold for the distance between two lattice points.
		In Conograph, $d$ is set to $2.0 \AA$.
		
		In (\ref{eq:formula of N_{peak}}), $N_{peak} \leq 48$ is forced,
		because it is frequently meaningless to use large $q$-values owing to severe peak overlap. 
		However, 48 is still much larger than the 20--30 that are normally adopted for powder auto-indexing.
		Conograph uses such many $q$-values for another reason;
		in ($A$\ref{item:size of cells}) and the paragraphs following ($A$\ref{item:q_{max} - q_{min}}),
		we explained 
		that $q_{max} > D_N := 3 \cdot 2^{-1/3} d^{-2} \approx 0.6 \AA^{-2}$ is required at least theoretically. 
		(In (\ref{eq:formula of N_{peak}}), $q_{max} > 4 D_N$ is adopted.)
		Figure \ref{Range [q_{min}, q_{max}] of the smallest N_{peak} q-values} presents the range of the first $N_{peak}$ $q$-values 
		in the test data of Table \ref{Range [q_{min}, q_{max}] of the smallest N_{peak} q-values}.
		Owing to the threshold of 48, the interval $[q_{min}, q_{max}]$ 
		is frequently much smaller than $0.6 \AA^{-2}$, although powder auto-indexing succeeded in all the test data.
		We suppose 48 $q$-values might be insufficient in some exceptional cases.
		Users are recommended to increase $N_{peak}$ manually,
		if results obtained with the default parameters are unsatisfactory.
	
\begin{figure}
\begin{minipage}{\textwidth}
\begin{center}
\scalebox{0.65}{\includegraphics{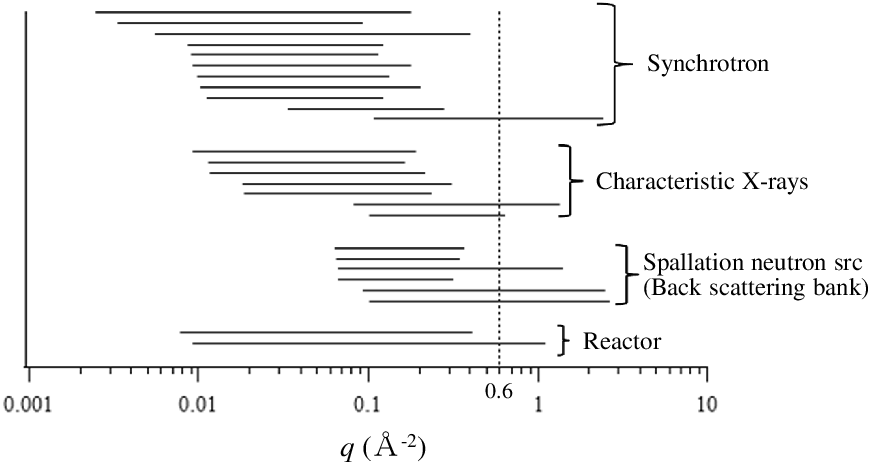}}
\end{center}
\footnotetext[1]{
If $q_{max} > 0.6 \AA^{-2}$ is imposed,
the interval $[q_{min}, q_{max}]$ often includes more than several tens of $q$-values of diffraction peaks.
We should not increase $q_{min}$,
because smaller $q$-values have better accuracy.
As a result, $[q_{min}, q_{max}]$ set by the default parameters often does not satisfy the theoretical requirement $q_{max} > 0.6 \AA^{-2}$. 
Nevertheless, powder auto-indexing succeeded in all the test data.
This is considered to be because all our test data satisfy $d \geq 2.7 \AA$ (and many also satisfy $d \geq 4 \AA$),
and because the formula (\ref{eq:formula by Lagarias}) of Lagarias \textit{et al.\ } provides an overestimation of $D_3$.
}
\end{minipage}
\caption{Range $[q_{min}, q_{max}]$ of $N_{peak}$ $q$-values used in powder auto-indexing.}
\label{Range [q_{min}, q_{max}] of the smallest N_{peak} q-values}
\end{figure}
		
	\item {\bf Threshold for the maximum size of $(\{ Q_1, Q_2 \}, \{ Q_3, Q_4 \})$.}
		\begin{eqnarray}\label{eq:definition of N_{zone}}
			N_{zone} := \frac{1}{3} N_{peak} (N_{peak} + 1).
		\end{eqnarray}

	\item {\bf Threshold for the maximum number of candidate solutions.} 
		\begin{eqnarray}
			N_{sol} := {\rm min}\{ 64000, N_{zone}^2 \}.
		\end{eqnarray}
		
	\item {\bf Threshold for the minimum and maximum of the volume of $\RealField^3 / L$.}
		\begin{eqnarray}
			{\rm Vol}_{min} &:=& {\rm max}\{ 5, v_{20}^{-1} \}, \label{eq:definition of Vol_{min}} \\
			{\rm Vol}_{max} &:=& 30 {\rm Vol}_{min}, \label{eq:definition of Vol_{max}}
		\end{eqnarray}
		where $5 \AA^3$ is chosen as the lower threshold for the volumes of existing crystals,
		and $v_{20}$ is the upper bound of ${\rm Vol}(\RealField^3 / L^*)$, 
		estimated using the 20 smallest elements of $\Lambda^{obs} := \{ q_1, q_2, \ldots, q_{N_{Peak}} \}$ by:
		\begin{eqnarray}\label{eq:formula of c_j}
			v_j := \frac{2 \pi}{3}  \frac{ q_j^{3/2} - q_1^{3/2} }{j - 1}.
		\end{eqnarray}
		Equation (\ref{eq:formula of c_j}) is based on the following formula, which holds for any $0 < r < R$. 
		\begin{eqnarray}
			& & \hspace{-10mm}
			\frac{2 \pi(R^3 - r^3)}{3} \abs{ \{ r^2 \leq \abs{l^*}^2 \leq R^2 : l^* \in L^* \} }^{-1} \nonumber \\
 			& \geq & 
			\frac{4 \pi(R^3 - r^3)}{3} \abs{ \{ l^* \in L^* : r^2 \leq \abs{l^*}^2 \leq R^2 \} }^{-1}
			\rightarrow {\rm Vol}(\RealField^3 / L^*) \text{ as } R \rightarrow \infty. \hspace{10mm}
		\end{eqnarray}
		Note that 20 and 30 are chosen empirically,
		and $v_{N_{Peak}}$ is used instead of $v_{20}$ if $N_{Peak} < 20$.
		We have found no cases when
		this $[ {\rm Vol}_{min}, {\rm Vol}_{max}]$ fails to contain the correct volume.
\end{enumerate}

\subsection{Results}
\label{Results}

We prepared 26 + 4 powder diffraction patterns as test data.
A summary of the first 26 test data is presented in Table \ref{Summary of test data}.
The remaining four are presented in 
Examples \ref{exam:Non-unique solutions}--\ref{exam:bad quality} to illustrate some rather special cases.

\begin{table}[htbp]
\caption{Summary of 26 test data$^a$}
\label{Summary of test data}
\begin{minipage}{\textwidth}
	\begin{tabular}{lll}
		\hline
		{\bf Diffractometer} & number of patterns \\
			Synchrotron & 11 \\
			Characteristic X-rays & 7 \\
			Spallation neutron sources (time-of-flight) & 6 \\
			Reactor neutron sources & 2 \\
		\hline
		{\bf Symmetry of lattice} \\
			Triclinic & 7 \\ 
			Monoclinic (P) & 4 \\ 
			Monoclinic (B) & 1 \\
			Orthorhombic (P) & 5 \\
			Tetragonal (P) & 2 \\
			Tetragonal (I) & 1 \\
			Rhombohedral & 1 \\
			Hexagonal & 1 \\
			Cubic (I) & 1 \\
			Cubic (F) & 3 \\
		\hline
		{\bf Distribution of absolute zero-point shifts$^b$} \\
			    -- 0.1 &  17 \\
			0.1 -- 0.15 & 1 \\
		    0.15 --          & 2 
	\end{tabular}
\footnotetext[1]{
Among eight samples distributed in SDPDRR-2,
samples 1--4, 8 are included herein.
The result of sample 7 is used in Example \ref{exam:bad quality}.
Samples 5, 6 are excluded this time, because we could not obtain solutions with de Wolff figures of merit $M_{20} > 10$
by using all the peaks with sufficiently large intensities.
These patterns seem to contain a considerable number of false peaks.
}
\footnotetext[2]{
The zero-point shifts were computed after execution of powder auto-indexing by non-linear least squares method.
Time-of-flight data are excluded here.
}
\end{minipage}
\end{table}

We now evaluate how the method in Section \ref{Speed-up using topographs} improved enumeration times.
In Section \ref{Lattice determination from weighted theta series in N = 3}, we explained that the time is roughly proportional to $N_{zone}^2$. 
From the formula (\ref{eq:definition of N_{zone}}),
the computation time of the algorithm in Table \ref{Enumeration of the three dimensional lattices enumTwoDimLattices} is approximately proportional to ${N_{peak}}^4$.
Figure \ref{Computation time of Conograph} illustrates the relation between the following pairs in practical tests:
\begin{itemize}
	\item number of $q$-values $N_{peak}$ and time for enumeration,
	\item number of $q$-values $N_{peak}$ and total time for powder auto-indexing,
	\item number of enumerated solutions and time for enumeration,
	\item number of enumerated solutions and total time for powder auto-indexing.
\end{itemize}


\begin{figure}
\begin{minipage}{\textwidth}
\begin{minipage}{0.52\textwidth}
\scalebox{0.5}{\includegraphics{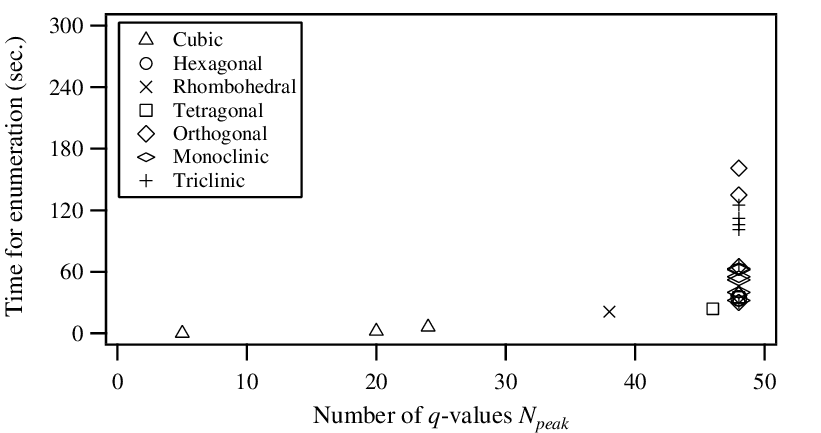}}
\end{minipage}
\begin{minipage}{0.52\textwidth}
\scalebox{0.5}{\includegraphics{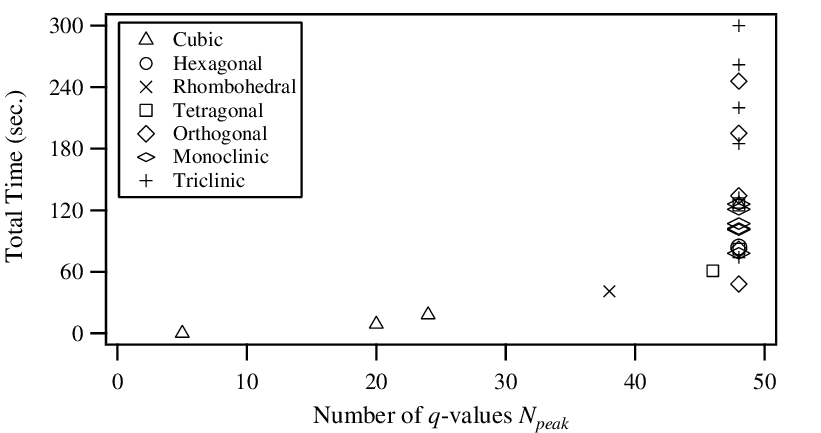}}
\end{minipage}
\begin{minipage}{0.52\textwidth}
\scalebox{0.5}{\includegraphics{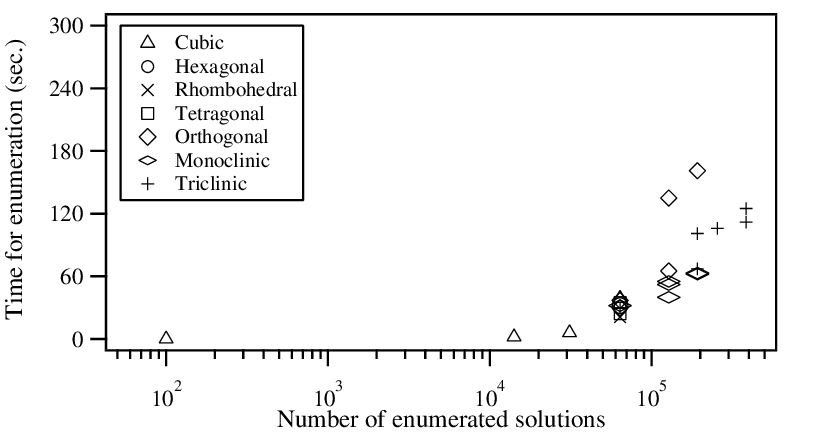}}
\end{minipage}
\begin{minipage}{0.52\textwidth}
\scalebox{0.5}{\includegraphics{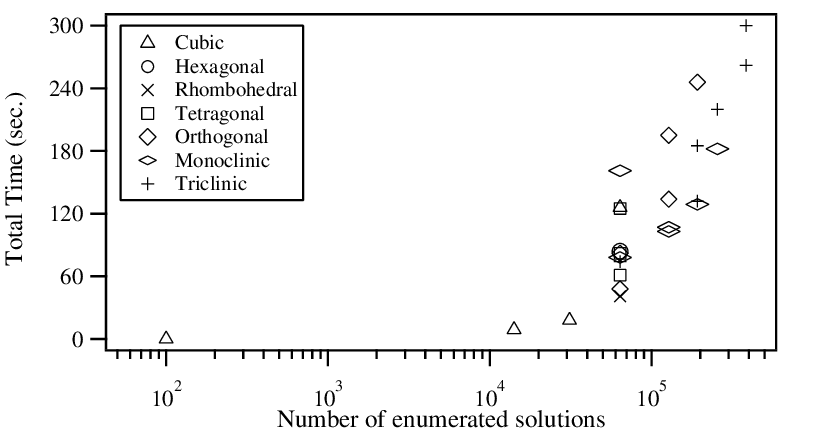}}
\end{minipage}
\footnotetext[1]{
Total time includes a) Enumeration of solutions, 
b) Transformation of every solution to a reduced form,
c) Computation of figures of merit,
d) Refinement of solutions using linear optimization,
e) Error-stable Bravais lattice determination, 
f) Removal of solutions having a figure of merit less than a user-input threshold,
g) Removal of duplicate solutions,
h) Sorting solutions with a selected figure of merit.
Of these, a) and e) 
are the most time-consuming.  
The time required for a) consumed approximately half of the total time.
We have already contributed to reducing the time for e) in \cite{Tomiyasu2012}.
}
\end{minipage}
\caption{Computation time of Conograph.}
\label{Computation time of Conograph}
\end{figure}

Figure \ref{Decrease ratio of the size of A_2} shows the rate of decrease in the size of $A_2$ as a result of applying the method described in Section \ref{Speed-up using topographs}. 

\begin{figure}
\begin{minipage}{\textwidth}
\begin{minipage}{0.52\textwidth}
\scalebox{0.5}{\includegraphics{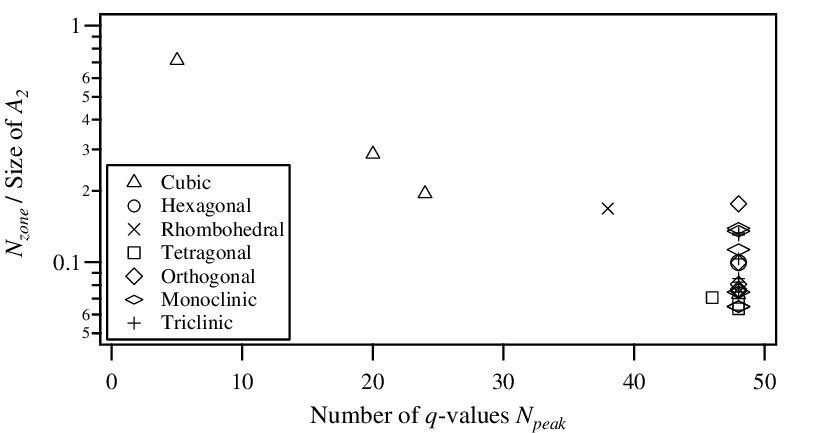}}
\end{minipage}
\begin{minipage}{0.52\textwidth}
\scalebox{0.5}{\includegraphics{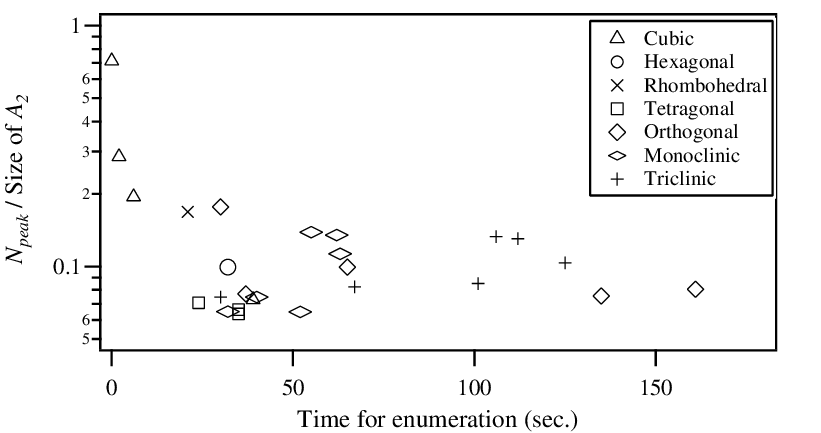}}
\end{minipage}
\footnotetext[1]{
Except for cases of small $N_{peak}$,
all the rates are in the range 0.06--0.18.
Under the assumption that the time for enumeration is proportional to $N_{zone}^2$,
the enumeration is about 32--250 times faster as a consequence.
The right-hand figure
indicates that a larger improvement was made in more time-consuming cases.
}
\end{minipage}
\caption{Rates of elements of $A_2$ used in enumeration algorithm.}
\label{Decrease ratio of the size of A_2}
\end{figure}

Even in the following difficult cases,
solutions were obtained without special parameter settings.
By Conograph, reliable powder auto-indexing results will become available even for less-experienced users.

\begin{example}\label{exam:Non-unique solutions}{\bf Non-unique solutions (Figure \ref{Results of Conograph} (a)).}
For any fixed $C > 0$,
the following lattice parameters have exactly the same $q$-values (\CF \ref{item:Case of N=3}.\ in Appendix \ref{Summary of known theorems on multiple solutions}).
\begin{eqnarray}
\text{Cubic(P)} &:& a = b = c = C, \alpha = \beta = \gamma = 90, \\ 
\text{Tetragonal(P)} &:& \sqrt{2} a = \sqrt{2} b = c = C, \alpha = \beta = \gamma = 90. 
\end{eqnarray}
Conograph succeeded in finding both of these.
\end{example}

\begin{example}\label{exam:small q lost}{\bf Small $q$-values are lost (Figure \ref{Results of Conograph} (b)).}
When the size of the unit cell is large,
many $q$-values are frequently lost because they are smaller than $q_{min}$.  
We have confirmed that Conograph is very robust to such a loss.
This example presents the case in which the 19 smallest non-zero $q$-values are not  
included in the observed range $[q_{min}, q_{max}]$. 
\end{example}

\begin{example}\label{exam:two phase data}{\bf Two-phase data (Figure \ref{Results of Conograph} (c)).}
This powder diffraction pattern is a two-phase sample with the mass ratio $58 : 42$.
The lattice parameter of the second phase was also enumerated.
\end{example}

\begin{example}\label{exam:bad quality}{\bf Poor quality powder diffraction pattern (Figure \ref{Results of Conograph} (d)).}
This is sample 7 distributed in SDPDRR-2.
According to a recent personal report from Le Bail, crystal structures other than sample 7 have been determined.
Conograph obtained three reasonable solutions.
\end{example}

\begin{figure}
\begin{minipage}{\textwidth}
\begin{minipage}{0.52\textwidth}
(a) Case of more than one solution. \\
\scalebox{0.5}{\includegraphics{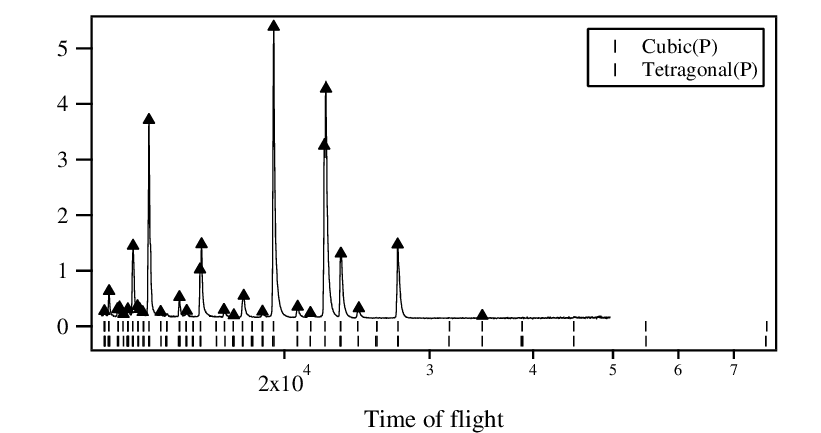}}
\end{minipage}
\begin{minipage}{0.52\textwidth}
(b) The 19 smallest $0 \neq q \in \Lambda_{L^*}$ are lost. \\
\scalebox{0.5}{\includegraphics{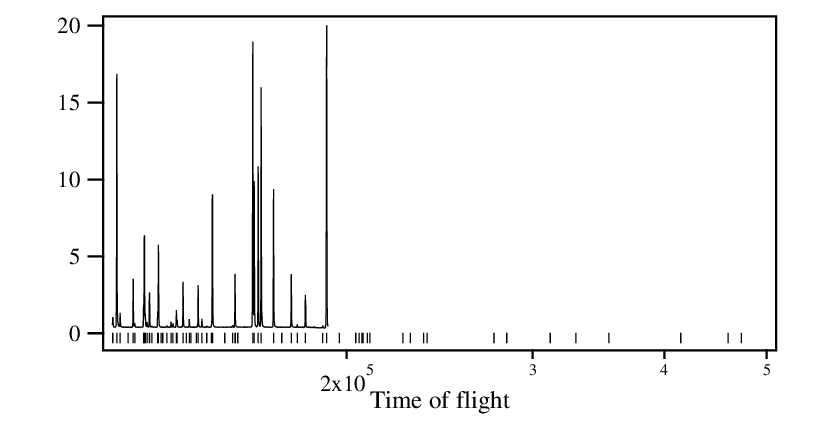}}
\end{minipage}
\begin{minipage}{0.52\textwidth}
(c) Two-phase data. \\
\scalebox{0.5}{\includegraphics{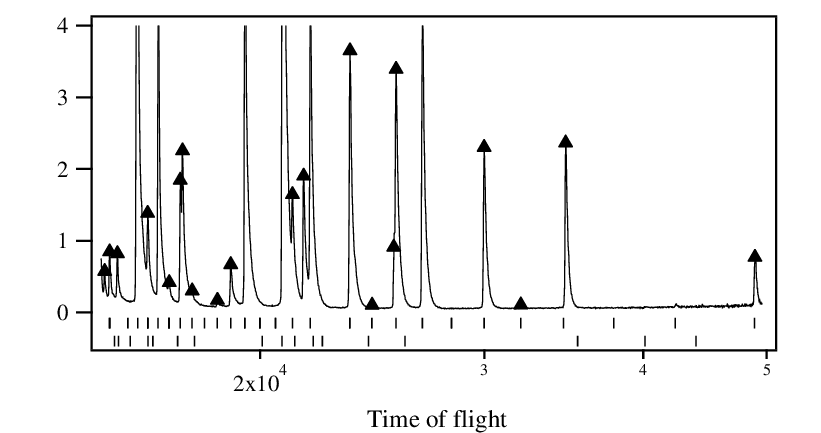}}
\end{minipage}
\begin{minipage}{0.52\textwidth}
(d) Poor quality powder pattern. \\
\scalebox{0.5}{\includegraphics{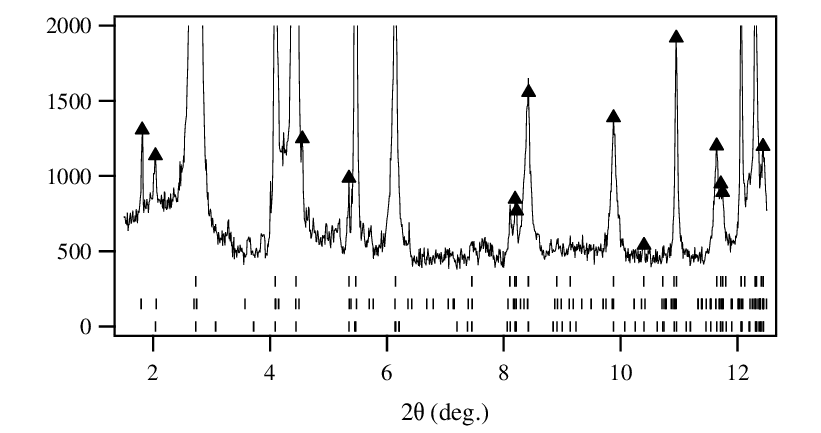}}
\end{minipage}
\footnotetext[1]{
Triangles indicate peak positions detected by a peak-search program (corresponding to $q$-values in $\Lambda^{obs}$). 
Tick marks represent the peak positions corresponding to $q \in \Lambda_{L^*} := \{ \abs{l^*}^2 : l^* \in L^* \}$. 
It is seen that two different lattices have exactly the same $q$-values.
Both parameters are detected by Conograph.
}
\footnotetext[2]{
This example presents the case in which the 19 smallest non-zero $q$-values are not  
in the observed range. 
It is confirmed that Conograph is very robust to such a loss.
The lattice parameters are $a = 8.82$, $b=9.76$, $c=9.78$, $\alpha=\gamma=90$, $\beta=104$,
and a diffraction pattern from a back-scattering bank of neutron sources is used.
}
\footnotetext[3]{
This powder diffraction pattern is a case of a two-phase sample with the mass ratio $58 : 42$.
\begin{eqnarray}
\text{Upper} &:&  a = b = c = 12.0\ (\text{Cubic(I)}, M_{20} = 7.0), \\
\text{Lower} &:&  a = b = 4.8, c = 13.0\ (\text{Rhombohedral}, M_{20} = 1.4).
\end{eqnarray}
The unit cell of the first phase
is much larger than that of the second phase.
As a result, the lattice parameter of the first phase 
gained the best de Wolff figure of merit among all the enumerated solutions, regardless of peaks derived from the second phase.
The lattice parameter of the second phase was also enumerated by Conograph.
This supports our claim in Section \ref{Problems on quality of powder diffraction patterns}
that the enumeration process is robust to missing and false elements of $\Lambda^{obs}$.  
However, the figure of merit of the second phase
was considerably small, owing to the first-phase peaks.
Therefore, it is almost impossible to judge which one is the correct second-phase lattice.
}
\footnotetext[4]{
This is sample 7 distributed in SDPDRR-2, held in 2002.
The tick marks are peak positions of the following solutions output by Conograph:
\begin{eqnarray}
\text{Upper} &:&  a = 3.99, b=11.5, c=17.1, \alpha=77.9, \beta=84.6, \gamma=80.4\ (M_{20} = 6.6), \\
\text{Middle} &:& a = 4.07, b=22.5, c=25.7, \alpha=88.8, \beta=87.4, \gamma=84.8\ (M_{20} = 7.7), \\
\text{Lower} &:&  a = 3.95, b=17.1, c=22.8, \alpha=78.4, \beta=,86.6, \gamma=84.1\ (M_{20} = 12.9).
\end{eqnarray}
The upper lattice parameter is the solution proposed by participants in SDPDRR-2.
This regards the two leftmost peaks to be false.
The others are newly found by Conograph.
The second solution assumes that some peaks are embedded in background noise,
and the third is intermediate.
This result proves that Conograph makes the stage of enumeration much more reliable.
Although the correct solution is provided by a larger $M_{20}$ normally,
this is not always right.
In particular, it becomes very difficult to select the best solution from peak positions alone,
when powder patterns are assumed to have a number of diffraction peaks embedded in background noise or false peaks as intense as diffraction peaks.
}
\end{minipage}
\caption{Results of Conograph.}
\label{Results of Conograph}
\end{figure}

\section{Conclusion}
\label{Conclusion}

Powder auto-indexing is divided into two main stages: enumeration and sort of solutions.
We contributed mainly to the stage of enumeration by providing a quick and strongly reliable algorithm.
For the purpose, Conway's definition of topographs was generalized to lattices of any rank $N$ using Voronoi's second reduction theory,
holding the association of edges and four lengths $\abs{l_1^*}$, $\abs{l_2^*}$, $\abs{l_1^*+l_2^*}$, $\abs{l_1^*-l_2^*}$ ($l_1^*, l_2^* \in L^*$).
By using common properties of systematic absences proved in Theorems \ref{fact:two-dimensional extinction rules}, \ref{thm:distribution rules on CT_{2, S_2}} and \ref{thm:distribution rules on CT_{3, S_3}},
the algorithm is shown to work regardless of the type of systematic absences.
This properties are stated 
as distribution rules for reciprocal lattice vectors corresponding to systematic absences on a topograph.
Such rules have not been known so far, and will be also useful in other problems of crystallography. 
Our enumeration algorithm was implemented in Conograph.
Conograph obtained successful results, even for difficult cases.
These examples proved that the new enumeration method is robust against missing and false elements in 
the set of lattice vector lengths extracted from a powder diffraction pattern.
Topographs were also utilized to speed up our enumeration algorithm.
In practical tests, we found that the improvement reduced the enumeration time to 1/250--1/32.


\paragraph{Acknowledgments}
The author would like to extend her gratitude to Professor T. Oda of the University of Tokyo for his daily encouragements,
to Visible Information Inc. for their cooperation in implementing the Conograph GUI, 
and to Professor E. Hitzer of International Christian University for proposing the impressive name ``Conograph.'' I would also like to thank Dr. K. Fujii and Professors H. Uekusa, T. Ozeki of the Tokyo Institute of Technology, Dr. S. Torii, Dr. J. Zhang, Dr. M. Ping, and Professors M. Yonemura, T. Kamiyama of KEK, and Professors A. Hoshikawa and T. Ishigaki of Ibaraki University
for their valuable comments and for offering test data.
This research was partly supported by a Grant-in-Aid for Young Scientists (B) (No. 22740077) and the Ibaraki Prefecture (J-PARC-23D06).

\appendix

\section{On equivalence between powder diffraction patterns and average theta series}
\label{On equivalence between powder diffraction pattern and average theta series}

For any periodic function $\wp$ with the period lattice $L$ satisfying 
$\int_{ \RealField^N / L } \abs{ \wp(x) } dx < \infty$,
an \textit{average theta series} $\Theta_\wp(z)$ is defined by: 
\begin{eqnarray}\label{eq:definition of a weighted theta series}
	\Theta_\wp(z) := \frac{1}{{\rm vol}(\RealField^N/L)}
		\sum_{l \in L} \int_{ (\RealField^N / L)^2 } \wp(x) \wp(y) e^{\pi \sqrt{-1} z |x-y+l|^2} dx dy, 
 \end{eqnarray}
where ${\rm vol}(\RealField^N/L)$ is the volume of $\RealField^N / L$. ($\Theta_\wp(z)$ is invariant 
if a sublattice $L_2 \subset\neq L$ is regarded as the period of the same $\wp$.
This definition is a generalization of the average theta series defined in 2.3, Chapter2 of \cite{Conway98}.)
The infinite sum converges uniformly and absolutely in any compact subset of $\{ z \in \ComplexField : {\rm Im}(z) > 0 \}$. 

Using the Poisson summation formula, the functional equation for $\Theta_\wp(z)$ is obtained:
\begin{eqnarray}\label{eq:functional equation}
	\Theta_\wp(z)
	= \left( \frac{\sqrt{-1}}{z} \right)^{N/2} \sum_{l^* \in L^*} e^{-\frac{\pi \sqrt{-1}}{z} \abs{l^*}^2} | \hat{\wp}(l^*) |^2,
\end{eqnarray}
where $\hat{\wp}(l^*) := {\rm vol}(\RealField^N/L)^{-1} \int_{\RealField^N / L} \wp(x) e^{-2 \pi \sqrt{-1} x \cdot l^* } dx$.

Assuming that $f_{powder}(q; \wp)$ in (\ref{eq:powder diffraction pattern}) equals $0$ for any $q < 0$,
the Fourier transform of $(2 \sqrt{q})^{-1} f_{powder}(q; \wp)$ is an average theta series.
\begin{eqnarray}
	\int_{\RealField} f_{powder}(q; \wp) e^{2 \pi \sqrt{-1} q z} \frac{d q}{ 2 \sqrt{q} }
	&=& \int_{\RealField} \left( \int_{|x^*|^2 = q} \sum_{l^* \in L^*} |\hat{\wp}(l^*)|^2 \delta(x^* - l^*) dx^* \right) e^{2 \pi \sqrt{-1} q z} \frac{d q}{ 2 \sqrt{q} } \nonumber \\
	&=& \int_{\RealField^N} \sum_{l^* \in L^*} |\hat{\wp}(l^*)|^2 \delta(x^* - l^*) e^{2 \pi \sqrt{-1} \abs{x^*}^2 z} dx^* \nonumber \\
	&=& \sum_{l^* \in L^*} e^{2 \pi \sqrt{-1} \abs{l^*}^2 z} |\hat{\wp}(l^*)|^2 = (-2 \sqrt{-1} z)^{-N/2} \Theta_{\wp} \left( -\frac{1}{2z} \right). \hspace{10mm}
\end{eqnarray}

Therefore, information obtained from a powder diffraction pattern 
is theoretically equivalent to that from an average theta series.

\section{Theorems on the cardinality of solutions}
\label{Summary of known theorems on multiple solutions}

It is well known that the equivalence class of $S_0 \in {\mathcal S}^N_{\succ 0}$ is not uniquely determined,
even if all elements of $\Lambda_{S_0} := \{ \tr{v} S_0 v : 0 \neq v \in \IntegerRing^N \}$ are provided. 
However, for $N \leq 4$, it is possible to obtain a finite set containing all the equivalence classes of $S \in {\mathcal S}^N_{\succ 0}$ with 
$\Lambda_S = \Lambda_{S_0}$ (\CF Appendix \ref{Appendix:lattice determination algorithm from complete set of the lengths of lattice vectors}).

In this section, several known theorems about the cardinality of solutions are summarized for reference.
\begin{enumerate}
	\item {\bf Case $N=1$.} 
	In this case, $\Lambda_\wp$ in (\ref{item:peak-positions}) generates $L^*$ over $\IntegerRing$.
(Otherwise, let $L_2^* \subsetneq L^*$ be the lattice generated by $\Lambda_\wp$.
Then, the reciprocal lattice $L_2$ of $L_2^*$ is the period lattice of $\wp$, 
since $\wp(x) = \sum_{l^* \in L_2^*} \hat{\wp}(l^*) e^{2 \pi \sqrt{-1} x \cdot l^*}$ holds. This is a contradiction.)
Therefore, the determination of $L$ 
is straightforward, even from $\Lambda_\wp$.

	\item {\bf Case $N=2$.}
		It was shown by Delone (and independently by Watson \cite{Watson79}, \cite{Watson80}) that, up to a factor, the following is the only pair of inequivalent positive-definite symmetric matrices that have the same representations over $\IntegerRing$.
		\begin{eqnarray}
			\begin{pmatrix}
				2 & 1 \\
				1 & 2 
			\end{pmatrix},
			\begin{pmatrix}
				2 & 0 \\
				0 & 6 
			\end{pmatrix}.
		\end{eqnarray}

	\item \label{item:Case of N=3} {\bf Case $N=3$.}
		From the case of $N = 2$, an infinite family of inequivalent pairs that have the same representations over $\IntegerRing$
		is obtained.
		\begin{eqnarray}
			\begin{pmatrix}
				2 & 1 & 0 \\
				1 & 2 & 0 \\
				0 & 0 & c 
			\end{pmatrix},
			\begin{pmatrix}
				2 & 0 & 0 \\
				0 & 6 & 0 \\
				0 & 0 & c 
			\end{pmatrix}.
		\end{eqnarray}

		The following is another example (see \cite{Moon2008}\nocite{Moon2008} for the proof). 
		\begin{eqnarray}
			\begin{pmatrix}
				1 & 0 & 0 \\
				0 & 1 & 0 \\
				0 & 0 & 1 
			\end{pmatrix},
			\begin{pmatrix}
				1 & 0 & 0 \\
				0 & 2 & 0 \\
				0 & 0 & 2 
			\end{pmatrix}.
		\end{eqnarray}
		A powder auto-indexing result of Conograph for this case is presented in 
		Example \ref{exam:Non-unique solutions} of Section \ref{Results}. 
		
	\item {\bf Case $N=4$.}
		$S \in {\mathcal S}^N_{\succ 0}$ is called \textit{universal}
		if $\Lambda_S$ equals $\IntegerRing_{> 0}$, the set of all positive integers.
		It was confirmed by Bhargava and Hanke
		that the number of equivalence classes of universal $S \in {\mathcal S}^4_{\succ 0}$
		equals 6436 \cite{Bhargava2005}. 

	\item {\bf Case $N \geq 5$.}
		From the existence of universal quadratic forms in $N = 4$,
		there are infinitely many $S \in {\mathcal S}^5_{\succ 0}$ such that $\Lambda_S = \IntegerRing_{> 0}$.
\end {enumerate}

\section{Lattice determination from a complete set of lattice vector lengths}
\label{Appendix:lattice determination algorithm from complete set of the lengths of lattice vectors}

In this section, for any $N \leq 4$ and a given $\Lambda_{S_0} \subset \RealField_{> 0}$ of some $S_0 \in {\mathcal S}_{\succ 0}^N$,
an algorithm to enumerate all the equivalence classes of $S \in {\mathcal S}^N_{\succ 0}$
satisfying $\Lambda_{S_0} = \Lambda_S$ is introduced.

First, we recall that
$S := (s_{ij})_{1 \leq i, j \leq N} \in {\mathcal S}_{\succ}^N$
is \textit{Minkowski-reduced} if and only if $S$ satisfies
	\begin{eqnarray}\label{eq:condition1}
		s_{ii} = \min \{ \tr{ v_i } S v_i : \{ \mathbf{e}_1, \ldots, \mathbf{e}_{i-1}, v_{i} \} \text{ is a primitive set of } \IntegerRing^N \}. 
	\end{eqnarray}

When $S$ is Minkowski-reduced,
the entries of $S$ satisfy
\begin{eqnarray}\label{eq:condition2}
	0 < s_{11} \leq \cdots \leq s_{NN},\ 2 \abs{ s_{ij} } \leq s_{ii}\ (1 \leq i < j \leq N).
\end{eqnarray}


Table \ref{Recursive procedure} gives a recursive procedure to generate all the candidates for $S$
from $\Lambda^{obs} := \Lambda_{S_0} \cap (0, c)$, when $0 < c \leq \infty$ is large enough.
If the recursive procedure is started with arguments $m=n=I=J=1$, 
all positive-definite symmetric matrices $S := (s_{ij})_{1 \leq i, j \leq N}$ satisfying the followings 
are enumerated in an array $Ans$.
\begin{eqnarray}\label{eq:condition3}
	\begin{cases}
		s_{11} \leq \cdots \leq s_{NN},\ 
		2 \abs{ s_{ij} } \leq s_{ii} \ (1 \leq i < j \leq N),\\ 
		s_{i i+1} \leq 0 \ (1 \leq i < N),\\
		\{ s_{ii} : 1 \leq i \leq N \} \cup
		\{ s_{ii}+s_{jj}+2s_{ij} : 1 \leq i < j \leq N \} \subset \Lambda^{obs}.
	\end{cases}
\end{eqnarray}

\begin{small}
\begin{table}
	\caption{Recursive procedure for lattice determination from a perfect set of representations.}
	\label{Recursive procedure}
	\begin{minipage}{\textwidth}
	\begin{tabular}{lccp{90mm}}
	\hline
	\multicolumn{4}{c}{{\bf void enumerateLattice($N, \Lambda^{obs}, m, n, S, I, J, Ans$)}} \\
	(Input) & $1 \leq N \leq 4$ & : & number of rows and columns of $S_0 \in {\mathcal S}^N_{\succ 0}$, \\
	& $\Lambda^{obs} := \langle q_1, \ldots, q_M \rangle \neq \emptyset$ & : & a sorted sequence of all the elements of $\Lambda_{S_0}$ that belong to the interval $(0, c)$,  \\
	& $m, n$ & : & integers $1 \leq m \leq n \leq N$ indicating the $(m, n)$-entry of $S$, \\ 
	& $S := (s_{ij})$ & : & $N \times N$ symmetric matrix that fulfills \\
	& & & 	
						$0 < s_{11} \leq \cdots \leq s_{n n}$,
						$2 \abs{ s_{ij} } \leq s_{ii}$ ($1 \leq i < j < n$),
						$s_{i i+1} \leq 0$ ($1 \leq i \leq n$).
		\\
	& $I, J$ & : & integers indicating \\
	& & & 	$\begin{cases}
					s_{nn} \in \{ q_i : I \leq i \leq J \} & \text{if } m = n, \\
					s_{nn} = q_I \text{ and } s_{mm} + s_{nn} + 2 s_{mn} \in \{ q_i : I \leq i \leq J \} & \text{otherwise.}
			\end{cases}$ \\
	(Output) & $Ans$ & : & array of positive-definite $N \times N$ symmetric matrices. \\
	\end{tabular}
	\vspace{5mm}
	\begin{tabular}{p{5mm}p{10mm}p{110mm}}
	1: & (start)  & for $l=I$ to $J$ do \\
	2: & & \hspace{5mm}	if $m=n$ then \\
	3: & & \hspace{10mm}	          $s_{n n} = q_{l}$. \\
	4: & & \hspace{5mm}	else \\
	5: & & \hspace{10mm}            $s_{m n} = s_{n m} = \frac{1}{2}(q_{l} - s_{m m} - s_{n n})$. \\
	6: & & \hspace{5mm} end if \\
	7: & & \hspace{5mm}	if $m=1$ then \\
	8: & & \hspace{10mm}		if $\det (s_{ij})_{1 \leq i, j \leq n} > 0$ then \\
	9: & & \hspace{15mm}			if $n \geq N$ then \\
	10: & & \hspace{20mm}					Insert $S$ in $Ans$. \\ 
	11: & & \hspace{15mm}			else  \\
	12: & & \hspace{20mm} 					${m_2}^a := \max \left\{ 1 \leq i \leq M :  
												\begin{matrix}
													q_{1}, \ldots, q_{i-1} \text{ are representations of} \\
													\text{ the submatrix } 
																	(s_{ij})_{1 \leq i, j \leq n} \text{ over } \IntegerRing 
												\end{matrix} \right\}$. \\
	13: & & \hspace{20mm}		if $m=n$ then \\
	14: & & \hspace{25mm}    				Call searchLattice($N, \Lambda^{obs}, n+1, n+1, S, l, m_2, Ans$). \\
	15: & & \hspace{20mm}		else \\
	16: & & \hspace{25mm}    				Call searchLattice($N, \Lambda^{obs}, n+1, n+1, S, I, m_2, Ans$). \\
	17: & & \hspace{20mm} end if \\
	18: & & \hspace{15mm} end if \\
	19: & & \hspace{10mm} end if \\
	20: & & \hspace{5mm}	else  \\
	21: & & \hspace{10mm}		if $m=n$ then \\
	22: & & \hspace{15mm}     		$m_2 := \max \{ 1 \leq i \leq M : q_i \leq s_{n-1 n-1} + s_{n n} \}$.  \\
	23: & & \hspace{15mm}   		Call searchLattice($N, \Lambda^{obs}, m-1, n, S, l, m_2, Ans$). \\
	24: & & \hspace{10mm}		else \\
	25: & & \hspace{15mm}     		$m_2 := \max \{ 1 \leq i \leq M : q_i \leq 2s_{m-1 m-1} + s_{n n} \}$. \\
	26: & & \hspace{15mm}   		Call searchLattice($N, \Lambda^{obs}, m-1, n, S, I, m_2, Ans$). \\
	27: & & \hspace{10mm}		end if \\
	28: & & \hspace{5mm} end if \\
	29: & & end for
	\end{tabular}
	\footnotetext[1]{Proposition \ref{lem:no sublattice} claims that there always exists $m_2 < \infty$,
	even if the sequence $\Lambda^{obs}$ is replaced by $\Lambda_{S_0}$ virtually.}
	\end{minipage}
\end{table}
\end{small}

Consequently, any Minkowski-reduced $S$ satisfying (\ref{eq:condition3}) and $3 s_{NN} \leq q_{M}$ is output in $Ans$.
As a result of Proposition \ref{lem:no sublattice},
the recursive procedure is always completed in a finite number of steps,
even if we set $c = \infty$, i.e., if $\Lambda_{S_0}$ is used instead of $\Lambda^{obs}$.
This indicates
all the equivalence classes of $S \in {\mathcal S}^N$
satisfying $\Lambda_S = \Lambda_{S_0}$ are enumerated by the recursive procedure,
if sufficiently large $c$ is selected.
Consequently,  
the number of equivalence classes of $S \in {\mathcal S}^N_{\succ 0}$ satisfying $\Lambda_S = \Lambda_{S_0}$ is finite
for any $\Lambda_{S_0} \subset \RealField_{> 0}$.

In the remainder of this section, we give a proof of Proposition \ref{lem:no sublattice}.
\begin{proposition}\label{lem:no sublattice}
Suppose that $S \in {\mathcal S}^N_{\succ 0}$ and $S_2 \in {\mathcal S}^{N_2}_{\succ 0}$
and $0 < N_2 < N \leq 4$.
Then, $\Lambda_{S_2} \not\supset \Lambda_S$.
\end{proposition}

Lemma \ref{lem:decomposition of positive definite symmety matrices} is utilized in the proof.
\begin{lemma}\label{lem:decomposition of positive definite symmety matrices}
Any $S \in {\mathcal S}^N_{\succ 0}$ is represented 
as a finite sum $\sum_{k} \lambda_k S_k$ such that every $S_k$ is a positive-definite symmetric matrix with rational entries, and $\lambda_k \in \RealField_{> 0}$ are linearly independent over $\RationalField$.
\end{lemma}
\begin{proof}
For any $S := (s_{ij})_{1 \leq i, j \leq N} \in {\mathcal S}^N$,
let $v_S$ be the vector $\tr{(s_{11}, s_{12}, \ldots, s_{ij}, \ldots, s_{NN})}$ of length $\frac{N(N+1)}{2}$.
Then, ${\mathcal S}^N$ is identified in $\frac{N(N+1)}{2}$-dimensional vector space
by the map $S \mapsto v_S$.
Define a set $P$ by:
\begin{eqnarray}
	P &:=& \{ I \subset \{ s_{ij} : 1 \leq i \leq j \leq N \} : s_{ij} \in I \text{ are linearly independent over $\RationalField$} \}.
\end{eqnarray}
Then, $P$ is not empty.
Let $\{ t_1, \ldots, t_m \}$ be one of the maximal elements of $P$ under inclusive order.
When vectors $\tr{ (t_1, \ldots, t_m) }$ and $\tr{ (1, \ldots, 1) }$ of length $m$ are denoted by $\mathbf{t}$ and $\mathbf{1}_m$ respectively,
there exists an $\frac{N(N+1)}{2} \times m$ rational matrix $C$
such that $v_S = C \mathbf{t}$.
Furthermore, there exists $\epsilon > 0$ such that for any $U := (u_{ij}) \in GL_m(\RealField)$ with entries $\abs{ u_{ij} } < \epsilon$, every column of $C ({\mathbf t} \tr{ \mathbf{1}_m } - U)$ 
is the image of a positive-definite symmetric matrix by the map $S \mapsto v_S$.

Let $1_m$ be the identity matrix of size $m$.
If $\tr{\mathbf{1}_m} U^{-1} {\mathbf t} \neq 1$, we have equations:
\begin{eqnarray}
	({\mathbf t} \tr{\mathbf{1}_m} - U)^{-1}
 &=& U^{-1}  ( ( \tr{\mathbf{1}_m} U^{-1} {\mathbf t} - 1 )^{-1} {\mathbf t} \tr{\mathbf{1}_m} U^{-1} - 1_m ), \hspace{5mm} \\
	({\mathbf t} \tr{\mathbf{1}_m} - U)^{-1} {\mathbf t} &=& ( \tr{\mathbf{1}_m} U^{-1} {\mathbf t} - 1 )^{-1} U^{-1} {\mathbf t}. 
\end{eqnarray}
If all the entries of $U^{-1} \mathbf{t}$ are negative, we have that $\tr{\mathbf{1}_m} U^{-1} {\mathbf t} < 0$,
and every entry of $({\mathbf t} \tr{\mathbf{1}_m} - U)^{-1} {\mathbf t}$ is positive.
Clearly, there exists $U := (u_{ij}) \in GL_m(\RealField)$ such that 
$\abs{ u_{ij} } < \epsilon$,
every entry of $U^{-1} {\mathbf t}$ is negative,
and the matrix ${\mathbf t} \tr{ \mathbf{1}_m } - U$ is rational.
Fix such a $U$.

Let $S_k$ ($1 \leq k \leq m$) be a positive-definite symmetric matrix satisfying $C ({\mathbf t} \tr{ \mathbf{1}_m } - U) = (v_{S_1}, \ldots, v_{S_m})$.
$S$ is then represented as a linear sum of rational $S_k$ with positive coefficients as follows:
	\begin{eqnarray}
		v_S = C {\mathbf t} = (v_{S_1}, \ldots, v_{S_m}) ({\mathbf t} \tr{ \mathbf{1}_m } - U)^{-1} \mathbf{t}.
	\end{eqnarray}
Hence, the statement is proved.
\end{proof}

For any ring $R_2 \subseteq R$ and a symmetric matrix $S$ with entries in $R_2$, 
let $\Lambda_S(R)$ be the set $\{ \tr{v} S v : 0 \neq v \in R^N \}$ consisting of representations of $S$ over $R$.
If $0 \in \Lambda_S(R)$, $S$ is said to be \textit{isotropic} over $R$.
Otherwise, $S$ is \textit{anisotropic} over $R$.

\begin{proof}[Proof of Proposition \ref{lem:no sublattice}]
The statement holds if it is true when $N_2 + 1 = N = 4$.
By Lemma \ref{lem:decomposition of positive definite symmety matrices},
it is sufficient if the statement is proved in the case that $S$, $S_2$ are rational.
Any non-singular quadratic form over $\RationalField_p$ of rank 4 
satisfies $\Lambda_{S}(\RationalField_p) \supset \RationalField_p^\times$ for any $p$.
On the other hand, any anisotropic quadratic form over $\RationalField_p$ of rank 3,
there exists a finite prime $p$
such that 
$\Lambda_{S_2}(\RationalField_p) \not\supset \RationalField_p^\times$,
(\CF Corollary 2 of Theorem 4.1 in Chapter 6, \cite{Cassels78}). 
If $\Lambda_{S_2} \supset \Lambda_S$, 
$\Lambda_{S_2}(\RationalField) \supset \Lambda_{S}(\RationalField)$,
therefore $\Lambda_{S_2}(\RationalField_p) \supset \Lambda_{S}(\RationalField_p)$ is required for any $p$.
This is a contradiction.
\end{proof}

\bibliographystyle{plain}
\bibliography{docultex}

\end{document}